\documentclass[a4paper,10pt]{article}

\usepackage[english]{babel}

\usepackage[utf8]{inputenc}
\usepackage{amsmath,mathrsfs}
\usepackage{indentfirst}
\usepackage{fancyhdr}
\usepackage{amssymb}
\usepackage{amsthm}
\usepackage[dvips]{graphicx}
\usepackage{color}
\usepackage{cancel}
\usepackage{subcaption}
\usepackage[dvipsnames]{xcolor}
\usepackage{eucal}
\usepackage{latexsym}
\usepackage{chngcntr}
\usepackage{faktor}
\usepackage{microtype}
\usepackage{newlfont}
\usepackage{enumitem}
\usepackage{hyperref,geometry,mathtools} 
\usepackage{todonotes}

\usepackage{amsmath, amsfonts, amsthm,amssymb,  bbm}

\usepackage{tikz-cd}
\usepackage{multicol}
\usepackage{cancel}
\usepackage{mathtools}
\usepackage{graphicx}
\usepackage{MnSymbol}

\newcommand{\tth}{\mathtt{h}}
\newcommand{\ttf}{\mathtt{f}_\e}

\newcommand{\te}{\mathtt{e}}
\newcommand{\teWB}{\mathtt{e}_{\scriptscriptstyle{\textsc{WB}}}}
\newcommand{\DeltaBF}{\Delta_{\scriptscriptstyle{\textsc{BF}}}}
\newcommand{\tthWB}{\mathtt{h}_{\scriptscriptstyle{\textsc{WB}}}}

\DeclareMathOperator*{\sgn}{sgn}

\newcommand{\bbh}{\text{\textbf{b}}_{\mathtt{h}}}
\newcommand{\II}{{\mathrm{I\! I}}}

\newcommand{\e}{\epsilon}
\newcommand{\im}{\mathrm{i}\,}

\newcommand{\molt}[2]{\left(#1\,,\,#2\right)}
\newcommand{\lie}[2]{\left[#1\,,\, #2\right]}

\newcommand{\norm}[1]{{\| #1 \|}}

\newcommand\restr[2]{{
  \left.\kern-\nulldelimiterspace 
  #1 
  \vphantom{\big|} 
  \right|_{#2} 
  }}  
  \newcommand{\BVe}[4]{\left({\cal B}_{\mu,\e}\, \tf^{#1}_{{#2}} \ , \tf^{#3}_{#4} \right) }

\allowdisplaybreaks
\theoremstyle{plain}
\newtheorem{lem}{Lemma}
\newtheorem{teo}[lem]{Theorem}
\newtheorem{prop}[lem]{Proposition}

\theoremstyle{definition}

\newtheorem{sia}[lem]{Definition}

\newtheorem{rmk}[lem]{Remark}

\newtheorem*{rem}{Remark}

\renewcommand{\bar}{\overline}

\newcommand{\vet}[2]{\begin{bmatrix}#1 \\ #2 \end{bmatrix}}
\newcommand{\uno}{\mathrm{Id}}

\newcommand{\bR}{\mathbb{R}}
\newcommand{\bT}{\mathbb{T}}
\newcommand{\bZ}{\mathbb{Z}}
\newcommand{\bN}{\mathbb{N}}

\newcommand{\bC}{\mathbb{C}}

\newcommand{\tf}{\mathtt{f}}

\newcommand{\sL}{\mathscr{L}}
\newcommand{\cO}{\mathcal{O}}

\newcommand{\cR}{\mathcal{R}}
\newcommand{\cJ}{\mathcal{J}}
\newcommand{\cB}{\mathcal{B}}
\newcommand{\cL}{\mathcal{L}}
\newcommand{\cV}{\mathcal{V}}
\newcommand{\tB}{\mathtt{B}}
\newcommand{\tJ}{\mathtt{J}}
\newcommand{\tL}{\mathtt{L}}
\newcommand{\de}{\mathrm{d}}
\newcommand{\pa}{\partial}

\newcommand{\cH}{\mathcal{H}}

\newcommand{\cF}{\mathcal{F}}

\newcommand{\cU}{\mathcal{U}}
\newcommand{\cW}{\mathcal{W}}

\newcommand{\off}{\varnothing}

\newcommand{\bro}{\bar\rho}
\newcommand{\Gc}{\mathfrak{c}}

\newcommand{\tc}{{\mathtt{c}}}
\newcommand{\bem}{\text{\large \bf $\flat$}}

\newcommand{\ch}{{\mathtt c}_{\mathtt h}}

\numberwithin{equation}{section}
\counterwithin{lem}{section}

\title{\bf Benjamin-Feir instability of \\
Stokes waves in finite depth
}
\begin{document}

 \author{Massimiliano Berti, Alberto Maspero, Paolo Ventura\footnote{
International School for Advanced Studies (SISSA), Via Bonomea 265, 34136, Trieste, Italy. 
 \textit{Emails: } \texttt{berti@sissa.it},  \texttt{alberto.maspero@sissa.it}, \texttt{paolo.ventura@sissa.it}
 }}

\date{}

\maketitle
\begin{abstract}
Whitham and Benjamin predicted in 1967  that 
small-amplitude periodic  traveling Stokes 
waves
of  the 2d-gravity water waves equations 
are linearly unstable with respect to long-wave perturbations, 
if the depth 
$ {\mathtt h} $ is larger than a critical threshold $\tthWB \approx 1.363 $. 
In   
this paper we completely  describe, for any value of 
$ \mathtt h >0 $, 
 the four eigenvalues close to zero  of the linearized equations 
 at the Stokes wave, 
 as  the Floquet exponent $\mu$ is turned on.
 We prove in particular the existence of a unique  depth   
 $ \tthWB $, 
 which coincides with the one predicted by 
 Whitham and Benjamin,  
such that, for any $ 0 < \mathtt h < \tthWB $, the eigenvalues close to zero
remain purely imaginary 
and,  for any $ \mathtt h > \tthWB $, 
a pair of non-purely imaginary eigenvalues
depicts a closed figure ``8'', parameterized by the Floquet exponent.
As $ {\mathtt h}  \to \tthWB^{\, +} $ this figure ``8'' 
collapses to the origin of the complex plane. 
The  
proof combines a symplectic version of Kato's perturbative theory 
to compute the eigenvalues of a $ 4 \times 4 $  Hamiltonian and reversible  
matrix,   
and KAM inspired 
transformations to block-diagonalize it. 
The four eigenvalues  have all the  same size $\cO(\mu)$
--unlike the  infinitely deep water case in \cite{BMV1}-- and 
 the correct Benjamin-Feir phenomenon 
 appears only after one non-perturbative 
 block-diagonalization step. In addition one has to  
  track, along the whole proof, the 
explicit dependence of the entries of the $ 4 \times 4 $ 
reduced 
 matrix with respect to the depth $\tth$.
\end{abstract}

\tableofcontents

\section{Introduction to  main results}

A classical problem in fluid dynamics, pioneered by the famous work of Stokes \cite{stokes} in 1847,
concerns the spectral stability/instability of periodic traveling waves --called Stokes waves-- 
of the gravity water waves 
equations in any depth. 
 
Benjamin and Feir \cite{BF}, Lighthill   \cite{Li} and Zakharov \cite{Z0,ZK}
discovered in the sixties, through experiments and formal arguments, that Stokes waves in deep water are  unstable, proposing an heuristic mechanism which  leads 
to the disintegration of wave trains.
More precisely, these works predicted unstable eigenvalues of the linearized equations 
 at the Stokes wave, 
 near the origin of the complex plane, corresponding to small Floquet exponents $ \mu $ or, equivalently, to long-wave perturbations. The same phenomenon was later predicted 
 by Whitham \cite{Whitham} and Benjamin 
 \cite{Benjamin}   
 for Stokes waves of wavelength $ 2\pi\kappa $, in finite depth $ \mathtt h $,  provided that $ \kappa \mathtt h > 1.363 $ approximately. 
 This phenomenon is nowadays called 
``Benjamin-Feir" --or modulational-- instability, and it is supported by an enormous  amount of  physical observations and numerical simulations, see e.g. \cite{DO,KDZ}. We refer to
 \cite{ZO} for an historical survey. 

A serious  difficulty for a rigorous 
mathematical proof of the Benjamin-Feir instability 
is  that the  perturbed  eigenvalues bifurcate from  the  eigenvalue zero, which is 
{\it defective}, with multiplicity four. 
The first  rigorous proof of a local branch of unstable eigenvalues close to zero 
for  $ \kappa \mathtt h  $ larger than the Whitham-Benjamin threshold $1.363\ldots$
was obtained   by 
Bridges-Mielke \cite{BrM}  in finite depth (see also the preprint by  Hur-Yang \cite{HY}).
Their method, based on a spatial dynamics and a center manifold reduction,  breaks down in deep water.  
For dealing with this case 
Nguyen-Strauss \cite{NS} have recently developed a new approach, based
on a Lyapunov-Schmidt decomposition. 
The novel spectral approach developed  in Berti-Maspero-Ventura \cite{BMV1} 
allowed to fully describe, in deep water,  the 
splitting of the four eigenvalues close to zero, 
as the Floquet exponent is turned on,  
proving in particular the conjecture that a pair of non-purely imaginary eigenvalues depicts a closed figure ``8'', parameterized by the Floquet exponent.

\smallskip

The goal of this paper is to 
describe the full Benjamin-Feir instability phenomenon at any finite value of the depth $ \tth > 0 $. 
This  analysis has fundamental physical importance since  real-life experiments are performed 
in  water tanks
(for example the original Benjamin and Feir experiments,  in 
Feltham's National Physical Laboratory, had Stokes waves of wavelength 2.2 m and bottom's depth of  7.62 m, see \cite{Benjamin}).  We also remark that the Benjamin-Feir instability mechanism is a 
possible  responsible of  the 
emergence of rogue waves in the ocean, 
 we refer to \cite{JO, OS} and references therein for  a vast physical literature. 
 A first mathematically rigorous treatment of large waves 
 is given in \cite{GS}, via a probabilistic analysis, in the case of NLS. 

Along this paper,
with no loss of generality, we consider $2\pi$-periodic Stokes waves, i.e. with wave number $\kappa=1$. 
In Theorems  \ref{TeoremoneFinale} and \ref{thm:simpler} 
we prove 
the existence of a unique depth  
 $ \tthWB $, in perfect agreement with the Benjamin-Feir critical value 1.363..., 
such that: 
\begin{itemize}
\item{{\it Shallow water case:}} for any $ 0 < \mathtt h < \tthWB $  
the eigenvalues close to zero
remain purely imaginary for Stokes waves
of sufficiently small amplitude, see Figure \ref{figure-eigth}-left; 
\item{\it Sufficiently deep water case:} 
for any $ \tthWB  < \mathtt h < \infty $, 
there exists a pair of non-purely imaginary eigenvalues which 
traces a complete closed figure ``8'' (as shown in  Figure \ref{figure-eigth}-right) parameterized by the Floquet exponent $ \mu $.
By further increasing $ \mu $, 
 the eigenvalues recollide far from the origin on the imaginary axis where then they keep moving.
As $ {\mathtt h}  \to \tthWB^{\, +} $ the set of unstable Floquet exponents shrinks to zero and the Benjamin-Feir unstable eigenvalues 
collapse to the origin,  see Figure \ref{Remax}.   
This figure `8" 
was first numerically discovered  
by Deconink-Oliveras in \cite{DO}.
\end{itemize}
We remark that our approach fully 
describes {\it all}
the eigenvalues close to $ 0 $, 
providing a necessary and sufficient condition for the existence of  unstable 
eigenvalues, i.e.  the positivity of the 
Benjamin-Feir discriminant function $\DeltaBF(\tth;\mu,\e)  $ defined in
\eqref{primosegno}.

\smallskip

 The results of 
Theorems \ref{TeoremoneFinale} and \ref{thm:simpler} are complementary to those of \cite{BMV1}. In following the natural spectral approach developed 
in \cite{BMV1}, we
encounter a major difference in the proof, that we now anticipate. 
In the infinitely deep water  ideal case
it turns out that the  ``reduced'' $4\times 4$ matrix obtained by the 
Kato reduction procedure is a small perturbation of a block-diagonal matrix which
possesses yet the  correct Benjamin-Feir unstable eigenvalues. 
This is not the case  in  finite depth. 
The correct eigenvalues of the  ``reduced'' $4\times 4$ matrix emerge only after one 
non-perturbative step of block diagonalization.
We shall explain in detail  this point after the statement of Theorem  
\ref{TeoremoneFinale}. 
This is related with the fact that, in infinite deep water,  
among the  four eigenvalues close to zero of  the linearized operator at the Stokes wave, 
two are $ \cO(\mu)$, whereas the other two  have much larger size $\cO(\sqrt{\mu}) $,
whereas in finite depth all four eigenvalues  have size $\cO(\mu) $. 
In addition, along  the whole proof, one needs to carefully track the explicit dependence with respect to $\tth$ of the   entries of the reduced $4\times 4$ matrix.

Let us now present rigorously our results.

\subsection*{Benjamin-Feir instability  in finite depth}

We consider the pure gravity water waves equations  for a bidimensional fluid 
occupying a region with finite depth $ \mathtt h $. With no loss of generality
we set 
the gravity $ g = 1 $, see Remark \ref{rem:gravity}. We consider 
a $2\pi$-periodic Stokes wave 
with amplitude $0< \e \ll 1$ and  
speed 
$$ 
c_\e = \ch + \cO(\e^2) \, , \quad \ch := \sqrt{\tanh(\tth)} \, .
$$
The linearized  water waves equations at the Stokes wave
are, in the inertial reference frame moving with speed $c_\e$, a linear time independent system of the form
$ h_t = \mathcal{L}_{\e}  h  $ where $ \mathcal{L}_{\e} := \mathcal{L}_{\e}({\mathtt h}) $
is a linear operator  with $ 2 \pi $-periodic coefficients,  
see \eqref{cLepsilon} (the operator $ \mathcal{L}_{\e}  $ 
in \eqref{cLepsilon} is actually  obtained 
conjugating the linearized water waves equations in the Zakharov  formulation 
at the Stokes wave 
via the  ``good unknown of Alinhac" \eqref{Alin} and the 
Levi-Civita \eqref{LC} invertible transformations).
The operator 
$ \mathcal{L}_{\e} $ possesses the  eigenvalue $ 0 $, 
which is defective, with multiplicity four,  
 due to  symmetries of the water waves equations. 
The problem is to prove that  the linear system 
 $ h_t = \mathcal{L}_{\e}  h  $  
has  solutions  of the form $h(t,x) = \text{Re}\left(e^{\lambda t} e^{\im \mu x} v(x)\right)$
where $v(x)$ is a  $2\pi$-periodic function, $\mu$ in $ \bR$ is the  Floquet exponent
and $\lambda$ has positive real part, thus $h(t,x)$ grows exponentially in time.
By Bloch-Floquet theory, such $\lambda$ is an  eigenvalue 
  of the operator  $ \mathcal{L}_{\mu,\e} 
:= e^{-\im \mu x } \,\mathcal{L}_{\e} \, e^{\im \mu x } $
acting on $2\pi$-periodic functions. 
\smallskip

 The main result of this paper proves, for any finite value of the depth $ \mathtt h $,
  the full splitting of the four 
 eigenvalues  close to 
zero of the operator $ \mathcal{L}_{\mu,\e}  := \mathcal{L}_{\mu,\e} (\mathtt h ) $ 
when  $ \e $ and $ \mu $ are small enough, see Theorem \ref{TeoremoneFinale}. 
We first present Theorem \ref{thm:simpler} which focuses on the figure $``8" $ formed by 
the Benjamin-Feir {\it unstable} eigenvalues.

We first need to introduce the ``Whitham-Benjamin'' function
\begin{equation}\label{funzioneWB}
\teWB   := \teWB  (\tth)  := \frac{1}{\ch}
\Big[ \frac{9\ch^8-10\ch^4+9}{8\ch^6} - 
\frac{1}{\tth- \frac14\te_{12}^2} \Big(1 + \frac{1-\ch^4}{2} + \frac34 \frac{(1-\ch^4)^2}{\ch^2}\tth \Big) \Big]  \, , 
\end{equation}
where $\ch = \sqrt{\tanh(\tth)} $ is the speed of the linear Stokes wave, and 
\begin{equation}
\te_{12}  :=  \te_{12}(\tth)  := \ch+\ch^{-1}(1-\ch^4)\tth  \label{te12} > 0 \, , \quad 
\forall \tth > 0 \, .
\end{equation}
The function $ \teWB (\tth )$ is well defined for any $ \tth > 0 $ because 
the denominator $ \tth- \tfrac14 \te_{12}^2 > 0  $ in 
\eqref{funzioneWB} is positive for any $ \tth > 0 $, see Lemma \ref{defDh}.  
The function  \eqref{funzioneWB}  coincides, up to a non zero factor, 
with the celebrated function obtained by
Whitham \cite{Whitham}, Benjamin \cite{Benjamin} and Bridges-Mielke \cite{BrM}
which determines the ``shallow/sufficiently deep''  threshold regime. 
In particular the Whitham-Benjamin function $\teWB(\tth)$ vanishes at $ \tthWB = 1.363...$, it is negative for $ 
0 <  \tth < \tthWB $, positive for $ 
\tth > \tthWB $ and tends to $1$ as $\tth \to +\infty$, see Figure \ref{graficoe112}.  
We also introduce the positive coefficient 
\begin{equation}
\te_{22}  :=  \te_{22}(\tth)  := \dfrac{(1-\ch^4)(1+3\ch^4) \tth^2+2 \ch^2(\ch^4-1) \tth+\ch^4}{\ch^3} > 0 \, , \quad \forall \tth > 0 \, .
\label{te22}
\end{equation}
We remark that the functions $\te_{12}(\tth) > \tc_\tth $ and $\te_{22}(\tth) > 0 $ are positive for any $ \tth > 0 $, tend to $0$ as $\tth \to 0^+$ and to $1$ as $\tth \to +\infty$, see
Lemma  \ref{nondegcond}.\begin{figure}[h]
\centering
\includegraphics[width=5.5cm]{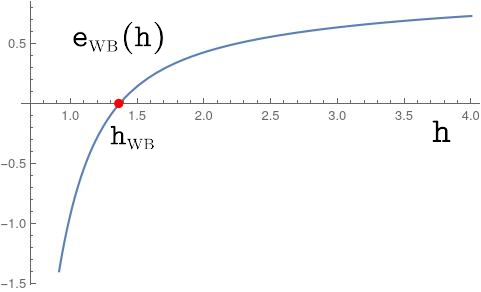}
\caption{Plot of the Whitham-Benjamin function  $ \teWB  (\tth) $.
The red dot shows its unique  root $\tthWB=1.363\dots $.
which is 
the celebrated ``shallow/sufficiently deep'' water threshold predicted  
independently by Whitham (cfr.\cite{Whitham} p.49) and Benjamin (cfr.\cite{Benjamin} p.68), and recovered  in the rigorous proof  of Bridges-Mielke \cite[p. 183]{BrM}.
 }
\label{graficoe112}
\end{figure}

Along  the paper we denote by $r(\e^{m_1} \mu^{n_1}, \ldots, \e^{m_p} \mu^{n_p})$
a real analytic function fulfilling  for some $C >0$ and $\e, \mu$ sufficiently small, the estimate  $| r(\e^{m_1} \mu^{n_1}, \ldots, \e^{m_p} \mu^{n_p}) | \leq 
C \sum_{j=1}^p
 |\e|^{m_j} |\mu|^{n_j}
$, where  the constant $C:=C(\tth)$ is uniform for $\tth$ in any compact set of $(0, + \infty)$.
 \begin{teo}\label{thm:simpler}
 {\bf (Benjamin-Feir unstable eigenvalues)}
For any $ \mathtt h > \tthWB $, there exist  $ \e_1, \mu_0 > 0 $  and 
an analytic  function $\underline \mu: [0,\e_1)\to [0,\mu_0)$, 
 of the form 
 \begin{equation}\label{barmuep}
 \underline  \mu(\e)  =    \te_{\tth}  \e(1+r(\e)) \, , \quad 
 \te_{\tth} := \sqrt{\frac{8\teWB(\tth)}{\te_{22}(\tth)}} \, ,  
 \end{equation}
 such that, 
 for any  $  \e \in [0, \e_1)  $, the 
 operator  $\cL_{\mu,\e}$ has two eigenvalues
  $\lambda^\pm_1 (\mu,\e)$ of the form 
 \begin{equation} \label{eigelemu}
 \begin{cases} 
     \im \frac12 \breve{\mathtt c}_\tth \mu+\im r_2(\mu\e^2,\mu^2\e,\mu^3)  \pm \tfrac18 \mu \sqrt{\te_{22}(\tth)}  (1+r(\e,\mu))  \sqrt{\DeltaBF(\tth;\mu,\e) },
& \forall \mu \in [0, \underline \mu (\e)) \!\!\! \\[1.5mm]
  \im \frac12  \breve{\mathtt c}_\tth \underline \mu (\e)+\im r(\e^3),  
   &
 \mu= \underline \mu  (\e)  \!\!\! \\[1.5mm]
    \im \frac12 \breve{\mathtt c}_\tth \mu+\im r_2(\mu\e^2,\mu^2\e,\mu^3)  \pm \im \tfrac18 \mu \sqrt{\te_{22}(\tth)}  (1+r(\e,\mu))  \sqrt{|\DeltaBF(\tth;\mu,\e)|}, &  \forall  \mu \in ( \underline \mu  (\e), \mu_0) \!\!\!
\end{cases}\!\!\!
\end{equation}
where  
$\breve{\mathtt c}_\tth:=2 \ch- \te_{12}(\tth) >0$ 
and 
$\DeltaBF(\tth;\mu,\e) $ is the ``Benjamin-Feir discriminant" function 
\begin{equation}\label{primosegno}
\DeltaBF(\tth;\mu,\e) := 8\teWB (\tth) \e^2 +r_1(\e^3,\mu\e^2)-\te_{22} (\tth) \mu^2\big(1+r_1''(\e,\mu)\big)  \,  . 
\end{equation}
Note that, for any $0<\e<\e_1$ (depending on $\tth$)  the function 
$ \DeltaBF(\tth;\mu,\e)  > 0 $ is positive,  
respectively $ < 0 $, 
provided $0<\mu < \underline{\mu}(\e)$,  respectively $\mu > \underline{\mu}(\e)$.
\end{teo}
Let us make some comments. 
\\[1mm]
\indent 1.
{\sc Benjamin-Feir unstable eigenvalues.}
For $ \mathtt h >  \tthWB $,
according to \eqref{eigelemu}, 
 for values of the Floquet parameter $ 0<\mu <  \underline \mu (\e) $, the eigenvalues 
$\lambda^\pm_1 (\mu, \epsilon) $ have opposite non-zero real part.
  As $ \mu $ tends to $  \underline  \mu (\e)$, the two eigenvalues $\lambda^\pm_1 (\mu,\epsilon) $ 
 collide on the imaginary axis {\it far} from $ 0 $ (in the upper semiplane $ \text{Im} (\lambda) > 0 $), 
along which they  keep moving  for $ \mu > \underline  \mu (\e) $,
see Figure \ref{figure-eigth}. For $ \mu < 0 $ the operator $ {\mathcal L}_{\mu,\e} $
possesses the symmetric eigenvalues 
$ \overline{\lambda_1^{\pm} (-\mu,\e)} $ in the semiplane $ \text{Im} (\lambda) < 0 $. 
For $ \mu \in [0, \underline \mu(\e)]$  
we obtain the upper part of  
the figure  ``8'', which is 
well approximated by the curves 
\begin{equation}\label{appdr}
\mu \mapsto \Big( \pm\frac{\mu}{8} \sqrt{\te_{22}} \sqrt{8\teWB\e^2 - \te_{22}\mu^2},  \ 
\tfrac12 \breve{\mathtt c}_\tth \mu \Big) \, ,
\end{equation}
in accordance with the  numerical simulations by Deconinck-Oliveras  \cite{DO}.
Note that for $ \mu > 0 $ the imaginary part in \eqref{appdr} is positive because 
$  \breve{\mathtt c}_\tth =  \ch^{-1} 
( \tanh (\tth) - (1- \tanh^2 (\tth)) \tth )> 0 $ 
for any $ \tth > 0 $.    
The higher order ``side-band" corrections of the eigenvalues $ \lambda_1^\pm (\mu,\e ) $ in \eqref{eigelemu}, 
provided by the analytic functions $r, r_1, r_1'', r_2 $,  
are explicitly computable. 
We finally remark that the eigenvalues \eqref{eigelemu} are {\it not analytic}  in $(\mu, \epsilon)$
close to the
value $(\underline{\mu}(\epsilon),\epsilon)$ where  $  \lambda^\pm_1 (\mu, \epsilon) $
 collide  at the top of the 
figure $ ``8" $ far from $0$ (clearly they are continuous).  
 \begin{figure}[h]
 \centering
 \subcaptionbox*{}[.4\textwidth]{\includegraphics[width=4.5cm]{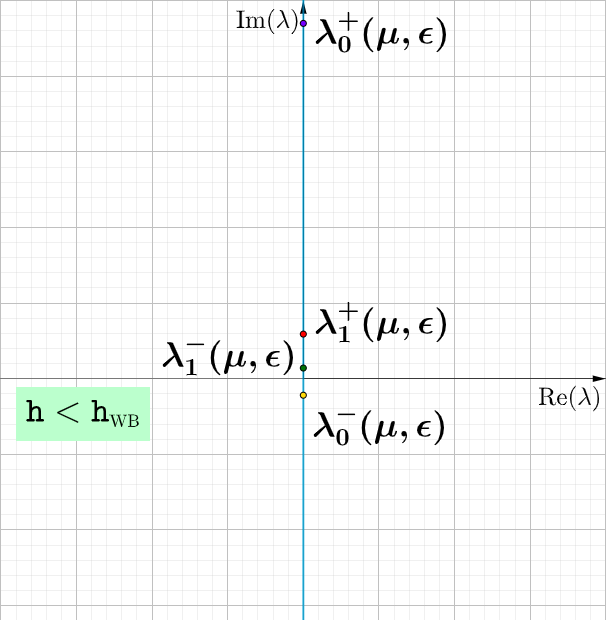}}\hspace{1cm}
\subcaptionbox*{}[.4\textwidth]{
\includegraphics[width=4.5cm]{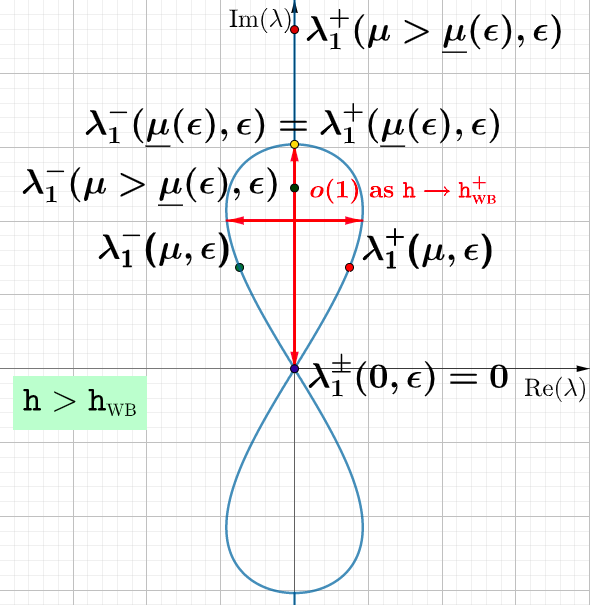} } \vspace{-0.8cm}
\caption{The picture on the left shows, in the ``shallow'' water regime $\tth< \tthWB $, the  eigenvalues $\lambda^\pm_{1} (\mu,\epsilon )$ and  $\lambda^\pm_{0} (\mu,\epsilon )$  which are purely imaginary.
 The picture on the right shows, in the ``sufficiently deep'' water regime $\tth> \tthWB $, the  eigenvalues $\lambda^\pm_1 (\mu,\epsilon )$ in the complex $ \lambda $-plane at fixed $|\e| \ll 1 $ as $\mu$ varies. This figure ``8 '' depends on $\tth$ and shrinks to $0$ as $\tth\to \tthWB^+$, see Figure \ref{Remax}. As $\tth\to+\infty$   the  spectrum resembles the
one in deep water found in \cite{BMV1}. \label{figure-eigth}}
\end{figure}
\\[1mm]
\indent 2.  {\sc Behaviour near the Whitham-Benjamin depth $ \tthWB $. }
As $ {\mathtt h}  \to \tthWB^+ $ the
constant $ \e_1 := \e_1(\tth) > 0 $ in Theorem \ref{thm:simpler}  tends to zero, 
the set of unstable Floquet exponents 
$  (0, \underline  \mu(\e) ) $  with $ \underline  \mu(\e) = \te_{\mathtt h} \e(1+r(\e))  $
given in \eqref{barmuep}
shrinks to zero  and the figure ``8'' of Benjamin-Feir  unstable eigenvalues 
collapse to zero, see Figure \ref{Remax}. 
In particular   
\begin{equation}\label{eqRemax}
\max_{\mu \in [0,\underline{\mu}(\e)]} \text{Re}\,\lambda_1^+(\mu,\e) =  \text{Re}\,\lambda_1^+(\mu_{\max},\e) =\frac12 {\teWB}(\tth) \e^2 + r(\e^3)\ \text{ and } 
\end{equation}
tends to zero as $\tth\to \tthWB^+$, since $0<\e<\e_1(\tth)$ and $\e_1(\tth)\to 0^+$.
 \begin{figure}[h]
\centering
\includegraphics[width=5cm]{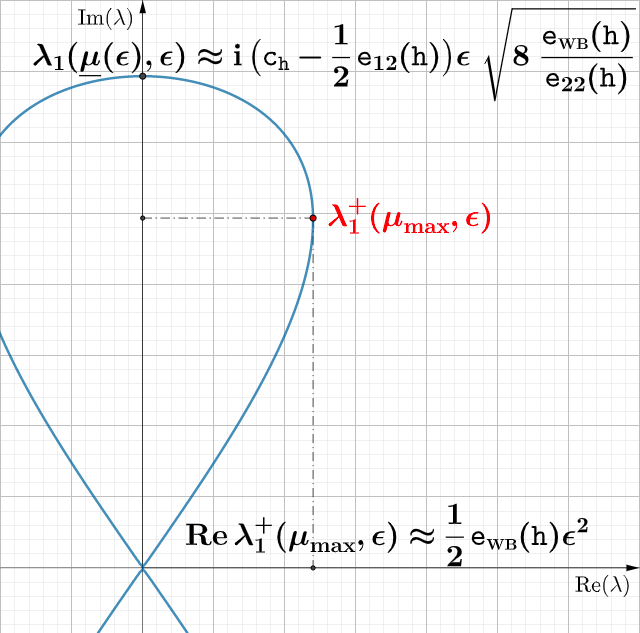}
\caption{ The Benjamin-Feir eigenvalue $\lambda^+_{1}(\mu_{\max},\e) $ in \eqref{eqRemax} with maximal real part, as well as the whole figure ``$8$''
 shrinks to zero 
as $\tth\to \tthWB^+$. }
\label{Remax}
\end{figure}
\\[1mm]
\indent 3. 
{\sc Relation with Bridges-Mielke \cite{BrM}.} Bridges and Mielke  describe 
 the unstable eigenvalues very close to the origin, namely the cross amid the `8". In order to make a precise comparison with our result 
let us spell out the relation of the functions  $\teWB $, $\te_{12}$ and $\te_{22}$ with   the coefficients obtained in 
\cite{BrM}. The Whitham-Benjamin function $\teWB  $ in \eqref{te11} is $\teWB= (\ch\tth)^{-1} \nu(F)$, where $\nu(F)$ is defined   in \cite[formula (6.17)]{BrM} and $F = \ch \tth^{- \frac12} $ is the Froude number, cfr. 
\cite[formula (3.4)]{BrM}. Moreover the term $\te_{12}$ in \eqref{te12} is $\te_{12} = 2 c_g $, where $ c_g = \frac12 \ch \big(1+ F^{-2} \text{sech}^2(\tth)\big) $ 
is the group velocity  defined in 
Bridges-Mielke \cite[formula (3.8)]{BrM}. Finally 
$\te_{22}(\tth) \propto \dot c_g$  where $ \dot c_g $ is the derivative of the group velocity defined in  \cite[formula (6.15)]{BrM},
which for gravity waves is negative in any depth.
\\[1mm]
\indent 4.  {\sc Complete spectrum near $ 0 $.}
In Theorem \ref{thm:simpler} we have described just the two unstable eigenvalues of $\cL_{\mu,\e}$ close to zero for $ \mathtt h >  \tthWB $. 
 There are also two larger 
purely imaginary eigenvalues of order $ \cO(\mu) $, see Theorem \ref{TeoremoneFinale}.  
We remark that our approach describes 
{\it all} the eigenvalues of $ {\mathcal L}_{\mu,\epsilon} $ close to $0 $ (which are $ 4 $).
\\[1mm]
\indent 5. {\sc Shallow water regime.}
In the shallow water regime $ 0 < \mathtt h <  \tthWB $, we prove in Theorem 
\ref{TeoremoneFinale}  that  all 
the four eigenvalues of  $ {\mathcal L}_{\mu,\e} $ close to zero remain purely imaginary for $\e$ sufficiently small. 
The eigenvalue expansions of 
Theorem 
\ref{TeoremoneFinale}  become singular as $ \tth \to 0^+ $. 
\\[1mm]
\indent 6. {\sc Behavior at the Whitham-Benjamin threshold $\tthWB$. } 
The analysis of 
Theorem \ref{thm:simpler} is not conclusive at the critical depth $\tth = \tthWB$. 
The reason is that $ \teWB (\tthWB) = 0 $ 
and the Benjamin-Feir discriminant  function 
 \eqref{primosegno} reduces to
\begin{equation}\label{commentointrigante}
\DeltaBF(\tthWB; \mu, \e) = r(\e^3) + r(\mu\e^2)  -\te_{22}(\tthWB) \mu^2 (1+r_1''(\e,\mu)) \, .
\end{equation}
Thus its quadratic expansion is not sufficient anymore 
to determine the sign of $\DeltaBF(\tthWB; \mu, \e)$. 
Note that \eqref{commentointrigante}  could be positive due to
the cubic term $r(\e^3)=\alpha\e^3+\dots$ for $\e$ and $\mu$ small enough.
The coefficient $\alpha$ could be explicitly computed  taking into account the third order expansion of the Stokes waves. 
\\[1mm]
 \indent 7. {\sc Unstable   Floquet exponents and amplitudes $ (\mu,\e) $. } 
In Theorem \ref{TeoremoneFinale} we actually prove that the expansion 
\eqref{eigelemu} of  the 
eigenvalues of $ \cL_{\mu,\e} $  holds for any value 
of 
$(\mu, \e) $ in a larger rectangle $ [0,\mu_0) \times [0,\e_0 )$, and
there exist Benjamin-Feir unstable eigenvalues 
 if and only if the  analytic  function $\DeltaBF(\tth; \mu, \e)$ in 
 \eqref{primosegno} is positive.
 The zero set of  $\DeltaBF(\tth; \mu, \e)$ is an analytic variety 
which, for $ \tth > \tthWB $,  
 is,  restricted to the rectangle 
$ [0,  \mu_0) \times [0, \e_1)$,   the graph of the analytic  function 
$ \underline  \mu(\e)  =   \te_{\mathtt h}\e(1+r(\e)) $ in \eqref{barmuep}. 
This function  is tangent at $ \e = 0 $ to the straight line 
$ \mu =  \te_{\mathtt h} \e $, and divides 
$ [0,\mu_0) \times [0,\e_1 )$ in the 
region where $\DeltaBF(\tth; \mu,\e) > 0 $ 
 --and thus the  eigenvalues of 
$ {\cL}_{\mu,\e}$ have non-trivial real part--, from the ``stable" one 
where all the eigenvalues of $ {\cL}_{\mu,\e}$ are purely imaginary, see Figure \ref{fig2}.
 In the region  $ [0,\mu_0)\times  [\e_1,\e_0)$ 
 the higher order polynomial approximations of $\DeltaBF(\tth;  \mu,\e )$ 
 (which are computable) will determine the sign of $\DeltaBF(\tth; \mu,\e)  $. 
 \begin{figure}[h]
\centering
\includegraphics[width=6.5cm]{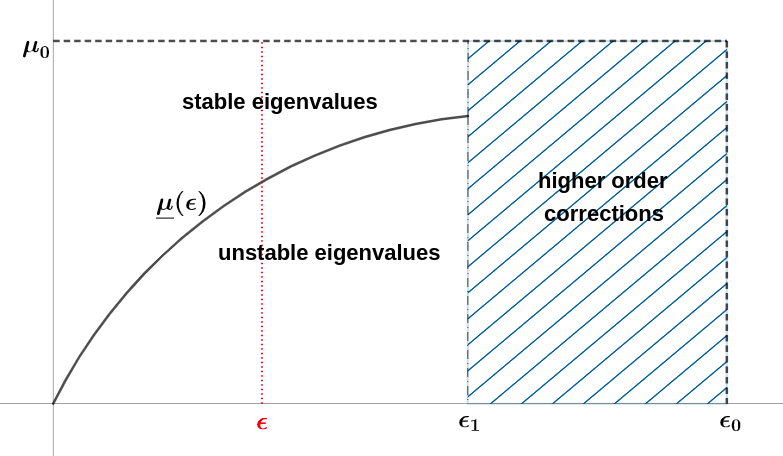}
\caption{The solid curve portrays the graph of the real analytic function $\underline{\mu}(\e)$ in \eqref{barmuep} as $\tth>\tthWB$. For values of $\mu$ below this curve, the two eigenvalues $\lambda^\pm_1(\mu,\e)$ have non zero real part. For $\mu$ above the curve, $\lambda^\pm_1(\mu,\e)$ are purely imaginary. In the region $[\e_1,\e_0)\times [0,\mu_0)$  
the eigenvalues are real/purely imaginary depending on the 
higher order corrections given by Theorem \ref{TeoremoneFinale}, which determine the 
sign of  $\DeltaBF(\tth; \mu,\e)  $.}
\label{fig2}
\end{figure}
\\[1mm] 
\indent 8.  {\sc Deep water limit}. 
Theorems \ref{thm:simpler} and \ref{TeoremoneFinale} do not pass to the limit as $ \tth \to + \infty $ since the remainders in the expansions of the eigenvalues are uniform only on any compact set of $\tth \in (0,+\infty)$. 
From a mathematical point of view, the difference is evident in the asymptotic behavior of $\tanh(\tth\mu) $ (and similar quantities) which, in the idealized  deep water case $\tth=+\infty$, is identically equal to $1$ 
for {\it any 
arbitrarily small} Floquet exponent $ \mu $, whereas $ \tanh(\tth\mu) = O(\mu\tth) $ for any $ \tth $ finite. However additional intermediate scaling regimes $  \tth \mu  \sim 1 $, $  \tth \mu  \ll 1 $, $  \tth \mu  \gg 1 $ are possible. It is well-known (e.g. see \cite{CGK}) that intermediate long-wave 
regimes   of the water-waves equations formally lead to  different physically-relevant limit equations as Boussinesq, KdV,  NLS, Benjamin-Ono, etc...
\smallskip

We shall describe in detail 
the ideas of proof and the differences with the deep water case below
the statement of Theorem \ref{TeoremoneFinale}. 
\\[1mm]
{\it Further literature.} 
Modulational instability has been studied also  for a variety of approximate water waves models, such as KdV, gKdV, NLS and the Whitham equation by, for instance, Whitham \cite{Wh}, Segur, Henderson, Carter and Hammack \cite{SHCH}, Gallay and Haragus \cite{GH}, Haragus and Kapitula \cite{HK}, Bronski and Johnson \cite{BJ}, Johnson \cite{J}, Hur and Johnson \cite{HJ}, Bronski, Hur and Johnson \cite{BHJ},  Hur and Pandey \cite{HP},   Leisman,  Bronski,   Johnson   and
 Marangell \cite{LBJM}.
Also for these approximate models, numerical simulations predict 
a figure ``8'' similar to that in  
Figure \ref{figure-eigth} for the bifurcation of the unstable eigenvalues close to zero. 
We expect the present approach can be adapted to describe the full bifurcation of the  eigenvalues also for these models. 

Finally we mention the  nonlinear modulational instability 
result of Jin, Liao, and Lin  \cite{JLL}  for  
several fluid model equations and the preprint by Chen-Su \cite{ChenSu} for Stokes waves in deep water.
Nonlinear transversal instability results of traveling 
solitary water waves  in finite depth 
decaying at infinity on $ \bR $ 
have been proved in \cite{RT} 
 (in deep water no solitary wave exists \cite{Hur12,IT}).
\\[1mm]
\noindent{\bf Acknowledgments.}
Research supported by PRIN 2020 (2020XB3EFL001) 
``Hamiltonian and dispersive PDEs".

\section{The  complete  Benjamin-Feir spectrum in finite depth}

In this section we present in detail the complete spectral Theorem \ref{TeoremoneFinale}. 
We first introduce  the pure  gravity water waves equations and the Stokes waves solutions.   
\\[1mm]{\bf The water waves equations.}
We consider the Euler equations for a 2-dimensional incompressible, irrotational fluid under the action of  gravity. The fluid fills the
region 
$$
{ \mathcal D}_\eta := \left\{ (x,y)\in \bT\times \bR\;:\; -\tth\leq  y< \eta(t,x)\right\} \, , 
\quad \bT :=\bR/2\pi\bZ \,, 
$$  
with finite depth
and  space periodic boundary conditions. 
The irrotational velocity field is the gradient  
of a harmonic scalar potential $\Phi=\Phi(t,x,y) $  
determined by its trace $ \psi(t,x)=\Phi(t,x,\eta(t,x)) $ at the free surface
$ y = \eta (t, x ) $.
Actually $\Phi$ is the unique solution of the elliptic equation
$    \Delta \Phi = 0 $ in $ {\mathcal D}_\eta $ with Dirichlet datum $
     \Phi(t,x,\eta(t,x)) = \psi(t,x)$ and $    \Phi_y(t,x,y)  =  0 $ at 
     $y =  - \tth $.

The time evolution of the fluid is determined by two boundary conditions at the free surface. 
The first is that the fluid particles  remain, along the evolution, on the free surface   (kinematic
boundary condition), and the second one is that the pressure of the fluid  
is equal, at the free surface, to the constant atmospheric pressure  (dynamic boundary condition). Then, as shown by Zakharov \cite{Zak1} and Craig-Sulem \cite{CS}, 
the time evolution of the fluid is determined by the 
following equations for the unknowns $ (\eta (t,x), \psi (t,x)) $,  
\begin{equation}\label{WWeq}
 \eta_t  = G(\eta)\psi \, , \quad 
  \psi_t  =  
- g \eta - \dfrac{\psi_x^2}{2} + \dfrac{1}{2(1+\eta_x^2)} \big( G(\eta) \psi + \eta_x \psi_x \big)^2 \, , 
\end{equation}
where $g > 0 $ is the gravity constant and $G(\eta):= G(\eta, \tth)$ denotes 
the Dirichlet-Neumann operator $
 [G(\eta)\psi](x) := \Phi_y(x,\eta(x)) -  \Phi_x(x,\eta(x)) \eta _x(x)$. 
 In the sequel, with no loss of generality, we set the gravity constant $ g = 1 $, see Remark
 \ref{rem:gravity}.

The equations \eqref{WWeq} are the Hamiltonian system
\begin{equation}\label{PoissonTensor}
 \pa_t \vet{\eta}{\psi} = \cJ \vet{\nabla_\eta \mathcal{H}}{\nabla_\psi \mathcal{H}}, \quad \quad \cJ:= \begin{bmatrix} 0 & \uno \\ -\uno & 0 \end{bmatrix} ,
 \end{equation}
 where $ \nabla $ denote the $ L^2$-gradient, and the Hamiltonian
$  \mathcal{H}(\eta,\psi) :=  \frac12 \int_{\mathbb{T}} \left( \psi \,G(\eta)\psi +\eta^2 \right) \de x
$
is the sum of the kinetic and potential energy of the fluid. 
In addition of being Hamiltonian, the water waves 
system \eqref{WWeq} possesses other important symmetries.
First of all it is time reversible with respect to the involution 
\begin{equation}\label{revrho}
\rho\vet{\eta(x)}{\psi(x)} := \vet{\eta(-x)}{-\psi(-x)}, \quad \text{i.e. }
\mathcal{H} \circ \rho = \mathcal{H} \, .
\end{equation}
Moreover, the equation \eqref{WWeq} 
 is space invariant, since, being the bottom flat, 
$$ 
\tau_\theta G(\eta)\psi = G( \tau_\theta \eta)[ \tau_\theta \psi] \, ,  \quad
 \forall \theta \in \bR \, , \quad \text{where} \quad 
 \tau_\theta u (x) := u (x + \theta ) \, .  
 $$ 
In addition, the Dirichlet-Neumann operator satisfies
$ G(\eta+m,\tth) = G(\eta,\tth+m) $, for any $ m \in \bR $.
\\[1mm]{\bf Stokes waves.}
The Stokes waves are traveling
solutions of \eqref{WWeq}  of 
the form $\eta(t,x)=\breve \eta(x-ct)$ and $\psi(t,x)=\breve \psi(x-ct)$ for some real  $c$   and  $2\pi$-periodic functions  $(\breve \eta (x), \breve \psi (x)) $.
In a reference frame in translational motion with constant speed $c$,  the water waves equations \eqref{WWeq} become
\begin{equation}\label{travelingWW}
\eta_t  = c\eta_x+G(\eta)\psi \, , \quad 
 \psi_t  = c\psi_x - \eta - \dfrac{\psi_x^2}{2} + \dfrac{1}{2(1+\eta_x^2)} \big( G(\eta) \psi + \eta_x \psi_x \big)^2  
\end{equation}
and the Stokes waves $(\breve \eta, \breve \psi)$ are  equilibrium 
steady solutions 
of \eqref{travelingWW}. 

The bifurcation result of small amplitude of Stokes waves 
   is due to  Struik \cite{Struik} in finite depth, and  Levi-Civita \cite{LC}, 
and Nekrasov \cite{Nek} in infinite depth.
 We denote by $B(r):= \{ x \in \bR \colon \  |x| < r\}$ the real ball with center 0 and radius $r$. 
 \begin{teo}\label{LeviCivita}
{\bf (Stokes waves)} For any $\tth >0$ there exist $\e_*:=\e_*(\tth) >0$ and a unique family  of real analytic 
 solutions $(\eta_\e(x), \psi_\e(x), c_\e)$, parameterized by the amplitude $|\e| \leq \e_*$, of 
\begin{equation}\label{travelingWWstokes}
c \, \eta_x+G(\eta)\psi = 0 \, , \quad 
c \, \psi_x -  \eta - \dfrac{\psi_x^2}{2} + 
\dfrac{1}{2(1+\eta_x^2)} \big( G(\eta) \psi + \eta_x \psi_x \big)^2  = 0 \, , 
\end{equation}
  such that
 $ \eta_\e (x), \psi_\e (x) $ are $2\pi$-periodic;  $\eta_\e (x) $ is even
and $\psi_\e (x) $ is odd, of the form 
 \begin{equation}\label{exp:Sto}
 \begin{aligned}
  & \eta_\e (x) =  \e \cos (x) + \e^2 (\eta_{2}^{[0]} + \eta_{2}^{[2]} \cos (2x)) + 
  \cO(\e^3) , \\ 
  & \psi_\e (x)  = \e \ch^{-1} \sin (x) + \e^2 \psi_{2}^{[2]} \sin (2x)   
  +\cO(\e^3) \, ,  \\
  & c_\e = \ch + \e^2 c_2 +\cO(\e^3) \quad \text{where} \quad \ch = \sqrt{\tanh(\tth)} \, , 
   \end{aligned}
  \end{equation}
and 
  \begin{align}\label{expcoef}
 \eta_{2}^{[0]} :=  \frac{\ch^4-1}{4\ch^2} \, , \qquad 
\eta_{2}^{[2]} := \frac{3-\ch^4}{4\ch^6} \, , \qquad 
\psi_{2}^{[2]} := \frac{3+\ch^8}{8\ch^7} \, , \qquad 
\\ \label{expc2}
c_2:=
\frac{9-10\ch^4+9\ch^8  }{16\ch^7}
+ 
{\frac{(1-\ch^4) }{2\ch}}
\eta_2^{[0]} =  
  \frac{-2 \ch^{12}+13 \ch^8-12 \ch^4+9}{16 \ch^7}\,  .
\end{align}
More precisely for any  $ \sigma \geq  0 $ and $ s > \frac52 $, there exists $ \e_*>0 $ such that
the map $\e \mapsto (\eta_\e, \psi_\e, c_\e)$ is analytic from $B(\e_*) \to H^{\sigma,s}_{\mathtt{ev}} (\bT)\times H^{\sigma,s}_{\mathtt{odd}}(\bT)\times \bR$, where 
$ H^{\sigma,s}_{\mathtt{ev}}(\bT) $, respectively $ H^{\sigma,s}_{\mathtt{odd}}(\bT) $, denote the  space of even, respectively odd, 
 real valued $ 2 \pi $-periodic analytic functions
$ u(x) = \sum_{k \in \mathbb{Z}} u_k e^{\im k x} $
such that $ \| u \|_{\sigma,s}^2 := \sum_{k \in \mathbb{Z}} |u_k|^2 \langle k \rangle^{2s} 
e^{2 \sigma |k|} < + \infty$. 
\end{teo}

The expansions \eqref{exp:Sto}-\eqref{expc2}
 are  derived in the Appendix \ref{sec:App2} for completeness,
 although present in the literature (they coincide with  \cite[section 13, chapter 13]{Wh}  and  \cite[section 2]{Benjamin}). Note that in the shallow water regime $\tth\to 0^+$ the expansions  \eqref{exp:Sto}-\eqref{expc2} become singular. 
 For the analiticity properties of the maps
stated in Theorem \ref{LeviCivita} we refer to \cite{BMV2}. 
 
We also mention that  more general time quasi-periodic traveling Stokes waves 
-- which are nonlinear superpositions of multiple Stokes waves traveling with 
rationally independent speeds  -- have been recently proved 
for \eqref{WWeq} in \cite{BFM2} in finite depth, in \cite{FG} in infinite depth,
and in \cite{BFM} for capillary-gravity water waves in any depth. 
\\[1mm]
{\bf Linearization at the Stokes waves.}\label{goodunknown}
In order to determine the stability/instability of the Stokes  waves given by Theorem \ref{LeviCivita}, 
we linearize  the water waves equations \eqref{travelingWW} with $ c = c_\e $ at  $(\eta_\epsilon(x), \psi_\epsilon(x))$. 
In the sequel we   closely follow \cite{BMV1} pointing out the 
differences 
of the finite depth case. 

By using the shape derivative formula
for the differential $ \de_\eta G(\eta)[\hat \eta ]$ of the Dirichlet-Neumann operator 
one obtains the autonomous real linear system
\begin{equation}\label{linearWW}
 \vet{\hat \eta_t}{\hat \psi_t}
 = \begin{bmatrix} -G(\eta_\e)B-\pa_x \circ (V-c_\e) & G(\eta_\e) \\ -1+B(V-c_\e)\pa_x - B \pa_x \circ (V-c_\e) - BG(\eta_\e)\circ B  & - (V-c_\e)\pa_x + BG(\eta_\e) \end{bmatrix}\vet{\hat \eta}{\hat \psi}
 \end{equation}
 where
\begin{equation}\label{espVB}
V := V(x) := -B (\eta_\e)_x + (\psi_\e)_x \, , \ \  
 B := B(x) := \frac{G(\eta_\e)\psi_\e + (\psi_\e)_x (\eta_\e)_x}{1+(\eta_\e)_x^2}
 = \frac{ (\psi_\e)_x- c_\e}{1+(\eta_\e)_x^2}(\eta_\e)_x
  \, . 
\end{equation}
The functions $(V,B)$ are the horizontal and vertical components of the velocity field
$ (\Phi_x, \Phi_y) $ at the free surface. 
Moreover $\e \mapsto (V,B)$ is  analytic as a map 
$B(\e_0) \to H^{\sigma, s-1}(\bT)\times H^{\sigma,s-1}(\bT)$.

The real system \eqref{linearWW} is Hamiltonian, i.e. of the form $ \cJ \mathcal A $
for a symmetric operator  $ \mathcal A = \mathcal A^\top $,
where $\mathcal A^\top$ is the transposed operator with respect the standard real  scalar product of $L^2(\bT, \bR)\times L^2(\bT, \bR)$.

Moreover, since $ \eta_\e $ is
even in $x$ and $ \psi_\e $ is 
odd in $x$, then  the functions $ (V, B) $ are respectively 
even and odd in $ x $, and the 
linear operator  in \eqref{linearWW} is   
 reversible, i.e. it anti-commutes with the involution $ \rho $ in \eqref{revrho}. 

Under the time-independent  
``good unknown of Alinhac" linear transformation
\begin{equation}\label{Alin}
 \vet{\hat \eta}{\hat \psi} := Z \vet{u}{v} \, , \qquad  Z = \begin{bmatrix}  1 & 0 \\ B & 1\end{bmatrix}, \quad Z^{-1} = \begin{bmatrix}  1 & 0 \\ -B & 1\end{bmatrix},
\end{equation}
the system \eqref{linearWW} assumes the 
simpler form 
\begin{equation}\label{linearWW2}
 \vet{u_t}{v_t} = \widetilde{\mathcal{L}}_\e \vet{u}{v} , 
 \qquad 
  \widetilde{\mathcal{L}}_\e:= \begin{bmatrix} -\pa_x\circ (V-c_\e) & G(\eta_\e) \\ -1 - (V-c_\e) B_x  & - (V-c_\e)\pa_x \end{bmatrix}  \, . 
\end{equation}
Note that, since the transformation $ Z $ is symplectic, i.e.
$ Z^\top \cJ Z = \cJ $, 
and reversibility
preserving, i.e. $ Z \circ \rho = \rho \circ Z $, the linear system \eqref{linearWW2}
is Hamiltonian and reversible as \eqref{linearWW}. 

Next 
we perform a conformal change of variables to flatten 
the water surface. Here the finite depth case induces 
a  modification with respect to the deep water case.  
By \cite[Appendix A]{BBHM}, 
 there exists a diffeomorphism of $\mathbb{T}$,
 $ x\mapsto x+\mathfrak{p}(x)$, with a small $2\pi$-periodic  function $\mathfrak{p}(x)$, 
 and a small constant $\ttf $, such that, by defining the associated composition operator $ (\mathfrak{P}u)(x) := u(x+\mathfrak{p}(x))$, the Dirichlet-Neumann operator writes as \cite[Lemma A.5]{BBHM}
\begin{equation}\label{Gneta}
 G(\eta_\e) = \pa_x \circ \mathfrak{P}^{-1} \circ {\mathcal H} \circ
 \tanh\big((\tth+\ttf)|D| \big)
 \circ \mathfrak{P} \, , 
\end{equation}
where $ {\mathcal H} $ is the Hilbert transform, i.e. the  Fourier multiplier operator
$$
 \mathcal H(e^{\im j x}):= - \im \textup{sign}(j) e^{\im j x} \, , 
 \quad  \forall j \in \bZ \setminus \{0\} \, , 
 \quad \mathcal H(1) := 0 \, . 
$$
The function $\mathfrak p(x)$ and the constant $\ttf $ are  determined as a fixed point  of  
(see  \cite[formula (A.15)]{BBHM})
\begin{equation} \label{def:ttf}
\mathfrak{p}  =  \frac{\cH}{\tanh \big((\tth + \ttf)|D| \big)}[\eta_\e ( x + \mathfrak{p}(x))] \, , 
 \qquad 
 \ttf:= \frac{1}{2\pi} \int_\bT \eta_\e (x +  \mathfrak{p}(x)) \de x \, . 
  \end{equation}
 By the analyticity of the map $\e \to \eta_\e \in H^{\sigma,s}$, $\sigma >0$, $s > 1/2$,  
the 
analytic implicit function theorem
implies the existence of a solution   
$\e \mapsto \mathfrak{p}(x):=\mathfrak{p}_\e(x) $, $ \e \mapsto \ttf $,   analytic
as a map 
$B(\e_0) \to H^{s}(\bT) \times \bR $.
Moreover, since $\eta_\e$ is even,  the function $\mathfrak p(x)$ is odd.
In Appendix \ref{sec:App2} we prove the expansion
\begin{equation}
 \label{expfe}
  \mathfrak p(x)  =\e \ch^{-2} 
  \sin(x)+\e^2\frac{(1+\ch^4)(3+\ch^4)}{8\ch^8}\sin(2x)+\cO(\e^3) \, ,  \quad
   \ttf =
   \e^2\frac{\ch^4-3}{4\ch^2} +
   \cO(\e^3) \, . 
 \end{equation}
Under the symplectic  and reversibility-preserving map 
\begin{equation}\label{LC}
 \mathcal{P} := \begin{bmatrix}(1+\mathfrak{p}_x)\mathfrak{P} & 0 \\ 0 & \mathfrak{P} \end{bmatrix} \, , 
\end{equation}
the system \eqref{linearWW2} transforms, by \eqref{Gneta},  into
the linear system $ h_t = \cL_\e h $
 where  $ \cL_\e $ is the Hamiltonian and reversible real operator
\begin{equation}
\begin{aligned}\label{cLepsilon}
\cL_\e  := \mathcal{P} \, \widetilde{\mathcal L}_\e \, \mathcal{P}^{-1} 
& =  
\begin{bmatrix} \pa_x \circ (\ch+p_\e(x)) &  |D|\tanh((\tth+\mathtt{f}_\e) |D|) \\ - (1+a_\e(x)) &   (\ch+p_\e(x))\pa_x \end{bmatrix} \\
& = \cJ \begin{bmatrix}   1+a_\e(x) &   -(\ch+p_\e(x)) \pa_x \\ 
\pa_x \circ (\ch+p_\e(x)) &  |D|\tanh((\tth+\mathtt{f}_\e) |D|)  \end{bmatrix} 
\end{aligned}
\end{equation}
where 
\begin{equation}\label{def:pa}
\ch+p_\e(x) :=  \displaystyle{\frac{ c_\e-V(x+\mathfrak{p}(x))}{ 1+\mathfrak{p}_x(x)}} \, , \quad 1+a_\e(x):=   \displaystyle{\frac{1+ (V(x + \mathfrak{p}(x)) - c_\e)
 B_x(x + \mathfrak{p}(x))  }{1+\mathfrak{p}_x(x)}} \, .
\end{equation}
By the analiticity results of the functions $ V, B, \mathfrak{p}(x) $ given above,  
the functions   $p_\e$ and $a_\e$   are analytic in $\e$ as maps $B(\e_0)\to H^{s} (\mathbb T)$.
In the Appendix \ref{sec:App2} we prove the following expansions. 
\begin{lem}\label{lem:pa.exp}
The analytic functions $p_\e (x) $ and $a_\e (x) $  in \eqref{def:pa} 
are even in $ x $, and
\begin{equation}\label{SN1}
p_\e (x)  
= \e p_1 (x) + \e^2 p_2 (x)  + \cO(\e^3) \, , \qquad   
a_\e (x)  
= \e a_1(x) +\e^2 a_2 (x) + \cO(\e^3) \, , 
\end{equation}
where 
\begin{align}\label{pino1fd}
 p_1(x) &= p_1^{[1]}\cos (x)\, , \qquad \quad \quad p_1^{[1]} :=  - 2 \ch^{-1}\, ,\\
p_2(x) & \label{pino2fd} = 
 p_2^{[0]}+p_2^{[2]}\cos(2x)\, , \quad  
 p_2^{[0]} := 
 \frac{9+12 \ch^4+ 5\ch^8-2 \ch^{12}}{16 \ch^7}\, , \quad
  p_2^{[2]}:= - \frac{3+\ch^4}{2\ch^7}\, , 
 \end{align}
 and 
 \begin{align}
a_1(x) \label{aino1fd} &= a_1^{[1]} \cos (x)\, , \qquad \qquad 
a_1^{[1]}:= - ( \ch^2 + \ch^{-2})\, , \\
a_2(x)& \label{aino2fd} = a_2^{[0]}+a_2^{[2]}\cos(2x)\, ,\quad   a_2^{[0]}:=\frac32 + \frac1{2\ch^4}\, , 
\quad a_2^{[2]} := 
 \frac{-14\ch^4+9\ch^8-3}{4\ch^8}\, .
\end{align}  
\end{lem}
\noindent
{\bf Bloch-Floquet expansion.}\label{Bloch-Floquet}
Since the operator $\cL_\e$ in \eqref{cLepsilon} has $2\pi$-periodic coefficients, Bloch-Floquet theory guarantees  that 
$$
\sigma_{L^2(\bR)} (\cL_\e ) = \bigcup_{\mu\in [- \frac12, \frac12)} \sigma_{L^2(\bT)}  (\cL_{\mu, \e}) \qquad \text{where} \quad 
\qquad \cL_{\mu,\e}:= e^{- \im \mu x} \, \cL_\e \, e^{\im \mu x}  \, .
$$
 The  domain $ [- \frac12, \frac12) $ is called, in solid state physics, 
 the ``first zone of Brillouin". 
In particular, if $\lambda$ is an eigenvalue of $\cL_{\mu,\e}$ on $L^2(\bT, \bC^2)$
with eigenvector $v(x)$, then  $h (t,x) = e^{\lambda t} e^{\im \mu x} v(x)$ solves $h_t = \cL_{\e} h$. We remark that:
\\[1mm]
\indent 1.
If $A = \mathrm{Op}(a) $  
is a pseudo-differential  operator with   
 symbol $ a(x, \xi ) $, which is $2\pi$ periodic in the $x$-variable, 
then 
$  A_\mu := e^{- \im \mu x}A  e^{ \im \mu x}  =  \mathrm{Op} (a(x, \xi + \mu )) $.
\\[1mm]
\indent 2.
If $ A$ is a real operator then 
$ \overline{ A_\mu} = A_{- \mu } $. 
As a consequence the spectrum  
$ \sigma (A_{-\mu}) = \overline{  \sigma (A_{\mu}) } $ and 
we can study $  \sigma (A_{\mu}) $ just for $ \mu > 0 $.
Furthermore $\sigma(A_{\mu})$ is a 1-periodic set with respect to $\mu$, so   one can restrict  to  $\mu \in [0, \frac12)$.  

By the previous remarks the Floquet operator  associated with the real operator $\cL_\e$ in \eqref{cLepsilon} is  the complex  \emph{Hamiltonian} and \emph{reversible} operator
\begin{align}\label{WW}
 \cL_{\mu,\e} :&= \begin{bmatrix} (\pa_x+\im\mu)\circ (\ch+p_\e(x)) & |D+\mu| \tanh\big((\tth + \ttf) |D+\mu| \big) \\ -(1+a_\e(x)) & (\ch+p_\e(x))(\pa_x+\im \mu) \end{bmatrix} \\ 
 &= \underbrace{\begin{bmatrix} 0 & \uno\\ -\uno & 0 \end{bmatrix}}_{\displaystyle{=\cJ}} \underbrace{\begin{bmatrix} 1+a_\e(x) & -(\ch+p_\e(x))(\pa_x+\im \mu) \\ (\pa_x+\im\mu)\circ (\ch+p_\e(x)) & |D+\mu| \tanh\big((\tth + \ttf) |D+\mu| \big) \end{bmatrix}}_{\displaystyle{=:\cB_{\mu,\e}}} \, . \notag 
\end{align}
We regard $  \cL_{\mu,\e} $ as an operator with 
domain $H^1(\bT):= H^1(\mathbb{T},\bC^2)$ 
and range $L^2(\bT):=L^2(\mathbb{T},\bC^2)$, equipped with  
the complex scalar product 
\begin{equation}\label{scalar}
(f,g) := \frac{1}{2\pi} \int_{0}^{2\pi} \left( f_1 \bar{g_1} + f_2 \bar{g_2} \right) \, \text{d} x  \, , 
\quad
\forall f= \vet{f_1}{f_2}, \ \  g= \vet{g_1}{g_2} \in  L^2(\bT, \bC^2) \, .
\end{equation} We also  denote $ \| f \|^2 = (f,f) $.

The complex operator $\cL_{\mu,\e}$ in \eqref{WW}  is \emph{Hamiltonian} 
and \emph{Reversible},  according to the  following definition.

\begin{sia} {\bf (Complex Hamiltonian/Reversible operator)} \label{def:compl.ham}
 A complex operator $\cL : H^1(\bT,\bC^2) \to L^2(\bT,\bC^2) $ 
is 
\\[1mm] 
($i$) \emph{Hamiltonian}, if 
$\cL = \cJ \cB $ where $ \cB  $ is a self-adjoint operator, namely 
$ \cB = \cB^*   $, 
where  $\cB^*$ (with domain $H^1(\bT)$)  is the adjoint with respect to the complex scalar product \eqref{scalar} of $L^2(\mathbb{T})$.
\\[1mm] 
($ii$) \emph{Reversible},  if
\begin{equation}\label{Reversible}
 \cL \circ \bro =- \bro \circ \cL \, ,
\end{equation}
where $\bro$ is the complex involution (cfr. \eqref{revrho})
\begin{equation}\label{reversibilityappears}
 \bro \vet{\eta(x)}{\psi(x)} := \vet{\bar\eta(-x)}{-\bar\psi(-x)} \, .
\end{equation}
\end{sia}
The property \eqref{Reversible} for $ \cL_{\mu,\epsilon} $ 
follows  because 
$ \cL_\e $ is a real  operator 
which is reversible with respect to the involution $ \rho $ in \eqref{revrho}.
Equivalently, since $\cJ \circ \bro = -\bro \circ \cJ$, the self-adjoint operator $\cB_{\mu,\e}$ is \emph{reversibility-preserving},  i.e. 
\begin{equation}
\label{B.rev.pres}
\cB_{\mu,\e} \circ \bro = \bro \circ \cB_{\mu,\e} \, .
\end{equation}
In addition $(\mu, \e) \to \cL_{\mu,\e} \in \cL(H^1(\bT), L^2(\bT))$ is analytic, 
since  the functions $\e \mapsto a_\e$, $p_\e$  defined in \eqref{SN1} are analytic as maps $B(\e_0) \to H^1(\bT)$ 
and  ${\mathcal L}_{\mu,\e}$ is analytic with respect to $\mu$, since,   
for  any $ \mu \in [-\frac12, \frac12) $,  
\begin{equation}\label{D+mu an}
|D+\mu| \tanh\big((\tth + \ttf) |D+\mu| \big) = (D + \mu)  \tanh\big((\tth + \ttf) (D+\mu) \big) \, . 
\end{equation} 
We also note that  
(see \cite[Section 5.1]{NS})
\begin{equation}\label{|D+mu|}
|D+\mu| =  |D| + \mu (\sgn(D)+\Pi_0) \, , \quad \forall \mu > 0 \, , 
\end{equation}
where $\sgn(D)$ is the Fourier multiplier operator, 
acting on $2\pi$-periodic functions,  
 with symbol 
\begin{equation}\label{def:segno}
 \sgn(k) := 1\  \forall k > 0 \, , \quad \sgn(0):=0 \, ,\quad \sgn(k) := -1 \ \forall k < 0 \, , 
\end{equation}
and $\Pi_0$ is the projector operator on the zero mode, 
$\Pi_0f(x) := \frac{1}{2\pi} \int_\bT f(x)\de x. $

\begin{rmk}\label{rem:gravity}
If $ (\eta(x), \psi(x) , c ) $ solve the traveling wave 
equations \eqref{travelingWWstokes} then the rescaled functions 
$ (\widetilde \eta (x), \widetilde \psi (x), \widetilde c ) :=
(\eta (x), \sqrt{g}  \psi (x), \sqrt{g} c)   $
solve the same equations with gravity constant $ g $ instead of $ 1$.
The eigenvalues of the corresponding linearized operators
\eqref{linearWW} and 
 \eqref{WW} for a general gravity  $g$ are those of the  $g = 1$ case multiplied by  $\sqrt{g}$.
\end{rmk}

Our aim is to prove the existence of eigenvalues of $  \cL_{\mu,\e}  $ in  \eqref{WW}
with non zero real part. 
We remark that the Hamiltonian structure of $\cL_{\mu,\e}$ implies that eigenvalues with non zero real part may arise only from multiple
  eigenvalues of $\cL_{\mu,0}$ (``Krein criterion"), 
  because   if $\lambda$ is an eigenvalue of $\cL_{\mu,\e}$ then  also $-\bar \lambda$ is, and the total algebraic multiplicity of the eigenvalues is conserved under small perturbation. 
We now  describe the spectrum of $\cL_{\mu,0}$.
\\[1mm]{\bf The spectrum of $\cL_{\mu,0}$.}\label{initialspectrum}
The spectrum of the Fourier multiplier matrix operator 
\begin{equation}\label{cLmu}
 \cL_{\mu,0} = \begin{bmatrix} \ch( \pa_x+\im\mu ) & |D+\mu| \,  \tanh\big(\tth  |D+\mu| \big)
  \\ -1 & \ch(\pa_x+\im \mu) \end{bmatrix}
\end{equation}
consists of the purely imaginary eigenvalues $\{\lambda_k^\pm(\mu)\;,\; k\in \bZ \} $, where
\begin{equation} \label{omeghino}
 \lambda_k^\pm(\mu):=
\im \big( \ch( \pm k+\mu) \mp  \sqrt{|k \pm \mu|\tanh(\tth|k \pm \mu|)} \big) \, . 
\end{equation}
For $\mu=0$ the {real} operator $\cL_{0,0}$ possesses the  eigenvalue $0$ with algebraic multiplicity $4$, 
$$ 
\lambda_0^+(0) = \lambda_0^-(0) = \lambda_1^+(0) = \lambda_{1}^-(0)=0 \, ,
$$
 and geometric multiplicity $3$. A real  basis of the 
 Kernel of $\cL_{0,0}$ is
\begin{align}\label{basestart}
 f_1^+ :=  \vet{ \ch^{1/2} \cos(x)}{\ch^{-1/2} \sin(x)}, 
 \quad 
 f_1^{-} :=  \vet{- \ch^{1/2} \sin(x)}{\ch^{-1/2}\cos(x)},\qquad f_0^-:=\vet{0}{1} \, ,
\end{align}
 together with the generalized eigenvector 
\begin{align}\label{basestartadd} 
f_0^+:=\vet{1}{0} , \qquad \cL_{0,0}f_0^+ =-f_0^- \, . 
\end{align} 
Furthermore $0$ is an isolated eigenvalue for $\cL_{0,0}$, namely the 
 spectrum   $\sigma\left(\cL_{0,0}\right)  $ decomposes in two separated parts
\begin{equation}
\label{spettrodiviso0}
\sigma\left(\cL_{0,0}\right) = \sigma'\left(\cL_{0,0}\right) \cup \sigma''\left(\cL_{0,0}\right)
\quad \text{where} \quad  \sigma'(\cL_{0,0}):=\{0\} 
\end{equation}
and 
$$
 \sigma''(\cL_{0,0}):= \big\{ \lambda_k^\sigma(0),\ 
 k = 0,1 \, , \sigma = \pm   \big\} 
 \, .
$$
We shall also use that, as proved in Theorem 4.1 in \cite{NS}, 
the operator $ {\mathcal L}_{0,\e} $ 
possesses, for any sufficiently small $\e \neq 0$, 
 the eigenvalue $ 0 $
with a four 
dimensional generalized  Kernel,  spanned by  $ \e $-dependent 
vectors $ U_1, \tilde U_2, U_3, U_4 $ satisfying, for some real constant $ \alpha_\e, \beta_\e $, 
 \begin{equation}\label{genespace} 
 {\mathcal L}_{0,\e} U_1 = 0 \, , \  \ 
 {\mathcal L}_{0,\e} \tilde U_2 = 0 \, ,  \  \  {\mathcal L}_{0,\e}  U_3 =  
 \alpha_\e \,  \tilde U_2 \, , \ 
 \  {\mathcal L}_{0,\e}  U_4 =  - U_1- \beta_\e \tilde U_2 \, , \quad U_1 :=  \vet{0}{1}  \, . 
 \end{equation}
 By Kato's perturbation theory (see  Lemma \ref{lem:Kato1} below)
for any $\mu, \e \neq 0$  sufficiently small, the perturbed spectrum
$\sigma\left(\cL_{\mu,\e}\right) $ admits a disjoint decomposition as 
\begin{equation}\label{SSE}
\sigma\left(\cL_{\mu,\e}\right) = \sigma'\left(\cL_{\mu,\e}\right) \cup \sigma''\left(\cL_{\mu,\e}\right) \, ,
\end{equation}
where $ \sigma'\left(\cL_{\mu,\e}\right)$  consists of 4 eigenvalues close to 0. 
We denote by $\cV_{\mu, \e}$   the spectral subspace associated with  $\sigma'\left(\cL_{\mu,\e}\right) $, which   has  dimension 4 and it is  invariant by $\cL_{\mu, \e}$.
Our goal is to prove that,  for $ \e $ small, for values of the Floquet
exponent $ \mu $ in an interval of order $ \e $,  the $4\times 4$ matrix 
which represents
 the operator  $ \cL_{\mu,\e} : \mathcal{V}_{\mu,\e} \to  \mathcal{V}_{\mu,\e} $
 possesses a pair of eigenvalues close to zero 
with opposite non zero real parts.

Before stating our main result, let us introduce a notation we shall use through all the paper:

\begin{itemize}\label{notation}
\item[$\bullet$] {\bf  Notation:}
we denote by  $\cO(\mu^{m_1}\e^{n_1},\dots,\mu^{m_p}\e^{n_p})$, $ m_j, n_j \in \bN  $ (for us $\bN:=\{1,2,\dots\} $),  
analytic functions of $(\mu,\e)$ with values in a Banach space $X$ which satisfy, for some $ C > 0 $ uniform for  $\tth$ in any compact set of $(0, + \infty)$, the bound
 $\|\cO(\mu^{m_j}\e^{n_j})\|_X \leq C \sum_{j = 1}^p |\mu|^{m_j}|\e|^{n_j}$ 
 for small values of $(\mu, \e)$. 
Similarly we denote $r_k (\mu^{m_1}\e^{n_1},\dots,\mu^{m_p}\e^{n_p}) $
scalar  functions  $\cO(\mu^{m_1}\e^{n_1},\dots,\mu^{m_p}\e^{n_p})$ which are  also {\it real} analytic.
\end{itemize}

Our complete  spectral result is the following: 

\begin{teo}\label{TeoremoneFinale}
{\bf (Complete Benjamin-Feir spectrum)} 
There exist $ \e_0, \mu_0>0 $, uniformly for the 
depth $ \tth  $ in any compact set of $ (0,+\infty )$, 
 such that, 
for any  $ 0\, <\, \mu < \mu_0 $ and $ 0\leq \e < \e_0 $, 
 the operator $ \cL_{\mu,\e} : \mathcal{V}_{\mu,\e} \to  \mathcal{V}_{\mu,\e} $ 
 can be represented by a $4\times 4$ matrix of the form 
 \begin{equation} \label{matricefinae}
  \begin{pmatrix} \mathtt{U} & \vline & 0 \\ \hline 0 & \vline & \mathtt{S} \end{pmatrix},
 \end{equation}
 where $ \mathtt{U} $ and $ \mathtt{S} $ are  $ 2 \times 2 $
matrices, with identical diagonal entries each, of the form 
 \begin{align}
\notag
& \mathtt{U} = {\begin{pmatrix} 
\im  \big((\ch- \tfrac12\te_{12})\mu+ r_2(\mu\e^2,\mu^2\e,\mu^3) \big) & -\te_{22}\frac{\mu}{8}(1+r_5(\e,\mu)) \\
-\mu \e^2 \teWB +  r_1'( \mu\e^3, \mu^2\e^2 ) +\te_{22}\frac{\mu^3}{8} (1+r_1''(\e,\mu)) & \im  \big( (\ch-\tfrac12\te_{12})\mu+ r_2(\mu\e^2,\mu^2\e,\mu^3) \big)  
 \end{pmatrix}}\, ,   \\
 & \label{S} \mathtt{S} = 
 \begin{pmatrix} \im\ch\mu+ \im { r_9(\mu\e^2, \mu^2\e)}  & \tanh(\tth \mu)+ {r_{10}(\mu\e)} \\ 
-\mu + {r_8(\mu\e^2, \mu^3 \e)} &    \im\ch\mu+\im {r_9(\mu\e^2,\mu^2\e) }
 \end{pmatrix}\, , 
\end{align}
where 
$\teWB $, $\te_{12}, \te_{22}$ are defined in \eqref{funzioneWB}, \eqref{te12}, \eqref{te22}.
The eigenvalues  of  $ \mathtt{U} $ have the form  
\begin{equation}\label{l1piumeno} 
\begin{aligned}
\lambda_1^\pm(\mu,\e) 
& =   \im \frac12\breve{\mathtt c}_\tth \mu+\im r_2(\mu\e^2,\mu^2\e,\mu^3)  \pm \tfrac18 \mu \sqrt{\te_{22}(\tth)  (1+r_5(\e,\mu)) } \sqrt{\DeltaBF(\tth; \mu, \e)}  \, , 
\end{aligned}
\end{equation}
where  $\breve{\mathtt c}_\tth:=2 \ch- \te_{12}(\tth)$ and $\DeltaBF(\tth; \mu, \e)$ is the Benjamin-Feir discriminant function \eqref{primosegno} (with $r_1(\e^3, \mu\e^2):=-8 r_1'(\e^3, \mu\e^2)$).
As $\te_{22}(\tth)>0$,   they 
have non-zero real part if and only if $\DeltaBF(\tth; \mu, \e)>0$.

The eigenvalues of the matrix $ \mathtt{S} $ are a pair of purely imaginary eigenvalues of the form
\begin{equation}\label{eigenstable}
 \lambda_0^\pm (\mu, \e) = 
  \im\ch\mu \big(1+{r_9(\e^2,\mu\e)\big)} \mp \im \sqrt{\mu\tanh(\tth\mu)}\big(1+ {r(\e)}\big)\, . 
\end{equation}
{For $ \e = 0$ the eigenvalues $ \lambda_1^\pm(\mu,0),  \lambda_0^\pm (\mu,0) $ 
coincide with those in \eqref{omeghino}.}
\end{teo}

\begin{rmk}
At $\e = 0$, the eigenvalues in \eqref{l1piumeno} have the Taylor expansion
$$
\lambda^\pm_1(\mu,0) = \im (\ch- \frac12 \te_{12}(\tth)) \mu 
\pm \im \frac{\te_{22}(\tth)}{8} \mu^2 + \cO(\mu^3) \, , 
$$
which coincides with the  one of 
$\lambda^\pm_1(\mu)$ in 
\eqref{omeghino},  in view of the coefficients
 $\te_{12}(\tth)$ and $\te_{22}(\tth)$
defined in \eqref{te12}, \eqref{te22}.
\end{rmk}

We conclude this section describing in detail our  approach.
\\[1mm]
{\bf Ideas  and scheme of proof.} 
The proof follows the general ideas of the infinitely  deep water case \cite{BMV1}, although important differences arise in finite depth and require a different approach.
The first step 
is to exploit   Kato's theory 
to prolong 
the unperturbed symplectic basis
$\{f_1^\pm, f_0^\pm\}$ of $\cV_{0,0}$ in \eqref{basestart}-\eqref{basestartadd} into a 
symplectic basis of 
the spectral subspace $\cV_{\mu,\e}$ 
associated with  $\sigma'\left(\cL_{\mu,\e}\right) $  in \eqref{SSE}, depending 
analytically on $\mu, \e $. 
The transformation  operator $U_{\mu,\e}$ in Lemma \ref{lem:Kato1} is  
  symplectic, analytic in  $\mu,\e$, and 
 maps isomorphically $\cV_{0,0}$ into $\cV_{\mu,\e}$.
The vectors
$ f^\sigma_k(\mu,\e) := U_{\mu,\e}f_k^\sigma$, $ k = 0,1$, $\sigma = \pm$, 
are the required symplectic basis of  the symplectic subspace $\cV_{\mu,\e} $.
Its expansion in $\mu,\e$ is provided in Lemma \ref{expansion1}.
This procedure reduces
our spectral problem 
to determine 
the eigenvalues of the $4\times 4$ 
Hamiltonian and reversible   matrix $\tL_{\mu,\e}$ (cfr. Lemma \ref{lem:B.mat}), 
representing the action of the operator $ \cL_{\mu,\e }- \im \ch \mu$ on the basis $ \{f_k^\sigma(\mu,\e)\} $.  
In  Proposition \ref{BexpG} we prove that 
\begin{equation}
\label{L.form}
 \tL_{\mu,\e}=
\tJ_4 \begin{pmatrix} 
 E &  F \\ 
F^* &  G 
\end{pmatrix} = 
\begin{pmatrix} 
\tJ_2 E & \tJ_2 F \\ 
\tJ_2 F^* &\tJ_2 G 
\end{pmatrix} \qquad \text{where} \qquad 
\tJ_4 =  \begin{pmatrix} 
 \tJ_2 &  0 \\ 
0 &   \tJ_2 
\end{pmatrix} \, , \ \  
 \tJ_2 = \begin{psmallmatrix} 0 & 1 \\ -1 & 0 \end{psmallmatrix}  \, , 
\end{equation}
and the $ 2 \times 2 $ matrices
 $ E, G, F $ have the expansions \eqref{BinG1}-\eqref{BinG3}.  
  In finite depth this computation  is much more involved than in deep water,  
  as  we need to track  the exact dependence of the
 matrix entries with respect to $\tth$.
 In particular the matrix $E$ is 
\begin{equation}
\label{Eintro}
E = \begin{pmatrix} 
  \te_{11} \e^2(1+r_1'(\e,\mu \e)) - \te_{22}\frac{\mu^2}{8}(1+r_1''(\e,\mu))  
  & 
  \im \big( \frac12\te_{12}\mu+ r_2(\mu\e^2,\mu^2\e,\mu^3) \big)  \\
- \im \big( \frac12\te_{12} \mu+ r_2(\mu\e^2,\mu^2\e,\mu^3) \big) & -\te_{22}\frac{\mu^2}{8}(1+r_5(\e,\mu))
 \end{pmatrix}
 \end{equation}
 where the coefficients $\te_{11} $ and $\te_{22}$,  
  defined  in \eqref{te11} and \eqref{te22}, are strictly positive {\em for any} value of $\tth>0$. 
Thus the  submatrix $\tJ_2 E$ has a pair of eigenvalues with nonzero real part, for any value of $\tth>0$,  provided $0<\mu < \overline{\mu}(\e) \sim \e$.   
On the other hand, it has to come out that the complete $4\times4$ matrix $\tL_{\mu,\e}$ possesses unstable  eigenvalues {\em if and only if}  the depth exceeds the celebrated Whitham-Benjamin threshold $\tthWB \sim 1.363\ldots$.
Indeed the correct eigenvalues of $\tL_{\mu,\e}$  are not a small perturbation of those of
$ \footnotesize{\begin{pmatrix} 
\tJ_2 E & 0 \\ 
0 &\tJ_2 G 
\end{pmatrix}} $ and 
will emerge only after one non-perturbative step of block diagonalization.
 This was not the case in the infinitely deep water case \cite{BMV1}, where at this stage the corresponding submatrix $\tJ_2 E$ had already the correct  Benjamin-Feir eigenvalues, and we only had to check their stability  under perturbation.

\begin{rmk}
We also stress that \eqref{Eintro} 
is not a simple Taylor expansion in $\mu, \e $: note that 
the $(2,2)$-entry in \eqref{Eintro} does not have any term 
 $\cO ( \e^m )$ nor $ \cO ( \mu \e^m ) $ {\em for any} $ m \in \bN$. 
These terms would be dangerous because they could change the sign of the 
entry $(2,2)$ which instead, in \eqref{Eintro}, is always negative (recall that $\te_{22}(\tth) >0$).
We prove the absence of  terms  $\e^m$, $ m \in \bN$,  
 fully exploiting (as in \cite{BMV1}) the  structural information \eqref{genespace}
concerning the four dimensional generalized Kernel of 
 the operator $\cL_{0,\e}$ for any $\e >0$, see Lemma \ref{lem:B1}.
 Moreover, in finite depth 
 it turns out that there are no terms of order $ \mu \e^m $, $ m \in \bN$, 
 which instead are present in deep water,  and  were eliminated 
 in \cite{BMV1} via a further 
 change of basis.
We also note that the $2 \times 2$ matrices  $\tJ_2 E $ and $\tJ_2 G$ in \eqref{L.form} have both  eigenvalues of size $\cO(\mu)$. 
As already mentioned in the introduction, this is a crucial difference with the deep water case, where the eigenvalues of  $\tJ_2 G$ have the much larger size  $\cO(\sqrt \mu)$.
\end{rmk}

In order to determine the correct spectrum of  the matrix $\tL_{\mu,\e}$ in \eqref{L.form}, 
we 
perform a 
block diagonalization of $\tL_{\mu,\e}$  to  eliminate   the coupling term
 $  \tJ_2 F $ (which has size $ \e$, see \eqref{BinG3}). 
We proceed,  in  Section \ref{sec:block},  in three steps:
\\[1mm]
\indent 1. \emph{Symplectic rescaling}.
We first perform a  symplectic rescaling which is singular at $\mu =0$, see Lemma \ref{decoupling1prep}, obtaining the matrix $\tL_{\mu,\e}^{(1)}$.
The effects 
are twofold:  (i)  the diagonal elements of  
\begin{equation}\label{E11i}
E^{(1)} = 
\begin{pmatrix} 
  \te_{11} \mu\e^2(1+r_1'(\e,\mu\e))- \te_{22}\frac{\mu^3}{8}(1+r_1''(\e,\mu))  & \im \big( \frac12\te_{12}\mu+ r_2(\mu\e^2,\mu^2\e,\mu^3) \big)  \\
- \im\big( \frac12\te_{12}\mu+ r_2(\mu\e^2,\mu^2\e,\mu^3) \big) & -\te_{22}\frac{\mu}{8}(1+r_5(\e,\mu))
 \end{pmatrix}
\end{equation}
have  size  $\cO(\mu)$, as well as those of  $G^{(1)}$,  and (ii) the  matrix $F^{(1)}$ has the smaller size  $ { \cO(\mu \e ) } $. 
\\[1mm]
\indent 2. \emph{Non-perturbative step of block-diagonalization
 (Section \ref{sec:5.2})}. 
Inspired by KAM theory, we  perform one step of block decoupling to decrease further the size of the off-diagonal blocks.
This step 
modifies the matrix $\tJ_2 E^{(1)}$ in a substantial way, by a term $ \cO (\mu \e^2 )$. 
Let us explain better  this step.
In order to reduce the size of $\tJ_2 F^{(1)} $, 
we   conjugate  $\tL_{\mu,\e}^{(1)}$ by 
the symplectic  matrix $\exp(S^{(1)})$, where $S^{(1)}$  is a Hamiltonian matrix
with  the same form  of 
$ \tJ_2 F^{(1)} $, see \eqref{formaS}.  
The transformed  matrix $\tL_{\mu,\e}^{(2)} = 
\exp(S^{(1)}) \tL_{\mu,\e}^{(1)}\exp(-S^{(1)})
$  has the 
Lie expansion\footnote{recall that   $\exp(S) L \exp(-S) = \sum_{n \geq 0} \frac{1}{n!} \textup{ad}_S^n(L)$, where
$\textup{ad}_S^0(L) := L$,   $\textup{ad}_S^n(L) = [S,  \textup{ad}_S^{n-1}(L)]$ for $n \geq 1$.}
\begin{equation}\label{L2.lie.expl}
\begin{aligned}
\tL_{\mu,\e}^{(2)} 
& =
\begin{psmallmatrix} 
\tJ_2 E^{(1)} & 0 \\ 
0 &\tJ_2  G^{(1)}
\end{psmallmatrix} \\
& \quad + 
\begin{psmallmatrix} 
0 & \tJ_2 F^{(1)} \\ 
\tJ_2 [F^{(1)}]^* & 0 
\end{psmallmatrix}
 +\lie{S^{(1)}}{\begin{psmallmatrix} 
\tJ_2 E^{(1)} & 0 \\ 
0 &\tJ_2  G^{(1)}
\end{psmallmatrix} }\\
& \quad 
+ \frac12 \Big[ S^{(1)}, \Big[ S^{(1)}, \begin{psmallmatrix} 
\tJ_2 E^{(1)} & 0 \\ 
0 &\tJ_2  G^{(1)}
\end{psmallmatrix} \Big] \Big] 
   + \Big[ S^{(1)}, \begin{psmallmatrix} 
0 & \tJ_2 F^{(1)} \\ 
\tJ_2 [F^{(1)}]^* & 0 
\end{psmallmatrix} \Big]  + \mbox{h.o.t.}
\end{aligned}
\end{equation}
The first line in the right hand side of \eqref{L2.lie.expl} is the previous 
block-diagonal matrix, the second line of \eqref{L2.lie.expl} is a purely off-diagonal matrix and the third line is the sum of two  block-diagonal matrices  and ``h.o.t."  collects terms of much smaller size.
 $S^{(1)}$ is determined in such a way that 
the second line of \eqref{L2.lie.expl} vanishes, and therefore 
 the  remaining off-diagonal matrices (appearing in the h.o.t. remainder) are  smaller in size. 
Unlike the infinitely deep water case \cite{BMV1}, the 
 block-diagonal corrections in the third line of \eqref{L2.lie.expl} are {\em not} perturbative and they modify substantially the block-diagonal part. 
More precisely  we obtain that $ \tL_{\mu,\e}^{(2)}  $ has the form  \eqref{sylvydec} with 
$$
 E^{(2)}
:=
{\begin{pmatrix} 
\mu\e^2 \teWB+ r_1'(\mu \e^3, \mu^2 \e^2 )-\te_{22}\frac{\mu^3}{8}(1+r_1''(\e,\mu))  & \im  \big( \frac12\te_{12}\mu+ r_2(\mu\e^2,\mu^2\e,\mu^3) \big)  \\
- \im  \big( \frac12\te_{12}\mu+ r_2(\mu\e^2,\mu^2\e,\mu^3) \big) & -\te_{22}\frac{\mu}{8}(1+r_5(\e,\mu))
 \end{pmatrix}}\, .
$$
Note the appearance of 
the Whitham-Benjamin function $\teWB (\tth) $ 
in the   (1,1)-entry of $  E^{(2)} $,   
which changes sign at the critical  depth  $\tthWB$, see Figure \ref{graficoe112}, unlike the coefficient 
$ \te_{11} (\tth)> 0 $ in \eqref{E11i}.
If 
 $\teWB(\tth) >0$ and 
 $\e$ and $\mu$ are sufficiently small, the matrix  $\tJ_2 E^{(2)}$ has eigenvalues with non-zero real part (recall that
$\te_{22}(\tth)>0$ for any $ \tth $).
On the contrary, if $\teWB(\tth) <0$, then 
the eigenvalues of $\tJ_2 E^{(2)}$ lay on the imaginary axis. 
\\[1mm]
\indent 3. \emph{Complete block-diagonalization 
 (Section \ref{section34})}.
In Lemma \ref{ultimate}
we completely block-diagonalize $\tL^{(2)}_{\mu,\e}$
by means of a standard implicit function theorem.
By this procedure
the original matrix $\tL_{\mu,\e}$
is conjugated 
into  the Hamiltonian and reversible matrix \eqref{matricefinae}, proving
Theorem \ref{TeoremoneFinale}. 

\section{Perturbative approach to the separated eigenvalues}\label{Katoapp}

We apply Kato's similarity transformation 
theory \cite[I-\textsection 4-6, II-\textsection 4]{Kato}  to study the splitting of the eigenvalues of 
$ \cL_{\mu,\e} $ close to $ 0 $ for small values of $ \mu $ and $ \e $,
following \cite{BMV1}. 
First of all, it is convenient to decompose the operator $ \cL_{\mu,\e}$ in \eqref{WW} as 
  \begin{equation}\label{calL}
 \cL_{\mu,\e}  = \im  \ch \mu + \sL_{\mu,\e} \, , \qquad \mu > 0 \, ,  
\end{equation}
where, using also \eqref{|D+mu|}, 
\begin{equation}\label{calL2}
\sL_{\mu,\e}:= 
\begin{bmatrix} \pa_x\circ (\ch+p_\e(x)) + \im \mu \,  p_\e(x) &
 |D+ \mu| \, \tanh\big((\tth+\ttf)|D+\mu| \big)\\ -(1+a_\e(x)) & (\ch+p_\e(x))\pa_x+\im \mu\,  p_\e(x) \end{bmatrix}  \, . 
\end{equation}
 The operator $\sL_{\mu,\e}$ is still  Hamiltonian, having the form 
\begin{equation}\label{calL.ham}
\sL_{\mu,\e} = 
 \cJ \, {\cal B}_{\mu, \e} \, , \quad
{\cal B}_{\mu, \e}
:= \begin{bmatrix} 1+a_\e(x) & -(\ch+p_\e(x))\pa_x-\im \mu\,  p_\e(x)  \\
 \pa_x\circ (\ch+p_\e(x)) + \im \mu \,  p_\e(x) & |D+ \mu| \, \tanh\big((\tth+\ttf)|D+\mu|\big) \end{bmatrix}
\end{equation}
with ${\cal B}_{\mu, \e}$ selfadjoint, and 
 it is also  reversible, namely it satisfies, by \eqref{Reversible}, 
 \begin{equation}\label{calL.rev}
\sL_{\mu,\e}\circ \bro =- \bro \circ \sL_{\mu,\e} \, , 
\qquad
\bro \mbox{ defined in }  \eqref{reversibilityappears} \, , 
 \end{equation}
 whereas ${\cal B}_{\mu,\e}$ is reversibility-preserving, i.e. fulfills \eqref{B.rev.pres}. Note also that ${\cal B}_{0,\e}$ is a real operator.

 The scalar operator $ \im \ch \mu \equiv \im \ch \mu \, \text{Id}$ 
just translates the spectrum of $ \sL_{\mu,\e}$ 
along the imaginary  axis of  the quantity $ \im \ch \mu $, that is, in view of \eqref{calL}, 
$ 
\sigma ({\mathcal L}_{\mu,\e}) = \im \ch \mu + \sigma (\sL_{\mu,\e}) \, . 
$ 
Thus in the sequel we focus on studying the spectrum of 
$ \sL_{\mu,\e}$.

Note also that $\sL_{0,\e} = \cL_{0,\e}$ for any $\e \geq 0$.
In particular $\sL_{0,0}$ has zero as isolated eigenvalue with algebraic multiplicity 4, geometric multiplicity 3 and generalized kernel spanned by the vectors  $\{f^+_1, f^-_1, f^+_0, f^-_0\}$ in \eqref{basestart}, \eqref{basestartadd}.
Furthermore its spectrum is separated as in \eqref{spettrodiviso0}.
For  any $\e \neq 0$ small, $\sL_{0,\e}$ has  zero as  isolated eigenvalue 
with geometric multiplicity $2$, and two generalized eigenvectors satisfying 
\eqref{genespace}.

We  remark that, in view of \eqref{|D+mu|}, the operator
 $\sL_{\mu,\e}$ is analytic with respect to $\mu$.
The operator $ \sL_{\mu,\e}  : Y \subset X \to X $   
has domain $Y:=H^1(\mathbb{T}):=H^1(\mathbb{T},\bC^2)$ and range $X:=L^2(\mathbb{T}):=L^2(\mathbb{T},\bC^2)$. 


\begin{lem}\label{lem:Kato1}
{\bf (Kato theory for separated eigenvalues)}
 Let $\Gamma$ be a closed, counterclockwise-oriented curve around $0$ in the complex plane separating $\sigma'\left(\sL_{0,0}\right)=\{0\}$
  and the other part of the spectrum $\sigma''\left(\sL_{0,0}\right)$ in \eqref{spettrodiviso0}.
There exist $\e_0, \mu_0>0$  such that for any $(\mu, \e) \in B(\mu_0)\times B(\e_0)$  the following statements hold:
\\[1mm] 
1. The curve $\Gamma$ belongs to the resolvent set of 
the operator $\sL_{\mu,\e} : Y \subset X \to X $ defined in \eqref{calL2}.
\\[1mm] 
2.
The operators
\begin{equation}\label{Pproj}
 P_{\mu,\e} := -\frac{1}{2\pi\im}\oint_\Gamma (\sL_{\mu,\e}-\lambda)^{-1} \de\lambda : X \to Y 
\end{equation}  
are well defined projectors commuting  with $\sL_{\mu,\e}$,  i.e. 
$ P_{\mu,\e}^2 = P_{\mu,\e} $ and 
$ P_{\mu,\e}\sL_{\mu,\e} = \sL_{\mu,\e} P_{\mu,\e} $. 
The map $(\mu, \epsilon)\mapsto P_{\mu,\epsilon}$ is  analytic from 
$B({\mu_0})\times B({\epsilon_0})$
 to $ \cL(X, Y)$.
\\[1mm] 
3.
The domain $Y$  of the operator $\sL_{\mu,\e}$ decomposes as  the direct sum
$$
Y= \mathcal{V}_{\mu,\e} \oplus \text{Ker}(P_{\mu,\e}) \, , \quad \mathcal{V}_{\mu,\e}:=\text{Rg}(P_{\mu,\e})=\text{Ker}(\uno-P_{\mu,\e}) \, ,
$$
of   closed invariant  subspaces, namely 
$ \sL_{\mu,\e} : \mathcal{V}_{\mu,\e} \to \mathcal{V}_{\mu,\e} $, $
\sL_{\mu,\e} : \text{Ker}(P_{\mu,\e}) \to \text{Ker}(P_{\mu,\e}) $.  
Moreover 
$$
\begin{aligned}
&\sigma(\sL_{\mu,\e})\cap \{ z \in \bC \mbox{ inside } \Gamma \} = \sigma(\sL_{\mu,\e}\vert_{{\mathcal V}_{\mu,\e}} )  = \sigma'(\sL_{\mu, \e}) , \\
&\sigma(\sL_{\mu,\e})\cap \{ z \in \bC \mbox{ outside } \Gamma \} = \sigma(\sL_{\mu,\e}\vert_{Ker(P_{\mu,\e})} )  = \sigma''( \sL_{\mu, \e}) \ ,
\end{aligned}
$$
proving the  ``semicontinuity property" \eqref{SSE} of separated parts of the spectrum.
\\[1mm] 
4.  The projectors $P_{\mu,\e}$ 
are similar one to each other: the  transformation operators\footnote{
 The operator $(\uno-R)^{-\frac12} $ is defined, for any 
operator $ R $ satisfying $\|R\|_{{\cL}(Y)}<1 $,  by the power series
\begin{align}\label{rootexp}
 (\uno - R)^{-\frac12} :=  \sum_{k=0}^\infty {-1/2 \choose k}(-R)^k = \uno + \frac{1}{2}R + \frac{3}{8}R^2+\cO(R^3) \, .
\end{align}
}
\begin{equation} \label{OperatorU} 
U_{\mu,\e}   := 
\big( \uno-(P_{\mu,\e}-P_{0,0})^2 \big)^{-1/2} \big[ 
P_{\mu,\e}P_{0,0} + (\uno - P_{\mu,\e})(\uno-P_{0,0}) \big] 
\end{equation}
are bounded and  invertible in $ Y $ and in $ X $, with inverse
$$
U_{\mu,\e}^{-1}  = 
 \big[ 
P_{0,0} P_{\mu,\e}+(\uno-P_{0,0}) (\uno - P_{\mu,\e}) \big] \big( \uno-(P_{\mu,\e}-P_{0,0})^2 \big)^{-1/2} \, , 
$$
 and 
$ U_{\mu,\e} P_{0,0}U_{\mu,\e}^{-1} =  P_{\mu,\e}  
$ as well as $ U_{\mu,\e}^{-1} P_{\mu,\e}  U_{\mu,\e} = P_{0,0} $. 

The map $(\mu, \epsilon)\mapsto  U_{\mu,\e}$ is analytic from  $B(\mu_0)\times B(\epsilon_0)$ to $\cL(Y)$.
\\[1mm] 
5. The subspaces $\mathcal{V}_{\mu,\e}=\text{Rg}(P_{\mu,\e})$ are isomorphic one to each other: 
$
\mathcal{V}_{\mu,\e}=  U_{\mu,\e}\mathcal{V}_{0,0}.
$
 In particular $\dim \mathcal{V}_{\mu,\e} = \dim \mathcal{V}_{0,0}=4 $, for any 
 $(\mu, \e) \in B(\mu_0)\times B(\e_0)$.
\end{lem}

\begin{proof}
For any $ \lambda \in \bC $ we decompose 
$\sL_{\mu,\e}-\lambda= \sL_{0,0}-\lambda + {\cal R}_{\mu,\e} $
where $ \footnotesize \sL_{0,0} = \begin{bmatrix} \ch\pa_x & |D| \tanh(\tth |D|) \\ -1 & \ch \pa_x \end{bmatrix}$ and 
$$
{\cal R}_{\mu,\e}:=\sL_{\mu,\e}-\sL_{0,0} = \begin{bmatrix}  (\pa_x +\im \mu) p_\e(x) & f_{\mu,\e}(D) \\ -a_\e(x) &  p_\e(x)(\pa_x  + \im \mu) \end{bmatrix}: Y \to X \, ,
$$
having used also  \eqref{|D+mu|} and setting 
$$
f_{\mu,\e}(D) :=  |D+ \mu| \, \tanh\big((\tth+\ttf)|D+\mu|\big) - 
 |D| \tanh(\tth |D|) \in \cL(Y) \, , \ \  \  
 \norm{f_{\mu,\e}(D)}_{\cL(Y)} = \cO(\mu,\e) \, . 
$$
 For any $\lambda \in \Gamma$, 
the operator  $\sL_{0,0}-\lambda$ is invertible and its inverse is the Fourier multiplier 
matrix operator  
$$
(\sL_{0,0}-\lambda)^{-1} = \text{Op}\left( \frac{1}{(\im \ch k-\lambda )^2 + |k| \tanh(\tth |k|)}
\begin{bmatrix} \im \ch k - \lambda & -|k|\tanh(\tth |k|) \\ 1 & \im \ch k - \lambda  \end{bmatrix} \right): X \to Y \, .
$$
Hence, for $|\e|<\e_0$ and  $|\mu|<\mu_0$ small enough, uniformly on the compact set $\Gamma$, the operator $(\sL_{0,0}-\lambda)^{-1}{\cal R}_{\mu,\e}:Y\to Y$ is bounded,  with small operatorial norm. Then $\sL_{\mu,\e}-\lambda$ is invertible by Neumann series and
$\Gamma$ belongs to the resolvent set of $\sL_{\mu,\e}$. 
 The remaining part of the proof follows exactly  as  in Lemma 3.1 in \cite{BMV1}. 
\end{proof}

The Hamiltonian and reversible nature of
the operator $ \sL_{\mu,\e} $, see \eqref{calL.ham} and \eqref{calL.rev}, imply
additional   algebraic properties 
for   spectral projectors $P_{\mu,\e}$ and the transformation operators $U_{\mu,\e} $.
By Lemma  3.2 in \cite{BMV1} we have that: 

\begin{lem}\label{propPU}
 For any $(\mu, \epsilon)  \in B(\mu_0)\times  B(\epsilon_0)$, the following holds true:
\\[1mm]
($i$)
The projectors $P_{\mu,\e}$ defined in \eqref{Pproj} are skew-Hamiltonian, namely $\cJ P_{\mu,\e}=P_{\mu,\e}^*\cJ $,  and reversibility preserving, i.e. 
$ \bro P_{\mu,\e} = P_{\mu,\e}  \bro $.
\\[1mm]
(ii) The transformation operators $U_{\mu,\e}$ in \eqref{OperatorU} are symplectic, namely 
 $ U_{\mu,\e}^* \cJ U_{\mu,\e}= \cJ $, and reversibility preserving.
 \\[1mm]
 (iii)  $P_{0,\e}$ and $U_{0,\e}$ are real operators, i.e. $\bar{P_{0,\e}}=P_{0,\e}$ and $\bar{U_{0,\e}}=U_{0,\e}$.
\end{lem}

By the previous lemma, the linear involution $\bro $ 
commutes with the spectral projectors $P_{\mu,\e}$ and then 
$\bro $ leaves invariant the subspace 
$ \mathcal{V}_{\mu,\e} = \text{Rg}(P_{\mu,\e}) $.  
\\[1mm]{\bf Symplectic and reversible basis of $\mathcal{V}_{\mu,\e}$.}
It is convenient to represent the Hamiltonian and reversible operator
$ \sL_{\mu,\e} : \mathcal{V}_{\mu,\e} \to \mathcal{V}_{\mu,\e} $ in a 
basis which is symplectic and reversible, according to the following definition. 
\begin{sia}\label{def:SR} 
{\bf (Symplectic and reversible basis)}
 A basis $\mathtt{F}:=\{\mathtt{f}^+_1,\,\mathtt{f}^-_1,\,\mathtt{f}^+_0,\,\mathtt{f}^-_0 \}$ of $\mathcal{V}_{\mu,\e}$ is 
 \begin{itemize}
 \item 
 \emph{symplectic} if, for any $ k, k' = 0,1 $,  
 \begin{equation}\label{symplecticbasis}
 \molt{\cJ \tf_k^-}{\tf_k^+} = 1 \, ,  \ \  
 \big( \cJ \tf_k^\sigma, \tf_k^\sigma \big) = 0 \, , \  \forall \sigma = \pm \, ;  
  \ \  \text{if} \ k \neq k' \  \text{then} \  
\big( \cJ \tf_k^\sigma, \tf_{k'}^{\sigma'} \big) = 
   0 \, , \ \forall \sigma, \sigma' = \pm  \, . 
\end{equation}
\item 
\emph{reversible} if
\begin{equation}
\label{reversiblebasis}
 \bar \rho \tf^+_1 =  \tf^+_1 , \quad   \bar \rho \tf^-_1 = - \tf^-_1 , \quad 
 \bar \rho \tf^+_0 =  \tf^+_0 , \quad   \bar \rho \tf^-_0 = - \tf^-_0, \quad \text{i.e. } \bro \tf_k^\sigma = \sigma \tf_k^\sigma \, , \ \forall \sigma = \pm, k = 0,1 \, .
\end{equation}
\end{itemize}
\end{sia}
We use the following 
notation along the paper:  we denote by $even(x)$ a real $2\pi$-periodic function which is even in $x$, and by $odd(x)$ a real $2\pi$-periodic function which is odd in $x$.

By the definition of the involution $\bro$ in \eqref{reversibilityappears}, 
the real and imaginary parts of 
a reversible basis $\mathtt{F}=\{\mathtt{f}^\pm_k \}$, $k=0,1$,  enjoy the following parity properties (cfr. Lemma 3.4 in \cite{BMV1})
\begin{equation}\label{reversiblebasisprop}
  \tf_k^+(x) = \vet{even(x)+\im odd(x)}{odd(x)+\im even(x)}, \quad
  \tf_k^-(x) = \vet{odd(x)+\im even(x)}{even(x)+\im odd(x)}.
\end{equation}
By Lemmata 3.5 and 3.6  in \cite{BMV1} we have the following result. 

\begin{lem}\label{lem:B.mat}
The $ 4 \times 4 $ matrix that represents the Hamiltonian and reversible operator 
$\sL_{\mu,\e}= \cJ {\cal B}_{\mu,\e}:\mathcal{V}_{\mu,\e}\to\mathcal{V}_{\mu,\e} $ with respect to a symplectic and reversible basis $\mathtt{F}=\{\tf_1^+,\tf_1^-,\tf_0^+,\tf_0^-\} $ of $\mathcal{V}_{\mu,\e}$ is 
\begin{align}\label{Lform}
 \tJ_4 \tB_{\mu,\e} \, ,\quad  \tJ_4 := 
 \begin{pmatrix} 
 \tJ_2& \vline & 0 \\
 \hline
0  & \vline & \tJ_2
\end{pmatrix}, \quad 
{\small \tJ_2 := \begin{pmatrix} 
 0 & 1 \\
-1  & 0
\end{pmatrix}}, \quad \text{where } \quad \tB_{\mu,\e}= \tB_{\mu,\e}^* \end{align}
is the self-adjoint matrix
\begin{equation}\label{matrix22} 
\tB_{\mu,\e} = 
\begin{pmatrix}
\BVe{+}{1}{+}{1} & \BVe{-}{1}{+}{1} & \BVe{+}{0}{+}{1} & \BVe{-}{0}{+}{1} \\
\BVe{+}{1}{-}{1} & \BVe{-}{1}{-}{1} & \BVe{+}{0}{-}{1} & \BVe{-}{0}{-}{1} \\
\BVe{+}{1}{+}{0} & \BVe{-}{1}{+}{0} & \BVe{+}{0}{+}{0} & \BVe{-}{0}{+}{0} \\
\BVe{+}{1}{-}{0} & \BVe{-}{1}{-}{0} & \BVe{+}{0}{-}{0} & \BVe{-}{0}{-}{0} \\
	\end{pmatrix}.
\end{equation}
 The entries of the matrix $\tB_{\mu,\e}$ are alternatively  real or purely imaginary: for any $ \sigma = \pm $, $ k = 0, 1 $, 
\begin{equation}\label{revprop}
 \molt{{\cal B}_{\mu,\e}  \, \tf^{\sigma}_{k}}{\tf^{\sigma}_{k'}} \text{ is real},\qquad 
   \molt{{\cal B}_{\mu,\e}  \, \tf^{\sigma}_{k}}{\tf^{-\sigma}_{k'}} \text{ is purely imaginary} \, .
\end{equation}
\end{lem}

 It is convenient to give a name to the matrices of the form obtained 
in Lemma \ref{lem:B.mat}.

\begin{sia} 
A $ 2n \times 2n $, $ n = 1,2, $ matrix of the form 
$\tL=\tJ_{2n} \tB$ is 
\\[1mm]
1. \emph{Hamiltonian} if  $ \tB $ is a self-adjoint matrix, i.e.   $\tB=\tB^*$;
\\[1mm]
2. \emph{Reversible} if $\tB$ is reversibility-preserving,   i.e. $\rho_{2n}\circ \tB = \tB \circ \rho_{2n} $, where 
\begin{equation}\label{involutionrep}
 \rho_4 := \begin{pmatrix}\rho_2 & 0 \\ 0 & \rho_2\end{pmatrix}, \qquad \rho_2 := \begin{pmatrix} \mathfrak{c}  & 0 \\ 0 & - \mathfrak{c} \end{pmatrix},
\end{equation}
and $\Gc : z \mapsto \bar z $ is the conjugation of the complex plane.
Equivalently, $\rho_{2n} \circ \tL  = -  \tL \circ \rho_{2n}$.
\end{sia}
In the sequel we shall mainly deal with $ 4 \times 4 $ Hamiltonian and reversible matrices.
The transformations 
preserving  the Hamiltonian structure  are called
  \emph{symplectic}, and  satisfy
\begin{align}\label{sympmatrix}
 Y^* \tJ_4 Y = \tJ_4 \, .
 \end{align}
 If $Y$ is symplectic then $Y^*$ and $Y^{-1}$  are symplectic as well. A Hamiltonian matrix $\tL=\tJ_4 \tB$, with $\tB=\tB^*$, is conjugated through $Y$ in the new Hamiltonian matrix
 \begin{equation}
 \tL_1 = Y^{-1}  \tL Y = Y^{-1}  \tJ_4 Y^{-*} Y^*  \tB Y = \tJ_4 \tB_1 \quad \text{where } \quad 
  \tB_1 := Y^* \tB Y = \tB_1^* \, . \label{sympchange}
 \end{equation}
Note that the matrix $ \rho_4 $ in \eqref{involutionrep} represents the action of the 
involution $\bar \rho : {\mathcal V}_{\mu,\e}
\to {\mathcal V}_{\mu,\e}  $ defined in \eqref{reversibilityappears} in a 
reversible basis (cfr. \eqref{reversiblebasis}).
A $ 4\times 4$ matrix $\tB=(\tB_{ij})_{i,j=1,\dots,4}$ is reversibility-preserving if and only if its entries are alternatively real and purely imaginary, namely $\tB_{ij}$ is real when $i+j$ is even and purely imaginary otherwise, as in \eqref{revprop}.
A $4\times 4$ complex matrix $\tL =(\tL_{ij})_{i,j=1, \ldots, 4}$ is reversible if and only if  $\tL_{ij}$ is purely imaginary when $i+j$ is even and real otherwise.
 
We finally mention 
that the flow of a Hamiltonian 
reversibility-preserving matrix  is symplectic
and reversibility-preserving (see Lemma 3.8 in \cite{BMV1}).

\section{Matrix representation of $ \sL_{\mu,\e}$ on $ \mathcal{V}_{\mu,\e}$}\label{sec:mr}

Using the transformation operators $U_{\mu,\e}$  in \eqref{OperatorU},  
we  construct  the 
basis of $\cV_{\mu,\e}$ 
\begin{equation}\label{basisF}
\cF := \big\{ 
f_{1}^+(\mu,\e), \  f_{1}^- (\mu,\e), \   f_{0}^+(\mu,\e),\   f_{0}^-(\mu,\e) \big\} \, , \quad 
f_{k}^\sigma(\mu,\e) := U_{\mu,\e} f_{k}^\sigma \, , \ \sigma=\pm \, , \,k=0,1 \, , 
\end{equation}
where 
\begin{equation}
\label{funperturbed}
f_1^+ = \vet{\ch^{1/2} \cos (x)}{\ch^{-1/2} \sin (x)} , \quad
f_1^- = \vet{-  \ch^{1/2} \sin (x)}{\ch^{-1/2} \cos (x)} , \quad 
f_0^+ = \vet{1}{0}, \quad
f_0^- = \vet{0}{1} \, , 
\end{equation}
form a basis of $ \mathcal{V}_{0,0} =\mathrm{Rg} (P_{0,0}) $, cfr. \eqref{basestart}-\eqref{basestartadd}. 
Note that the  real valued vectors $ \{ f_1^\pm, f_0^\pm \} $  
form  a  symplectic and reversible basis for 
$ \mathcal{V}_{0,0} $, according to Definition \ref{def:SR}. 
Then, by Lemma \ref{propPU} and  \ref{lem:Kato1} we deduce that (cfr. Lemma 4.1 in \cite{BMV1}):

\begin{lem}\label{base1symp}
The  basis $  \cF  $ of $\mathcal{V}_{\mu,\e}$ defined in \eqref{basisF}, 
is symplectic and reversible, i.e. satisfies \eqref{symplecticbasis} 
and \eqref{reversiblebasis}. 
Each map $(\mu, \e) \mapsto f^\sigma_k(\mu, \e)$ is analytic as a map $B(\mu_0)\times B(\epsilon_0) \to H^1(\bT)$. 
\end{lem}

In the next lemma we   expand the vectors  $ f_k^\sigma(\mu,\e) $ in $ (\mu, \e ) $. 
 We denote by $even_0(x)$ a real, even, $2\pi$-periodic function with zero space average. 
In the sequel 
 $\cO(\mu^{m} \e^{n}) \footnotesize\vet{even(x)}{odd(x)}$ denotes an analytic map in $(\mu, \e)$ with values in  $ H^1(\bT, \bC^2) $,
 whose first component is $even(x)$ and the second one  $odd(x)$;  similar meaning 
for $\cO(\mu^{m} \e^{n}) \footnotesize\vet{odd(x)}{even(x)}$, 
etc... 
\begin{lem}\label{expansion1} 
{\bf (Expansion of the basis $ \cF$)} For small values of $(\mu, \e)$ 
the basis $ \cF$ in \eqref{basisF}  has the expansion
\begin{align}
 f^+_1(\mu, \e) & = \vet{\ch^\frac12 \cos(x)}{\ch^{-\frac12}\sin(x)} + \im \frac{\mu}{4}\gamma_\tth \vet{\ch^\frac12\sin(x)}{\ch^{-\frac12}\cos(x)} + 
 \epsilon \vet{\alpha_\tth \cos(2x)}{\beta_\tth \sin(2x)} \label{exf41}
 \\ &  +
\cO(\mu^2) \vet{even_0(x) + \im odd(x)}{odd(x) + \im even_0(x)}  + \cO(\e^2) \vet{even_0(x)}{odd(x)} + \im\mu \epsilon\vet{odd(x)}{even(x)} + \cO(\mu^2\e,\mu\e^2) \, , \notag \\
 f^-_1(\mu, \e) \label{exf42} &= \vet{-\ch^\frac12 \sin(x)}{\ch^{-\frac12}\cos(x)} + \im \frac{\mu}{4} \gamma_\tth \vet{\ch^{\frac12}\cos(x)}{-\ch^{-\frac12}\sin(x)} 
 +
  \epsilon
   \vet{-\alpha_{\tth}\sin(2x)}{\beta_\tth \cos(2x)}\\ & + \cO(\mu^2) \vet{odd(x) + \im even_0(x)}{even_0(x) + \im odd(x)} + \cO(\e^2) \vet{odd(x)}{even(x)} 
+ \im\mu \epsilon\vet{even(x)}{odd(x)} + \cO(\mu^2\e,\mu\e^2) \, , \notag \\
 f^+_0(\mu, \e) \label{exf43} &= \vet{1}{0}+ \epsilon \delta_\tth \vet{\ch^\frac12 \cos(x)}{- \ch^{-\frac12}\sin(x)}
 + \cO(\e^2) \vet{even_0(x)}{odd(x)} + \im\mu \epsilon\vet{odd(x)}{even_0(x)}+ \cO(\mu^2\e,\mu\e^2) \, , \\
 f^-_0(\mu, \e) \label{exf44} & = \vet{0}{1}  + \im  \mu\e  \vet{even_0(x)}{odd(x)}+\cO(\mu^2\e,\mu\e^2) \, ,
\end{align}
where 
the remainders $\cO()$ are vectors in $H^1(\bT)$
and 
\begin{equation}\label{def:basiscoeff}
\alpha_\tth := \frac12 \ch^{-\frac{11}{2}}(3+\ch^4) \, , \quad
\beta_\tth := \frac14\ch^{-\frac{13}{2}}(1+\ch^4)(3-\ch^4) \, ,  
\quad \gamma_\tth :=  1+\frac{\tth (1-\ch^4)}{\ch^2} \, , \quad 
\delta_\tth := \frac{3+\ch^4}{4 \ch^{\frac52}} \, . 
\end{equation}
 For $\mu=0$ the basis $\{f_k^\pm(0,\e), k=0,1 \}  $ is real  and 
\begin{equation}\label{nonzeroaverage}
f^{+}_1 (0, \e) =  \vet{even_0(x)}{odd(x)},
\ 
f^{-}_1 (0, \e) =  \vet{odd(x)}{even(x)}, 
\  
f^{+}_0 (0, \e) = \vet{1}{0}+ \vet{even_0(x)}{odd(x)} \, , 
\   
f^{-}_0 (0, \e) =  \vet{0}{1} \, .
\end{equation}
\end{lem}
\begin{proof}
The long calculations are given in Appendix \ref{ProofExpansion}. 
\end{proof}

We now state the main result of this section.

\begin{prop}\label{BexpG}
The matrix that represents the Hamiltonian and reversible operator  
$ \sL_{\mu,\e} : \mathcal{V}_{\mu,\e} \to \mathcal{V}_{\mu,\e} $ in the symplectic and reversible basis $\mathcal{F}$ of $\mathcal{V}_{\mu,\e}$ defined in \eqref{basisF}, 
is a Hamiltonian matrix 
 $\tL_{\mu,\e}=\tJ_4 \tB_{\mu,\e}$, where 
 $\tB_{\mu,\e} $ is a self-adjoint and reversibility preserving
(i.e. satisfying \eqref{revprop})
 $ 4 \times 4$  matrix of the form 
  \begin{equation}\label{splitB}
\tB_{\mu,\e}=
\begin{pmatrix} 
E & F \\ 
F^* & G 
\end{pmatrix}, \qquad 
E  = E^* \, , \   \ G = G^* \, , 
\end{equation} 
  where $E, F, G$  are the $ 2 \times 2 $  matrices
\begin{align}\label{BinG1}
& E := 
\begin{pmatrix} 
  \te_{11} \e^2(1+r_1'(\e,\mu \e)) - \te_{22}\frac{\mu^2}{8}(1+r_1''(\e,\mu))  
  & 
  \im \big( \frac12\te_{12}\mu+ r_2(\mu\e^2,\mu^2\e,\mu^3) \big)  \\
- \im \big( \frac12\te_{12} \mu+ r_2(\mu\e^2,\mu^2\e,\mu^3) \big) & -\te_{22}\frac{\mu^2}{8}(1+r_5(\e,\mu))
 \end{pmatrix} \\
 & \label{BinG2} G := 
\begin{pmatrix} 
1+r_8(\e^2,\mu^2\e 
) &   - \im r_9(\mu\e^2,\mu^2\e) 
 \\
  \im  r_9(\mu\e^2, \mu^2\e 
  )  &\mu\tanh(\tth\mu)+ r_{10}(\mu^2\e 
  ) 
 \end{pmatrix} \\
 & \label{BinG3}
 F := 
\begin{pmatrix} 
\tf_{11}\e+ r_3(\e^3,\mu\e^2,\mu^2\e 
) & \im 
 \mu\e \ch^{-\frac12}   +\im  r_4({\mu\e^2}, \mu^2 \e 
 )  \\
  \im   r_6(\mu\e 
  )    & r_7(\mu^2\e 
  ) 
 \end{pmatrix} \, , 
 \end{align}
 with $\te_{12}$ and $\te_{22}$ given in \eqref{te12} and \eqref{te22} respectively, and 
 \begin{align}
\label{te11}
\te_{11} 
 & := 
 \dfrac{9\ch^8-10\ch^4+9}{8\ch^7} = \dfrac{9 (1-\ch^4)^2 +8 \ch^4 }{8\ch^7} > 
 0 \, , \qquad
 \tf_{11} :=    \tfrac12 \ch^{-\frac32}(1-\ch^4) \, .
\end{align} 
\end{prop}

The rest of this section is devoted to the proof of  Proposition \ref{BexpG}.  

We decompose $ {\cal B}_{\mu,\e}$ in \eqref{calL.ham} as 
  \begin{equation*}
   {\cal B}_{\mu,\e}  = {\cal B}_\e  + {\cal B}^\flat +  {\cal B}^\sharp  \, , 
\end{equation*}
where $ {\cal B}_\e $, $  {\cal B}^\flat $, $  {\cal B}^\sharp  $ 
are the self-adjoint and reversibility preserving operators
\begin{align}
 \label{cBe}
&  {\cal B}_\e  := {\cal B}_{0,\e} := \left[
\begin{array}{cc}
1+a_\e(x) &   - (\ch +p_\e(x)) \partial_x   \\
\partial_x \circ (\ch +p_\e(x)) &    |D|\tanh((\tth+\ttf)|D|)
\end{array}
\right],  \\ \label{cBflat}
& {\cal B}^\flat :=   
 \begin{bmatrix}
0 & 0 \\
0 &  |D+\mu|\tanh((\tth+\ttf)|D+\mu|)-|D|\tanh((\tth+\ttf)|D|)
\end{bmatrix} , \, \\
& \label{cBsharp} {\cal B}^\sharp :=  
\mu \begin{bmatrix}
0 &  -\im  p_{\e} \\
 \im p_{\e} & 0 
\end{bmatrix} \, .  
\end{align}
In view of \eqref{D+mu an} the operator $ {\cal B}^\flat $ is analytic in $ \mu $. 


\begin{lem}\label{lem:B1} {\bf (Expansion of $\mathtt{B}_\e$)} 
The self-adjoint  and reversibility preserving matrix  
$\tB_\e:= \tB_\e(\mu)$ associated, as in \eqref{matrix22}, with the self-adjoint and reversibility preserving operator $ {\cal B}_\e$ defined in \eqref{cBe}, with respect to the basis 
$\mathcal{F} $ of $ {\mathcal V}_{\mu,\e} $ 
in \eqref{basisF},  expands as 
\begin{align}\label{expB1}
 \mathtt{B}_\e = \begin{pmatrix} 
 \te_{11} \e^2+\zeta_\tth \mu^2+r_1(\e^3,\mu\e^3) & \im r_2(\mu\e^2) & \vline & \tf_{11} \e + r_3(\e^3, \mu\e^2) & \im r_4(\mu\e^3) \\
 -\im r_2(\mu\e^2) & \zeta_\tth \mu^2 & \vline & \im r_6(\mu\e) & 0 \\
 \hline
 \tf_{11} \e +r_3(\e^3, \mu\e^2) & -\im r_6(\mu\e) & \vline & 1+r_8(\e^2,\mu\e^2) & \im r_9(\mu\e^2)\\
 -\im r_4(\mu\e^3) & 0 & \vline & -\im r_9(\mu\e^2) & 0 \\
 \end{pmatrix}+\cO(\mu^2\e, \mu^3) \, ,
\end{align}
where 
$\te_{11}$, $\tf_{11}$ are defined  respectively in  \eqref{te11},  and 
\begin{equation} \label{zetah} \zeta_\tth := \tfrac{1}{8}\ch\gamma_\tth^2 \, .  \end{equation} 
\end{lem}
\begin{proof}
We expand the matrix $ \tB_\e(\mu) $ as 
\begin{equation}\label{TaylorexpBemu}
\tB_\e(\mu) = \tB_\e(0) + \mu (\pa_\mu \tB_\e)(0) + \frac{\mu^2 }{2} (\pa_\mu^2 \tB_0)(0) + \cO(\mu^2\e,\mu^3) \, . 
 \end{equation}
{\bf The matrix $\tB_\e(0)$.} The main result of this long paragraph 
is to prove that  the matrix $\tB_\e(0)$ has the expansion \eqref{Bsoloeps}. 
The matrix $\tB_\e(0)$ is real, because 
the operator $ {\cal B}_{\e}$ is real and the basis $ \{ f_k^\pm(0,\e) \}_{k=0,1}$ is  real.
Consequently, by \eqref{revprop},  its  matrix elements $(\tB_\e(0))_{i,j}$ are real  whenever  $i+j$ is even and vanish for $i+j$ odd.
In addition $f^-_0(0,\e)  = \footnotesize  \vet{0}{1}$ by \eqref{nonzeroaverage}, 
and, by \eqref{cBe}, we get 
$ {\cal B}_{\epsilon} f^-_0(0, \e)  = 0 $, for any $ \epsilon $.  
We deduce that the self-adjoint matrix $ \tB_\e(0) $ has the form 
 \begin{equation}\label{tBe}\tB_\e(0) =
\left( 
{\cal B}_\epsilon  \, f^\sigma_k(0, \e) , \, f^{\sigma '}_{k'}(0, \e)
\right)_{k, k'=0,1, \sigma, \sigma'  = \pm} = 
\begin{pmatrix}
E_{11}(0,\e)  &   0 & \vline &  F_{11}(0,\e)  &  0 \\
0 &  E_{22}(0,\e) & \vline & 0  &  0  \\
\hline
 F_{11}(0,\e)  &  0 & \vline & G_{11}(0,\e) &   0 \\
 0 & 0 & \vline & 0 & 0 
\end{pmatrix}\, ,
\end{equation}
where $E_{11}(0,\e)$, $E_{22}(0,\e)$, $G_{11}(0,\e)$, $F_{11}(0,\e)$ are real.
  We claim that
$ E_{22}(0,\e) = 0 $ for any $ \e $. As a first step, following \cite{BMV1}, we prove that
\begin{equation}\label{Jordancondition}
\text{ either } \ E_{22}(0,\e)\equiv 0 \, ,  \qquad\text{ or }  
\ E_{11}(0,\e)\equiv 0  \equiv F_{11}(0,\e) \, .
\end{equation}
Indeed, by \eqref{genespace}, 
the operator $ \sL_{0,\e} \equiv  {\mathcal L}_{0,\e}$ 
possesses, for any sufficiently small $\e \neq 0$, 
 the eigenvalue $ 0 $
with a four  dimensional 
generalized  Kernel  $
 \cW_\e := \text{span} \{ U_1, \tilde U_2, U_3, U_4 \} $, spanned 
by  $ \e$-dependent 
vectors $ U_1, \tilde U_2, U_3, U_4 $. 
 By Lemma \ref{lem:Kato1} it 
 results that $ \cW_\e = {\mathcal V}_{0,\e} = \text{Rg}(P_{0,\e} )$
 and by \eqref{genespace} we have $ \sL_{0,\e}^2 = 0 $ on $ \mathcal{V}_{0,\e} $. 
 Thus the matrix 
 \begin{equation}\label{formaLep}
\tL_\e(0):=\tJ_4 \tB_\e(0) = 
\begin{pmatrix}
0 &  E_{22}(0,\e) & \vline & 0  &  0  \\
-E_{11}(0,\e)  &   0 & \vline &  -F_{11}(0,\e)  &  0 \\
\hline
 0 & 0 & \vline & 0 & 0 \\
  -F_{11}(0,\e)  &  0 & \vline & -G_{11}(0,\e) &   0 
\end{pmatrix}\, ,
\end{equation}
which represents $ \sL_{0,\e}:\mathcal{V}_{0,\e}\to\mathcal{V}_{0,\e}$, 
satisfies $ \tL^2_\epsilon(0) = 0 $, namely 
$$ 
\tL^2_\epsilon(0) = -\begin{pmatrix}
  (E_{11}E_{22})(0,\e) & 0 & \vline & (F_{11}E_{22})(0,\e)  &  0  \\
0  &   (E_{11}E_{22})(0,\e) & \vline &  0  &  0 \\
\hline
 0 & 0 & \vline & 0 & 0 \\
 0  &   (F_{11}E_{22})(0,\e) & \vline & 0 &   0 
\end{pmatrix} = 0 
 $$
which  implies  \eqref{Jordancondition}. We now prove that the matrix $\tB_\e(0)$ defined in \eqref{tBe} expands as
\begin{equation}\label{Bsoloeps}
\tB_\e(0)  = \begin{pmatrix}\te_{11}\e^2+ {r(\e^3)} & 0 & \vline & \tf_{11}\e + r(\e^3) & 0 \\ 
 0& 0 & \vline & 0 & 0 \\
 \hline 
 \tf_{11}\e +r(\e^3)  & 0 & \vline & 1+ r(\e^2) & 0 \\
 0 & 0 & \vline & 0 & 0\end{pmatrix}
\end{equation}
where  $\te_{11}$ and $\tf_{11}$ are in \eqref{ta2} and \eqref{td1}.
 We expand  the operator $ {\cal B}_\e$ in \eqref{cBe} as
\begin{equation}\label{pezziB}
\begin{aligned}
&{\cal B}_\e= {\cal B}_0+\e {\cal B}_1+ \e^2 {\cal B}_2 + \cO(\e^3) , 
\quad 
 {\cal B}_0 := \begin{bmatrix} 1 & -\ch \pa_x \\ \ch \pa_x & |D|\tanh(\tth |D|) \end{bmatrix} \, ,\\  &{\cal B}_1 := \begin{bmatrix} a_1(x) & -p_1(x)\pa_x \\ \pa_x\circ p_1(x) & 0 \end{bmatrix} \,,  \ \; {\cal B}_2 := \begin{bmatrix} a_2(x) & -p_2(x)\pa_x \\ \pa_x\circ p_2(x) & -\tf_2 \pa_x^2 \big(1-\tanh^2(\tth |D|)\big)  \end{bmatrix} \,, 
 \end{aligned}
\end{equation}
where the remainder term  $\cO(\e^3) \in \cL(Y, X)$,
 the functions $a_1$, $p_1$, $a_2$, $p_2$ are given in \eqref{pino1fd}-\eqref{aino2fd}
and, in view of \eqref{expfe}, $\tf_2 := \tfrac14\ch^{-2}(\ch^4-3)$.

\noindent {$ \bullet$  \it Expansion of $E_{11}(0,\e)=\te_{11}\e^2+r(\e^3)$. }
 By \eqref{exf41}  we split the real function $f_1^+(0,\e)$ as 
 \begin{equation}\label{pezzig1p}
 \begin{aligned}
& \qquad \qquad f_1^+(0,\e) = f_1^+ + \e f_{1_1}^+ + \e^2 f_{1_2}^+ + \cO(\e^3) \, , \\
&    f_1^+ = \vet{\ch^\frac12 \cos(x)}{\ch^{-\frac12}\sin(x)},\   f_{1_1}^+ := \vet{\alpha_\tth \cos(2x)}{\beta_\tth \sin(2x)} \, , 
\ 
f_{1_2}^+ := \vet{even_0(x)}{odd(x)}, 
\end{aligned}
 \end{equation}
 where both $f_{1_2}^+$ and $\cO(\e^3)$ are vectors in  $H^1(\bT)$.
 Since $ {\cal B}_0f_1^+=\cJ^{-1} \sL_{0,0}f_1^+= 0$, and both $ {\cal B}_0$, $ {\cal B}_1$ are self-adjoint real operators, it results 
\begin{align}
 E_{11}(0,\e)&=\molt{{\cal B}_\e f^+_1(0,\e)}{ f^+_1(0,\e)} \notag \\
 &= \e \molt{{\cal B}_1 f_1^+}{f_1^+} + \e^2  \left[ \molt{{\cal B}_2 f_1^+}{f_1^+}+2\molt{{\cal B}_1 f_1^+}{f_{1_1}^+} + \molt{{\cal B}_0f_{1_1}^+}{f_{1_1}^+} \right]+\cO(\e^3) \, .
 \label{expa0}
\end{align}
By  \eqref{pezziB} one has
\begin{align} \label{Bon1p}
{\cal B}_1f_1^+  = \vet{A_1(1+\cos(2x))}{B_1\sin(2x)},\ \  
 {\cal B}_2f_1^+=\vet{A_2 \cos(x)+ A_3 \cos(3x)}{B_2\sin(x)+B_3\sin(3x)}, \  
 \  {\cal B}_0f_{1_1}^+ =\vet{A_4 \cos(2x)}{B_4 \sin(2x)} \, , 
\end{align}
with
\begin{equation}\label{cuntaz}
\begin{aligned}
& A_1:=  \tfrac12(a_1^{[1]}\ch^\frac12 - p_1^{[1]}\ch^{-\frac12}), \qquad 
B_1:= -p_1^{[1]} \ch^{\frac12} \, , \\
&  A_2 := \ch^{\frac12} a_2^{[0]} - \ch^{-\frac12}  p_2^{[0]}+\tfrac12 \ch^{\frac12} a_2^{[2]} - \tfrac12 \ch^{-\frac12}p_2^{[2]} \, , \qquad A_4 :=\alpha_\tth  -2\beta_\tth \ch \, , \\  
&  B_2 :=-\ch^{\frac12}p_2^{[0]}-\tfrac12 \ch^{\frac12}p_2^{[2]} + {\ch^{-\frac12}} \tf_2(1-\ch^4) \, , \qquad 
  B_4 := -2\alpha_\tth \ch  + \displaystyle{\frac{4\ch^2}{1+\ch^4}}{\beta_\tth}\, . 
 \end{aligned}
 \end{equation}
By \eqref{Bon1p} and \eqref{pezzig1p}, we deduce  
\begin{equation} \label{ta2} 
E_{11}(0,\e) = \te_{11}\e^2 + r(\e^3) \, ,
 \quad \te_{11} := \frac12 \big(A_2\ch^\frac12+ B_2\ch^{-\frac12}+ 2\alpha_\tth A_1 +2 B_1 \beta_\tth +	\alpha_\tth A_4+\beta_\tth B_4 \big)\,. 
 \end{equation}
By \eqref{ta2}, \eqref{cuntaz}, \eqref{def:basiscoeff}, \eqref{pino1fd}-\eqref{aino2fd} we obtain \eqref{te11}.
Since $\te_{11}>0$ the  second alternative in \eqref{Jordancondition} is  ruled out, implying   $E_{22}(0,\e) \equiv 0$.\\
{ $ \bullet$  \it  Expansion of $G_{11} (0,\e)=1+r(\e^2)$. }
 By \eqref{exf43}  we split the real-valued function $f_0^+(0,\e)$ as
\begin{equation}\label{pezzig0p}
f_0^+(0,\e) = f_0^+ + \e f_{0_1}^+ + \e^2 f_{0_2}^+ + \cO(\e^3) \, , \  \  
f_0^+ = \vet10 \, ,
\   f_{0_1}^+:= \delta_\tth \vet{\ch^{\frac12}\cos(x)}{-\ch^{-\frac12}\sin(x)} \, , 
\   f_{0_2}^+:= \vet{even_0(x)}{odd(x)} \, .
\end{equation} 
Since, by \eqref{basestart} and  
\eqref{pezziB}, 
 ${\cal B}_0f_0^+= f_0^+$, 
 using that ${\cal B}_0$, ${\cal B}_1$ are self-adjoint real operators,  
 and $\|f_0^+\| = 1$, $ (f_0^+ , f_{0_1}^+ ) $, we have 
$ G_{11}(0,\e) =\molt{{\cal B}_\e f^+_0(0,\e)}{ f^+_0(0,\e)} 
 =1+ \e \molt{{\cal B}_1 f_0^+}{f_0^+}  
 +r(\e^2) $. 
By  \eqref{pezziB} and \eqref{pino1fd}-\eqref{aino2fd}  one has
\begin{equation}\label{Bon0p}
{\cal B}_1f_0^+ = \vet{a_1^{[1]}\cos(x)}{-p_1^{[1]} \sin(x)} 
\end{equation}
and, by \eqref{pezzig0p}, we deduce 
$ G_{11} (0,\e) = 1+ r(\e^2 ) $.  \\
{$ \bullet$  \it  Expansion of $F_{11}(0,\e)=\tf_{11}\e+r(\e^3)$.}
By  \eqref{pezziB},  \eqref{pezzig1p}, \eqref{pezzig0p}, using that 
${\cal B}_0, {\cal B}_1$ are self-adjoint and real, 
and  ${\cal B}_0 f_1^+ = 0$, ${\cal B}_0 f_0^+ = f_0^+$, we obtain 
  \begin{align}
 F_{11}(0,\e)  & = \e \left[ \molt{{\cal B}_1 f_1^+}{f_0^+} + \molt{f_{1_1}^+}{ f_0^+ }\right] 
 \notag \\ 
 \notag
 & \quad +  \e^2  \big[ \molt{{\cal B}_2 f_1^+}{f_0^+}+\molt{{\cal B}_1 f_1^+}{f_{0_1}^+} +\molt{{\cal B}_1 f_0^+}{f_{1_1}^+} +\molt{f_{1_2}^+}{f_0^+} + \molt{{\cal B}_0f_{1_1}^+}{f_{0_1}^+}  \big]+r(\e^3) \,  .
 \end{align}
By \eqref{pezzig1p}, \eqref{Bon1p}, \eqref{cuntaz}, \eqref{pezzig0p}, \eqref{Bon0p}, all these scalar products vanish but the first one, and then 
\begin{equation} \label{td1}
 F_{11}(0,\e)=\tf_{11}\e+r(\e^3) \, ,\quad 
 \tf_{11} := A_1 = \tfrac12(a_1^{[1]}\ch^{\frac12}-p_1^{[1]}\ch^{-\frac12}) \, , 
 \end{equation}
 which,  by substituting the expressions of 
 $ a_1^{[1]} $, $ p_1^{[1]} $ in Lemma \ref{lem:pa.exp}, gives the expression in \eqref{te11}.
 
 The expansion \eqref{Bsoloeps} in proved. 
\\[1mm]
{\bf Linear terms in $ \mu $.}
We now compute the terms  of $\tB_\e(\mu)$  that are linear in $\mu$. It results 
\begin{equation}\label{MatrixX}
\pa_\mu \tB_\e(0) = X + X^* 
\qquad \text{where} \qquad X :=
 \big( {\cal B}_\e f_k^\sigma(0,\e), (\pa_\mu f^{\sigma'}_{k'})(0,\e) \big)_{k,k'=0,1,  \sigma,\sigma'=\pm} \, .  
 \end{equation}
We now prove that 
\begin{equation}\label{matX}
   X = \begin{pmatrix} 
 \cO(\e^3) & 0 & \vline & \cO(\e^2) & 0 \\ 
 \cO(\e^2)  & 0 & \vline & \cO( \e) & 0 \\
 \hline
 \cO(\e^3) & 0 & \vline & \cO(\e^2) & 0 \\
 \cO(\e^3) & 0 & \vline & \cO(\e^2) & 0
 \end{pmatrix}.
 \end{equation}
The matrix $ \tL_\e (0) $ in \eqref{formaLep} where $E_{22}(0,\e)=0$, represents the action 
of the operator $ \sL_{0,\e}:\mathcal{V}_{0,\e}\to\mathcal{V}_{0,\e}$ in the basis
$ \{ f^{\sigma}_k (0,\e) \} $  and then we deduce that  
 $ \sL_{0,\e} f_1^-(0,\e) = 0 $,  $ \sL_{0,\e} f_0^-(0,\e) = 0 $. 
Thus 
also $ {\cal B}_\e  f_1^-(0,\e) = 0 $, $ {\cal B}_\e  f_0^-(0,\e)  = 0 $, 
and the second and the fourth column of the matrix $X$ in \eqref{matX} are zero. 
To compute the other two columns we use the expansion of the derivatives. In view of \eqref{exf41}-\eqref{exf44} and by denoting with a dot the derivative w.r.t. $\mu$, one has
\begin{equation}\label{reuse}
\begin{aligned}
 &\dot f^{+}_{1}(0,\e) = \frac\im4 \gamma_\tth	\vet{\ch^{\frac12}\sin(x)}{\ch^{-\frac12}\cos(x)}+\im\e \vet{odd(x)}{even(x)}+\cO(\e^2) 	\, , \quad  \dot f^{+}_{0}(0,\e) = \im\e \vet{odd(x)}{even_0(x)}+\cO(\e^2) \, ,\\  
 &\dot f^{-}_{1}(0,\e) = \frac\im4\gamma_\tth \vet{\ch^{\frac12}\cos(x)}{-\ch^{-\frac12}\sin(x)}+\im\e \vet{even(x)}{odd(x)}+\cO(\e^2) \, ,\ \ \dot f^{-}_{0}(0,\e) =\im \e \vet{even_0(x)}{odd(x)} +\cO(\e^2) \, .
\end{aligned}
\end{equation} 
In view of \eqref{PoissonTensor}, \eqref{exf41}-\eqref{exf44},  \eqref{formaLep},
\eqref{nonzeroaverage},  \eqref{ta2},\eqref{td1},  
and since $ {\cal B}_\e f_k^\sigma(0,\e)=-\cJ \sL_\e f_k^\sigma(0,\e) $, 
we have
\begin{equation}\label{reuse2}
\begin{aligned}
{\cal B}_\e f_1^+(0,\e) &= E_{11}(0,\e) \,  \cJ f_1^-(0,\e) + F_{11}(0,\e) \, \cJ f_0^-
=\e \vet{\tf_{11}}{0}+ \e^2 \te_{11} \vet{\ch^{-\frac12}\cos (x)}{\ch^{\frac12}\sin (x)} + \cO (\e^3)\, , \\
{\cal B}_\e f_0^+(0,\e) &=F_{11}(0,\e)\, \cJ f_1^-(0,\e) + G_{11}(0,\e)\,\cJ f_0^-
= \vet{1}{0}+\e \tf_{11} \vet{\ch^{-\frac12}\cos(2x)}{\ch^{\frac12}\sin(2x)}+
\cO(\e^2)\, . 
\end{aligned}
\end{equation}
We deduce
\eqref{matX}  by  \eqref{reuse} and \eqref{reuse2}.
\\[1mm]
{\bf Quadratic terms in $ \mu $.}
By denoting with a double dot the double derivative w.r.t. $\mu$, we have 
\begin{equation}\label{dersec}
\pa_\mu^2 \tB_0(0) = \molt{{\cal B}_0 f_k^\sigma}{\ddot f_{k'}^{\sigma'}(0,0)}+\molt{\ddot f_{k}^{\sigma}(0,0)}{{\cal B}_0 f_{k}^{\sigma'}}+2\molt{{\cal B}_0\dot f_k^\sigma(0,0)}{\dot f_{k'}^{\sigma'}(0,0)}=:Y+Y^*+2Z \, .
\end{equation}
We claim that $Y = 0 $. Indeed, its first, second and fourth column are zero, since ${\cal B}_0f_k^\sigma=0$ for $f_k^\sigma \in \{ f_1^+,f_1^-,f_0^- \} $.
 The third column is also zero 
by noting  that $ {\cal B}_0 f_0^+ = f_0^+ $ and
$$
\ddot f_{1}^{+}(0,0) = \vet{even_0(x)+\im odd(x)}{odd(x)  +\im  even_0(x)}, \ \ \ddot f_{1}^{-}(0,0) = \vet{odd(x)  +\im  even_0(x)}{even_0(x)+\im odd(x)}, \ \ \ddot f_{0}^{+}(0,0)=\ddot 
f_{0}^{-}(0,0)=0 \, .
$$
We claim that
\begin{align}\label{matZ}
 Z = \molt{{\cal B}_0\dot f_k^\sigma(0,0)}{\dot f_{k'}^{\sigma'}(0,0)}_{\substack{k,k'=0,1,\\\sigma,\sigma'=\pm}} = \begin{pmatrix} 
 \zeta_\tth & 0 & \vline & 0 & 0 \\
 0 & \zeta_\tth & \vline & 0 & 0 \\
 \hline
 0 & 0 & \vline & 0 & 0 \\
 0 & 0 & \vline & 0 & 0 \\
 \end{pmatrix} \, ,
\end{align}
with $\zeta_\tth$ as in \eqref{zetah}.
Indeed,  by \eqref{reuse}, we have  $\dot f^+_0(0,0)=\dot f^-_0(0,0)= 0$. 
Therefore the last two columns of $Z$, and by self-adjointness  the last two rows, are zero. 
By \eqref{pezziB},  \eqref{reuse}
we obtain the matrix \eqref{matZ} with 
$$ 
 \zeta_\tth := \molt{{\cal B}_0  \dot f^+_1(0,0)}{ \dot f^+_1(0,0)} = \molt{{\cal B}_0 \dot f^-_1(0,0)}{\dot f^-_1(0,0)} = \tfrac{1}{8}\ch\gamma_\tth^2 \, . 
$$
In conclusion \eqref{TaylorexpBemu}, \eqref{MatrixX}, \eqref{matX}, \eqref{dersec}, the fact that $Y=0$ and \eqref{matZ} imply \eqref{expB1}, using 
also the selfadjointness of 
$\tB_\e$
 and 
\eqref{revprop}. 
\end{proof}

We now consider $ {\cal B}^\flat $.

\begin{lem}\label{lem:B2}
{\bf (Expansion of $\mathtt{B}^\flat$)} 
The self-adjoint and reversibility-preserving matrix 
$\tB^\flat$ associated, as in \eqref{matrix22}, to the self-adjoint 
and reversibility-preserving operator $  {\cal B}^\flat$, defined in \eqref{cBflat}, with respect to the basis $\mathcal{F}$  of $ {\mathcal V}_{\mu,\e} $
in \eqref{basisF}, admits  the expansion
\begin{equation}\label{EBflat}
\mathtt{B}^\flat=   \begingroup 
\setlength\arraycolsep{2pt} \begin{pmatrix} 
-\frac{\mu^2}{4}\bbh   & \im (\frac{\mu}{2} \te_{12} +r_2(\mu \e^2)) & \vline & 0 & 0 \\         
-\im (\frac{\mu}{2} \te_{12} +r_2(\mu \e^2)) & -\frac{\mu^2}{4} \bbh & \vline & 
\im r_6(\mu \e) & 0 \\ \hline
0 & - \im r_6(\mu \e)   & \vline & 0 & 0\\
0 & 0 & \vline &  0 & \mu \tanh(\tth \mu)
\end{pmatrix} \endgroup +\cO(\mu^2\e,\mu^3)
\end{equation}
where $\te_{12}$ is defined in \eqref{te12}\,  and
\begin{equation}
\bbh 
 :=   \gamma_\tth \ch + \ch^{-1}\tth (1-\ch^4) (\gamma_\tth - 2(1-\ch^2\tth)) \, 
 \label{Bflatdiag} \, . 
\end{equation}
\end{lem}
\begin{proof}
We have to compute the 
expansion of the matrix entries $ ( {\cal B}^\flat f^\sigma_k(\mu,\e), f^{\sigma'}_{k'}(\mu,\e)) $. 
First, by \eqref{exf44},  \eqref{cBflat} and since $\ttf=O(\e^2)$ (cfr. \eqref{expfe}) we have 
  \begin{align*}
 {\cal B}^\flat f^-_0(\mu,\e) = \vet{0}{\mu \tanh\big(\tth \mu\big)} + \vet{0}{\cO(\mu^2\e)} \, .
\end{align*}
Hence, by \eqref{exf41}-\eqref{exf44},  
the entries of the last column (and row) of $\mathtt{B}^\flat$ are
\begin{align*}
& \big({\cal B}^\flat f^-_0(\mu,\e), f^+_1(\mu,\e) \big)  =\cO(\mu^2 \e) \ , \quad  \big({\cal B}^\flat f^-_0(\mu,\e), f^-_1(\mu,\e) \big)  = \mu \tanh(\tth \mu) \cO(\e^2) + \cO(\mu^2 \e^2 ) = 
\cO(\mu^2 \e^2) \notag \\
& \big({\cal B}^\flat f^-_0(\mu,\e), f^+_0(\mu,\e) \big)  =  \cO(\mu^2\e, \mu^3)  \ , 
\quad  \big({\cal B}^\flat f^-_0(\mu,\e), f^-_0(\mu,\e) \big)  = \mu \tanh(\tth\mu) + \cO(\mu^2\e) \, ,  
\end{align*}
in agreement with \eqref{EBflat}.

In order to compute the other matrix entries we 
expand  $  {\cal B}^\flat $ in 
 \eqref{cBflat}
 at $\mu = 0$, obtaining
 \begin{equation}
 \label{B1bemolle}
 \begin{aligned}
&   {\cal B}^\flat  = \mu {\cal B}^\flat_1{(0)} +
\mu {\cal R}^\flat(\e) +
 \mu^2 {\cal B}^\flat_2 + {\cO(\mu^2	\e,\mu^3)\, ,\quad \text{where} } \\
  & {\cal B}^\flat_1(0) := \Big[\tth  D \big(1-\tanh^2(\tth|D|)\big) + \sgn(D)\tanh(\tth|D|)\Big] \Pi_{\II} \,  , \quad \Pi_{\II} := \begin{bmatrix} 0 & 0 \\ 0 & \uno \end{bmatrix} \, , \\
& {\cal R}^\flat(\e) := \cO(\e^2) \Pi_{\II}  \, , \qquad   {\cal B}^\flat_2:={\Big[\tth  \big(1-\tanh^2(\tth|D| )\big) \big( 1-\tth \tanh(\tth|D|)   |D| \big)\Big]\Pi_{\II}\,  . } 
 \end{aligned}
 \end{equation}
We note that 
\begin{equation}\label{restiniR}
\mu\big( {\cal R}^\flat(\e) f^{\sigma}_k (\mu, \e) , f^{\sigma'}_{k'} (\mu, \e) \big)  =  \mu\big( {\cal R}^\flat f^{\sigma}_k (0, \e) , f^{\sigma'}_{k'} (0, \e) \big) + \cO(\mu^2\e^2) = \begin{cases} 
\cO(\mu^2\e^2) & \mbox{if }\sigma=\sigma'\, , \\ \cO(\mu\e^2) & \mbox{if }\sigma\neq\sigma'\, . 
\end{cases}
\end{equation}
Indeed, if $\sigma=\sigma'$,  $ \big( {\cal R}^\flat f^{\sigma}_k (0, \e) , f^{\sigma'}_{k'} (0, \e) \big)$ is real by \eqref{revprop}, but 
purely imaginary\footnote{ \label{apuim}
An operator $\mathcal{A}$ is \emph{purely imaginary} if $\bar{\mathcal{A}}=-\mathcal{A}$. A purely imaginary operator sends real functions into purely imaginary ones.} too, since  the operator ${\cal R}^\flat $
 is purely imaginary (as ${\cal B}^\flat  $ is) and the basis $ \{ f_k^\pm(0,\e) \}_{k=0,1}$ is  real. The terms \eqref{restiniR} contribute to $ r_2 (\mu \e^2) $ and $ r_6 (\e \mu )$ in  \eqref{EBflat}. 
 
 Next we compute the other scalar products.
By \eqref{exf41}, \eqref{B1bemolle}, and the identities 
$ \sgn(D) \sin(kx)  = - \im\cos(kx) $ and  
$ \sgn(D)\cos(kx) = \im \sin(kx) $ for any $ k \in \bN $,  
 we have 
 \begin{equation*}
\mu {\cal B}^\flat_1 (0) f^+_1(\mu,\e)  = \footnotesize
  -\im\mu \bem_{1}\vet{0}{\cos(x)} - \frac{\mu^2}{4} \gamma_\tth \bem_{1}\vet{0}{\sin(x)} - \im\mu\e \bem_{2} \vet{0}{\cos(2x)} +\im \cO(\mu\e^2) \vet{0}{even_0(x)} +\cO(\mu^2\e,\mu^3)
 \end{equation*}
 where 
 \begin{equation} \label{bem12}
 \begin{aligned}
& \bem_{1}:= \ch^{-\frac12} (\ch^2 +(1-\ch^4)\tth)   \\ 
& \bem_{2}:= \beta_\tth \Big(  \tanh (2\tth ) +  2 \tth  (1-\tanh^2 (2\tth )) \Big)=\beta_\tth \Big(\frac{2\ch^2}{1+\ch^4}+2\tth\big(1-\frac{4\ch^4}{(1+\ch^4)^2}\big)\Big) \, .
 \end{aligned}
 \end{equation}
 Similarly $\mu^2 {\cal B}^\flat_2 f^+_1(\mu,\e)  = \mu^2  \bem_{3} \footnotesize{\vet{0}{\sin(x)}} + \cO(\mu^2\e,\mu^3) $, where 
 \begin{equation}\label{bem3}
\bem_{3}:= \tth \big(1-\tanh^2(\tth)  \big) \big(1-\tanh(\tth ) \tth \big) \ch^{-\frac12}= \tth (1-\ch^4)(1-\ch^2 \tth ) \ch^{- \frac12}  \, . 
 \end{equation}
Analogously, using \eqref{exf42}, 
$$
 \footnotesize
\mu {\cal B}^\flat_1 (0) f^-_1(\mu,\e) = \im\mu \bem_{1} \vet{0}{\sin(x)} -\frac{\mu^2}{4}\gamma_\tth \bem_{1} \vet{0}{\cos(x)}+\im\mu\e \bem_{3}\vet{0}{\sin(2x)}+\im  \cO(\mu\e^2) \vet{0}{ odd(x)}+\cO(\mu^2\e,\mu^3) \, ,
$$
and $\mu^2 {\cal B}^\flat_2 f^-_1(\mu,\e)  = \mu^2  \bem_{3}\footnotesize{\vet{0}{\cos(x)}} + \cO(\mu^2\e,\mu^3) $, with $\bem_j$, $j=1,2,3$, defined in \eqref{bem12} and \eqref{bem3}. 
 In addition, by  \eqref{exf43}-\eqref{exf44},  we get that 
$$
 \mu{\cal B}^\flat_1 (0) f^+_0(\mu,\e) = \im\mu\e \delta_\tth \bem_{1} \vet{0}{\cos(x)}+ \im \cO(\mu\e^2) \vet{0}{even_0(x)}+ \cO(\mu^2\e) \, ,
 \ \ 
  \mu^2{\cal B}^\flat_2 f^+_0(\mu,\e) = \vet{0}{\cO(\mu^2\e)}\,
$$
with $\bem_1 $ in \eqref{bem12}.
By taking the scalar products of the above expansions of 
 $ {\cal B}^\flat f^\sigma_k (\mu,\e) $
 with the functions $f^{\sigma'}_{k'}(\mu,\e) $  expanded as in 
 \eqref{exf41}-\eqref{exf44} we obtain that 
 (recall that the scalar product is conjugate-linear in the second component)
$$
\begin{aligned}
& \big( \mu {\cal B}^\flat_1(0) f^+_1(\mu,\e), f^+_1(\mu,\e) \big) \, , \ 
\big( \mu {\cal B}^\flat_1(0) f^-_1(\mu,\e), f^-_1(\mu,\e) \big)  = 
{-\frac{\mu^2}{4} \gamma_\tth \bem_1 \ch^{-\frac12}}+ \cO(\mu^2\e,\mu^3) \\
& 
\big( \mu^2 {\cal B}^\flat_2 f^+_1(\mu,\e), f^+_1(\mu,\e) \big) \, , \  
 \big( \mu^2 {\cal B}^\flat_2 f^-_1(\mu,\e), f^-_1(\mu,\e) \big)
 ={\frac{\mu^2}{2} \bem_3\ch^{-\frac12}} + \cO(\mu^2 \e, \mu^3)
\end{aligned} 
$$
and, recalling \eqref{B1bemolle}, \eqref{bem12}, \eqref{bem3}, we deduce 
the expansion of the entries $(1,1)$ and $(2,2)$ of the matrix $\tB^\flat$ 
in \eqref{EBflat}
 with $ \bbh = \ch^{-\frac12} 
 (\gamma_\tth \bem_1 - 2 \bem_3) $ in  \eqref{Bflatdiag}. Moreover
$$
  \big( \mu {\cal B}^\flat_1(0) f^-_1(\mu,\e), f^+_1(\mu,\e) \big)
   = 
{\im \frac{\mu}{2} \te_{12}} + \cO(\mu \e^2,\mu^2 \e, \mu^3) \, , 
\ \ 
\big( \mu^2 {\cal B}^\flat_2 f^-_1(\mu,\e), f^+_1(\mu,\e) \big)
   = 
\cO(\mu^3, \mu^2 \e) \, , 
$$
where $ \te_{12}:= \bem_1 \ch^{-\frac12}$ is equal to  \eqref{te12}. Finally we obtain
 \begin{equation*}
 \begin{aligned}
&  
 \big( \mu({\cal B}^\flat_1(0)+\mu {\cal B}^\flat_2) f^-_1(\mu,\e), f^+_0(\mu,\e) \big)
   =\cO(\mu \e, \mu^3)   \\
&   ( \mu ({\cal B}^\flat_1(0) +\mu {\cal B}^\flat_2) f^+_1(\mu,\e), f^+_0(\mu,\e))
   =\cO(\mu^3, \mu^2 \e) \, , \\
   & \big( \mu ({\cal B}^\flat_1(0) +\mu {\cal B}^\flat_2) f^+_0(\mu,\e), f^+_0(\mu,\e) \big)  = \cO(\mu^2 \e^2) \, .
   \end{aligned}
\end{equation*}
The expansion \eqref{EBflat} is proved.
\end{proof}


Finally we consider ${\cal B}^\sharp$. 

\begin{lem}\label{lem:B3}
{\bf (Expansion of  $\mathtt{B}^\sharp$)}
The self-adjoint and reversibility-preserving matrix 
$\tB^\sharp$ associated, as in \eqref{matrix22}, to the self-adjoint 
and reversibility-preserving operators $ {\cal B}^\sharp$, defined in  \eqref{cBsharp}, with respect to the basis $\mathcal{F}$  of $ {\mathcal V}_{\mu,\e} $
in \eqref{basisF}, admits the expansion
\begin{equation}\label{tBsharp}
 \mathtt{B}^\sharp = \begin{pmatrix} 
 0 & \im  r_2(\mu\e^2) & \vline & 0 & \im   \mu\e \ch^{-\frac12}+\im r_4(\mu\e^2) \\ 
 -\im r_2(\mu\e^2)& 0 & \vline & -\im r_6(\mu\e) & 0 \\
 \hline
 0 & \im r_6(\mu\e) & \vline & 0 & - \im r_9(\mu\e^2) \\
 -\im  \mu\e \ch^{-\frac12}-\im r_4(\mu\e^2)  & 0 & \vline & \im r_9(\mu\e^2) & 0
 \end{pmatrix}+\cO(\mu^2\e) \, .
\end{equation}
\end{lem}
\begin{proof}
Since ${\cal B}^\sharp  = -\im \mu p_\e \cJ$ and  
$p_\e=\cO(\e)$
  by \eqref{SN1},   we have the expansion
\begin{equation}\label{TaylorexpBsharpemu}
 \big( {\cal B}^\sharp f_k^\sigma(\mu,\e), f_{k'}^{\sigma'}(\mu,\e) \big) = 
\big({\cal B}^\sharp f_k^\sigma(0,\e), f_{k'}^{\sigma'}(0,\e) \big) + \cO(\mu^2\e) \, .
\end{equation}
The matrix entries
$ ( {\cal B}^\sharp f^{\sigma}_k (0, \e) , f^{\sigma}_{k'} (0, \e) )   $, $  k, k' =  0,1 $, 
$ \sigma = \{ \pm \} $ are zero,  
because they are simultaneously real by \eqref{revprop}, and  
purely imaginary, being  the operator ${\cal B}^\sharp$
purely imaginary 
  and the basis $ \{ f_k^\pm(0,\e) \}_{k=0,1}$   real. 
 Hence $\tB^\sharp$ has the form 
 \begin{equation}\label{formatBsharp}
 \mathtt{B}^\sharp = \begin{pmatrix} 
 0 & \im\beta & \vline & 0 & \im\delta \\ 
 -\im\beta & 0 & \vline & -\im\gamma & 0 \\
 \hline
 0 & \im\gamma & \vline & 0 & \im\eta \\
 -\im\delta & 0 & \vline & -\im\eta & 0
 \end{pmatrix}+\cO(\mu^2\e) 
 \quad \text{where} \quad 
 \left\{\begin{matrix}\molt{{\cal B}^\sharp f_1^-(0,\e)}{f_1^+(0,\e)}=:\im\beta \, , \\   \molt{{\cal B}^\sharp f_1^-(0,\e)}{f_0^+(0,\e)}=:\im\gamma \, , \\
\molt{{\cal B}^\sharp f_0^-(0,\e)}{f_1^+(0,\e)}=:\im\delta \, , 
\\ \molt{{\cal B}^\sharp f_0^-(0,\e)}{f_0^+(0,\e)}=:\im\eta \, ,\end{matrix}\right.
 \end{equation}
and $\alpha $, $ \beta $, $ \gamma $, $ \delta$ are real numbers.
As ${\cal B}^\sharp  = \cO(\mu \e)$ in $\cL(Y)$, we deduce  
that $ \gamma =r( \mu\e ) $. 
Let us compute the expansion of $\beta$, $\delta $ and $\eta$.  By \eqref{pino1fd} 
and \eqref{PoissonTensor}
we write the operator $ {\cal B}^\sharp$ in \eqref{cBsharp}  as
\begin{equation}\label{pezziBsharp}
{\cal B}^\sharp = \im\mu \e{\cal B}_1^\sharp+  \cO(\mu\e^2) \, ,
\quad {\cal B}_1^\sharp :=   2 \ch^{-1} \cos (x)  \begin{bmatrix} 0 &  \uno \\ - \uno & 0 \end{bmatrix} \, , 
\end{equation}
with $\cO(\mu \e^2) \in \cL(Y)$. In view of \eqref{exf41}-\eqref{exf44},
$f_1^\pm(0,\e) = f_1^\pm + \cO(\e)$, $f_0^+(0,\e)=f_0^+  +\cO(\e)$, 
$f_0^-(0,\e) = \footnotesize \vet{0}{1}$,  where $ f_k^\sigma $ are in 
 \eqref{funperturbed}. By  \eqref{pezziBsharp} we have 
$ \footnotesize  {\cal B}_1^\sharp f_1^- =  \vet{\ch^{-\frac32}(1+\cos(2x))}{\ch^{-\frac12}\sin(2x)} $,
$ \footnotesize  {\cal B}_1^\sharp f_0^- =  \vet{2\ch ^{-1} \cos(x)}{0} $
 and then 
\begin{equation*} 
\begin{aligned}
& \beta = \mu\e \molt{{\cal B}_1^\sharp f_1^-}{f_1^+} +r(\mu\e^2) = r(\mu\e^2) \, , \\ 
&  \delta = \mu\e \molt{{\cal B}_1^\sharp f_0^-}{f_1^+} +r(\mu\e^2) 
 =  \mu\e \ch^{-\frac12}+r(\mu\e^2) \, , \\
 &   \eta =  \mu\e \molt{{\cal B}_1^\sharp f_0^-}{f_0^+} +r(\mu\e^2) = r(\mu\e^2)  \, . 
 \end{aligned}
\end{equation*}
This proves  \eqref{tBsharp}.
 \end{proof}

 Lemmata \ref{lem:B1}, \ref{lem:B2}, \ref{lem:B3}  imply \eqref{splitB}
 where the matrix $ E $ has the form \eqref{BinG1} and
$$
\te_{22}:=2( \textbf{b}_{\mathtt{h}} - 4 \zeta_\tth) = 2\gamma_\tth \ch + 2\ch^{-1}\tth (1-\ch^4) (\gamma_\tth - 2(1-\ch^2\tth))  - \ch\gamma_\tth^2  \, ,
$$
 with $ \textbf{b}_{\mathtt{h}} $ in \eqref{Bflatdiag} and $ \zeta_\tth $ in \eqref{zetah}.
 The term $\te_{22}$ has the expansion in \eqref{te22}.   Moreover
\begin{align}
 & \label{BinG2in} G := G(\mu,\e) = 
 \begin{pmatrix} 
1+r_8(\e^2,\mu^2\e, \mu^3) &   - \im r_9(\mu\e^2,\mu^2\e,\mu^3) \\
  \im  r_9(\mu\e^2, \mu^2\e,\mu^3)  &\mu\tanh(\tth\mu)+ r_{10}(\mu^2\e,\mu^3) 
 \end{pmatrix} \\
 & \label{BinG3in}
 F := F(\mu,\e) = 
 \begin{pmatrix} 
\tf_{11}\e+ r_3(\e^3,\mu\e^2,\mu^2\e,\mu^3) & \im 
 \mu\e \ch^{-\frac12}   +\im  r_4({\mu\e^2}, \mu^2 \e, \mu^3)  \\
  \im   r_6(\mu\e, \mu^3)    & r_7(\mu^2\e,\mu^3) 
 \end{pmatrix} \, . 
 \end{align}
 In order to deduce the expansion 
 \eqref{BinG2}-\eqref{BinG3}  of the matrices $ F, G $ we exploit further information for 
\begin{equation}\label{Lmu0}
\sL_{\mu,0} := 
 \cJ  {\cal B}_{\mu,0} \, , \quad 
{\cal B}_{\mu,0} := \begin{bmatrix} 1 & - \ch\pa_x   \\
\ch   \pa_x  & |D+ \mu| \, \tanh\big(\tth|D+\mu|\big) \end{bmatrix}
 \, . 
\end{equation}
 We have 
\begin{lem}\label{lem:fina}
At $ \e = 0 $ the matrices are 
$  F (\mu,0) = 0 $ and $  G (\mu,0) =   \begin{pmatrix}
1 & 0 \\
0 & \mu \tanh ( \tth \mu )
\end{pmatrix} $. 
\end{lem}

\begin{proof}
By Lemma \ref{lem:A5} and \eqref{Lmu0} we have 
$ {\cal B}_{\mu,0} f_0^+ (\mu,0) = f_0^+ $ and $ {\cal B}_{\mu,0} f_0^- (\mu,0) = 
\mu \tanh (\tth \mu) f_0^- $, for any $ \mu $.   Then the lemma follows recalling 
\eqref{matrix22} and the fact that  
$f_1^+(\mu,0)$ and $f_1^-(\mu,0)$ have zero space average
by Lemma \ref{lem:A5}.
\end{proof}

In view of Lemma 
\ref{lem:fina} we deduce that the matrices
\eqref{BinG2in} and  \eqref{BinG3in} have the form \eqref{BinG2} and  \eqref{BinG3}.
This completes the proof of Proposition \ref{BexpG}.

\smallskip

We now show
  that  the constant $\te_{22}$ in \eqref{te22} is positive for any depth $\tth >0 $.  
\begin{lem}\label{nondegcond} For any $ \tth > 0 $ the term $\te_{22} $ in \eqref{te22} is positive, 
$\te_{22} \to 0 $  as $\tth\to 0^+$ and $\te_{22} \to 1 $  as $\tth\to +\infty$. As a consequence for any  $\tth_0 >0 $ the term $\te_{22}$ is bounded from below uniformly in $\tth>\tth_0$.
\end{lem}
\begin{proof}
The quantity  $ z := \ch^2 = \tanh(\tth) $  is in $ (0,1)$ for any $ \tth > 0 $. Then the quadratic polynomial 
$ (0, + \infty) \ni \tth \mapsto (1-z^2)(1+3z^2) \tth^2+2 z(z^2-1) \tth+z^2  $ 
is positive because
its discriminant $- 4z^4(1-z^2) $ is  negative as $ 0<z^2<1$. 
The limits for $ \tth\to 0^+$ and $\tth\to +\infty$ follow by inspection.
\end{proof}

\section{Block-decoupling and emergence of the Whitham-Benjamin function} 
\label{sec:block}

In this section we block-decouple 
the $ 4 \times 4 $ Hamiltonian matrix $\tL_{\mu,\e} = \tJ_4 \tB_{\mu,\e} $ obtained 
in Proposition \ref{BexpG}.  

We first perform a singular symplectic and reversibility-preserving
change of coordinates. 

\begin{lem}\label{decoupling1prep}
{{\bf (Singular symplectic rescaling)}}
 The conjugation of the Hamiltonian and reversible matrix $\tL_{\mu,\e} = \tJ_4 \tB_{\mu,\e} $ 
obtained in Proposition \ref{BexpG} through the symplectic and reversibility-preserving 
 $ 4 \times 4 $-matrix 
 \begin{equation}\label{Ychangeprep}
 Y := \begin{pmatrix} Q & 0 \\ 0 & Q \end{pmatrix}
 \quad  \text{with} \quad 
 Q:=\begin{pmatrix} \mu^{\frac12} & 0 \\ 0 & \mu^{-\frac12}\end{pmatrix} \, ,  \ \ \mu > 0 \, ,
\end{equation} 
yields the Hamiltonian and reversible matrix 
\begin{align}
&\tL_{\mu,\e}^{(1)} := Y^{-1} \tL_{\mu,\e} Y = \tJ_4\tB^{(1)}_{\mu,\e}  =
\begin{pmatrix}   \tJ_2 E^{(1)}  & \tJ_2  F^{(1)}  \\  
\tJ_2 [F^{(1)}]^*  &  \tJ_2 G^{(1)}  \end{pmatrix}  \label{LinHprep}
 \end{align}
where $ \tB_{\mu,\e}^{(1)} $ is a self-adjoint and reversibility-preserving 
 $ 4 \times 4$  matrix 
  \begin{equation}\label{splitB1prep}
\tB_{\mu,\e}^{(1)} =
\begin{pmatrix} 
E^{(1)} & F^{(1)} \\ 
[F^{(1)}]^* & G^{(1)} 
\end{pmatrix}, \quad 
E^{(1)}  = [E^{(1)}]^* \, , \ G^{(1)} = [G^{(1)}]^* \, , 
\end{equation} 
 where the $ 2 \times 2 $  reversibility-preserving matrices $E^{(1)} $, $ G^{(1)} $  and $ F^{(1)}$ extend analytically at $\mu  =0$ with the following expansion
 \begin{align}\label{BinH1prep}
& E^{(1)} = 
\begingroup
\setlength\arraycolsep{2.9pt}
\begin{pmatrix} 
  \te_{11} \mu\e^2(1+r_1'(\e,\mu\e))- \te_{22}\frac{\mu^3}{8}(1+r_1''(\e,\mu))  & \im \big( \frac12\te_{12}\mu+ r_2(\mu\e^2,\mu^2\e,\mu^3) \big)  \\
- \im\big( \frac12\te_{12}\mu+ r_2(\mu\e^2,\mu^2\e,\mu^3) \big) & -\te_{22}\frac{\mu}{8}(1+r_5(\e,\mu))
 \end{pmatrix}\, , \endgroup \\
 & \label{BinH2prep} G^{(1)} = 
 \begin{pmatrix} 
\mu+  r_8(\mu\e^2, \mu^3 \e ) &   - \im r_9(\mu\e^2,\mu^2\e)  
\\
  \im  r_9(\mu\e^2, \mu^2\e  )  & \tanh (\tth\mu) + r_{10}(\mu\e)
 \end{pmatrix}\, , \\
 & \label{BinH3prep}
 F^{(1)} = 
 \begin{pmatrix} 
\tf_{11}\mu\e+r_3(\mu\e^3,\mu^2\e^2,\mu^3\e ) & 
\im \mu\e \ch^{-\frac12}  + \im  r_4(\mu\e^2, \mu^2 \e) 
 \\
  \im   r_6(\mu\e)
   & r_7(\mu\e 
   ) 
 \end{pmatrix} 
 \end{align}
 where $\te_{11}, \te_{12}, \te_{22}, \tf_{11}$ are defined in 
 \eqref{te11}, \eqref{te12}, \eqref{te22}.
\end{lem}

\begin{rmk}
The matrix $\tL_{\mu,\e}^{(1)}$, a priori defined only for $\mu \neq 0$,  extends analytically to the zero matrix at  
 $\mu = 0$.
For $\mu \neq 0$ the spectrum of $\tL_{\mu,\e}^{(1)}$ coincides with the spectrum of $\tL_{\mu,\e}$.
\end{rmk}

\begin{proof}
The matrix $Y $ is symplectic, i.e. \eqref{sympmatrix} holds,
and since $\mu$ is real, it is reversibility preserving, i.e. 
satisfies  \eqref{revprop}.  
By \eqref{sympchange},  
\begin{equation*}\label{B1formaprep}
   \tB_{\mu,\e}^{(1)} = Y^* \tB_{\mu,\e} Y = \begin{pmatrix}  E^{(1)} & F^{(1)} \\ [F^{(1)}]^* &  
   G^{(1)}  \end{pmatrix}, 
\end{equation*}
with, $Q$ being self-adjoint, $E^{(1)}=QEQ = [E^{(1)}]^* $, $G^{(1)}=QGQ=[G^{(1)}]^*$ and $F^{(1)}=QFQ$. 
 In  view of \eqref{BinG1}-\eqref{BinG3}, we obtain
\eqref{BinH1prep}-\eqref{BinH3prep}.
\end{proof}

\subsection{Non-perturbative step of block-decoupling} 
\label{sec:5.2}

We first verify 
that the quantity $D_\tth:=\tth -\tfrac14 \te_{12}^2$  is nonzero for any  $\tth > 0 $. 
In view of the comment 3 after Theorem \ref{thm:simpler},
we have that $D_\tth = \tth-c_g^2$. 
The non-degeneracy property $ D_\tth \neq 0 $ corresponds to that in 
Bridges-Mielke \cite[p.183]{BrM} and \cite[p.409]{Whitham}. 
\begin{lem} For any $ \tth >0 $ it results 
\begin{equation}\label{defDh}
\mathtt{D}_\tth:=\tth -\tfrac14 \te_{12}^2> 0 \, ,\quad \text{and}\quad \lim_{\tth\to 0^+}\mathtt{D}_\tth=0\, . 
\end{equation}
\end{lem}
\begin{proof}
We write
$ \mathtt{D}_\tth = (\sqrt{\tth}+\frac12 \te_{12})(\sqrt{\tth}-\frac12 \te_{12})$ 
whose first factor is positive for $\tth>0$. 
We claim that also the second factor is positive. 
In view of \eqref{te12} it is equal to 
$  \tfrac12 \ch^{-1} f(\tth) $ with
\begin{align*}
f( \tth )  := 
\big(\sqrt{\tth}\tanh(\tth) - \sqrt{\tth}+\sqrt{\tanh(\tth)}\big)\big(\sqrt{\tth}\tanh(\tth) + \sqrt{\tth}-\sqrt{\tanh(\tth)}\big)=:q(\tth)p(\tth)\, . 
\end{align*}
The function $p(\tth)$ is positive since $ \tth >\tanh(\tth)$ for any $\tth>0$. 
We claim that also the function $q(\tth )$ is positive. 
Indeed its derivative
$$
q' (\tth)
 = \frac{1 - \tanh (\tth)}{2 \sqrt{\tth} \sqrt{ \tanh (\tth)}} 
\Big( - \sqrt{ \tanh (\tth)} + \sqrt{\tth} + \sqrt{\tth} \,   {\tanh (\tth) } \Big)
+ \sqrt{\tth} \big( 1 - \tanh^2 (\tth) \big) > 0   
$$
for any $ \tth > 0 $. Since $ q(0) = 0 $ we deduce that $ q (\tth) > 0 $ for any  $ \tth > 0 $.  
This  proves the lemma.
\end{proof}

We now state the main result of this section. 

\begin{lem} {\bf (Step of block-decoupling)}\label{decoupling2}
There exists a $2\times 2$ reversibility-preserving matrix $ X $,
analytic in $(\mu, \e) $,  of the form
\begin{align} \label{Xsylvy}
X & := \begin{pmatrix} x_{11} & \im x_{12} \\ \im x_{21} & x_{22} \end{pmatrix} 
\qquad \qquad \qquad \qquad	\text{with} \quad  x_{ij}\in\bR \, , \ i,j=1,2 \, , \\
& 
= 
  \begin{pmatrix}
    r_{11}(\e) 
    & \im\,     r_{12}(\e) 
    \\
  -\im \frac12 \mathtt{D}_\tth^{-1}  (\te_{12} \tf_{11} + 2\ch^{-\frac12})\e +  \im r_{21}(\e^2, \mu\e) & 
   \frac12 \mathtt{D}_\tth^{-1} (\ch^{-\frac12}\te_{12} +2\tth \tf_{11})\e+ r_{22}(\e^2,\mu\e)
   \end{pmatrix}\,  \notag ,
\end{align}
where 
$\te_{12}$, $\tf_{11}$ are defined in \eqref{te12}, \eqref{te11} and $ \mathtt{D}_\tth $ is the positive constant  in \eqref{defDh}, 
such that the following holds true. By conjugating the Hamiltonian and reversible matrix 
$\tL_{\mu,\e}^{(1)}$, defined in \eqref{LinHprep}, with the symplectic and reversibility-preserving $4\times 4$ matrix 
\begin{equation}\label{formaS}
\exp\left(S^{(1)} \right) \, , 
\quad \text{ where } 
\qquad S^{(1)} := \tJ_4 \begin{pmatrix} 0 & \Sigma \\ \Sigma^* & 0 \end{pmatrix} \, , \qquad \Sigma:= \tJ_2 X \, , 
\end{equation}
we get the Hamiltonian and reversible matrix  
\begin{equation}\label{sylvydec}
 \tL_{\mu,\e}^{(2)} := \exp\left(S^{(1)} \right)\tL_{\mu,\e}^{(1)} \exp\left(-S^{(1)} \right)= \tJ_4 \tB_{\mu,\e}^{(2)} =
\begin{pmatrix}   \tJ_2 E^{(2)}  & \tJ_2  F^{(2)}  \\  
\tJ_2 [F^{(2)}]^*  &  \tJ_2 G^{(2)}  \end{pmatrix}\, ,
\end{equation}
where the reversibility-preserving $2\times 2$ self-adjoint 
matrix   $[E^{(2)}]^*=E^{(2)}$ has the form 
\begin{align}
\label{Bsylvy1}
& E^{(2)} = 
\begin{pmatrix} 
\mu\e^2 \teWB+ r_1'(\mu \e^3, \mu^2 \e^2 )-\te_{22}\frac{\mu^3}{8}(1+r_1''(\e,\mu))  & \im  \big( \frac12\te_{12}\mu+ r_2(\mu\e^2,\mu^2\e,\mu^3) \big)  \\
- \im  \big( \frac12\te_{12}\mu+ r_2(\mu\e^2,\mu^2\e,\mu^3) \big) & -\te_{22}\frac{\mu}{8}(1+r_5(\e,\mu))
 \end{pmatrix}\, ,
 \end{align}
  where
\begin{align}\label{te2}
\teWB =\te_{11} -  
\mathtt{D}_\tth^{-1} \big( \ch^{-1} + \tth \tf_{11}^2 +\te_{12}\tf_{11}\ch^{-\frac12} \big)   
\end{align}
(with constants  $ \te_{11}$, $\mathtt{D}_\tth $, $\tf_{11} $, $ \te_{12}$, 
defined in \eqref{te11}, \eqref{defDh}, \eqref{te12}), 
is the Whitham-Benjamin function defined in \eqref{funzioneWB}, 
the reversibility-preserving $2\times 2$ self-adjoint 
matrix  $[G^{(2)}]^*=G^{(2)}$ has the form
\begin{equation}
  \label{Bsylvy2} G^{(2)} = 
 \begin{pmatrix} 
\mu+ r_8(\mu\e^2, \mu^3 \e )
&   - \im r_9(\mu\e^2,\mu^2\e) \\
  \im  r_9(\mu\e^2, \mu^2\e)  & \tanh(\tth\mu) + r_{10}(\mu\e) 
 \end{pmatrix}\, , 
 \end{equation}
and 
 \begin{equation}
 \label{Bsylvy3}F^{(2)}= \begin{pmatrix}
  r_3(\mu\e^3 ) 
& \im r_4(\mu\e^3 ) \\
\im r_6(\mu\e^3 ) &
r_7(\mu\e^3)
\end{pmatrix} \, .
 \end{equation}
\end{lem}

The rest of the section is devoted to the proof of Lemma \ref{decoupling2}.
 For simplicity let $ S = S^{(1)} $. 

The matrix  $\text{exp}(S)$ is symplectic and reversibility-preserving
because the matrix $ S $ in \eqref{formaS} is Hamiltonian and 
reversibility-preserving, cfr. Lemma 3.8 in \cite{BMV1}. 
Note that $ S $  is reversibility preserving  since $X$ 
has  the form  \eqref{Xsylvy}. 

We now expand in Lie series 
the Hamiltonian and reversible matrix $ \tL_{\mu,\e}^{(2)} 
= \exp (S)\tL_{\mu,\e}^{(1)} \exp (-S) $. 

We split $\tL_{\mu,\e}^{(1)}$ 
into its $2\times 2$-diagonal and off-diagonal Hamiltonian and reversible matrices
\begin{align}
& \qquad  \qquad \qquad  \qquad  \qquad \qquad \tL_{\mu,\e}^{(1)} = D^{(1)} + R^{(1)}  \, , \label{LDR}\\
& 
D^{(1)} :=\begin{pmatrix} D_1 & 0 \\ 0 & D_0 \end{pmatrix} :=  \begin{pmatrix} \tJ_2 E^{(1)} & 0 \\ 0 & \tJ_2 G^{(1)} \end{pmatrix}, \quad 
R^{(1)} := \begin{pmatrix}  0 & \tJ_2 F^{(1)} \\ \tJ_2 [F^{(1)}]^* & 0 \end{pmatrix} , \notag
\end{align} 
and we perform the Lie expansion
\begin{align}
\label{lieexpansion}
& \tL_{\mu,\e}^{(2)} 
 = \exp(S)\tL_{\mu,\e}^{(1)} \exp(-S)  =  D^{(1)} +\lie{S}{D^{(1)}}+ \frac12 [S, [S, D^{(1)}]] + 
 R^{(1)}  + [S, R^{(1)}]  \\
 & + 
\frac12 \int_0^1 (1-\tau)^2 \exp(\tau S)  \text{ad}_S^3( D^{(1)} )  \exp(-\tau S) \, \de \tau 
+ \int_0^1 (1-\tau) \, \exp(\tau S) \, \text{ad}_S^2( R^{(1)} ) \, \exp(-\tau S) \, \de \tau \notag 
\end{align} 
where $\text{ad}_A(B) := [A,B] := AB - BA $ denotes the commutator 
between the linear operators $ A, B $.

We look for a $ 4 \times 4 $ matrix $S$ as in \eqref{formaS} that solves
the homological equation
$  R^{(1)}  +\lie{S}{ D^{(1)} } = 0  $,
which, recalling \eqref{LDR}, reads 
\begin{equation}\label{homoesp}
\begin{pmatrix} 0 & \tJ_2F^{(1)}+\tJ_2\Sigma D_0
- D_1\tJ_2\Sigma \\ 
\tJ_2{[F^{(1)}]}^*+\tJ_2\Sigma^*D_1-D_0\tJ_2\Sigma^* & 0 \end{pmatrix} =0 \, .
\end{equation}
Note that the equation $  \tJ_2F^{(1)}+\tJ_2\Sigma D_0 - D_1\tJ_2\Sigma = 0 $ implies 
also  $  \tJ_2{[F^{(1)}]}^*+\tJ_2\Sigma^*D_1-D_0\tJ_2\Sigma^*  = 0 $ and viceversa. 
Thus, writing  $ \Sigma =\tJ_2  X  $, namely $ X = - \tJ_2  \Sigma $,  
the equation \eqref{homoesp} amounts to solve the  ``Sylvester" equation 
\begin{equation}\label{Sylvestereq}
D_1 X - X D_0 = - \tJ_2F^{(1)}   \, .
\end{equation}
We write  the matrices $ E^{(1)}, F^{(1)}, G^{(1)}$ in \eqref{LinHprep}
as 
\begin{equation}\label{splitEFGprep}
E^{(1)} = 
\begin{pmatrix} 
E_{11}^{(1)} & \im E_{12}^{(1)} \\ 
- \im E_{12}^{(1)} & E_{22}^{(1)}
\end{pmatrix}\, , \quad
F^{(1)} = 
\begin{pmatrix} 
F_{11}^{(1)} & \im F_{12}^{(1)} \\ 
\im F_{21}^{(1)} & F_{22}^{(1)}
\end{pmatrix} \, , \quad
G^{(1)} = 
\begin{pmatrix} 
G_{11}^{(1)} & \im G_{12}^{(1)} \\ 
- \im G_{12}^{(1)} & G_{22}^{(1)} 
\end{pmatrix} 
\end{equation} 
where the real numbers 
$ E_{ij}^{(1)}, F_{ij}^{(1)}, G_{ij}^{(1)} $, $ i , j = 1,2 $, have the expansion 
in  \eqref{BinH1prep}, \eqref{BinH2prep}, \eqref{BinH3prep}.  
Thus, by  \eqref{LDR}, \eqref{Xsylvy}  and \eqref{splitEFGprep},
the equation \eqref{Sylvestereq}  amounts  to solve the 
 $4\times 4$ real linear system 
\begin{align}\label{Sylvymat}
\underbrace{ \begin{pmatrix}  
G_{12}^{(1)} - E_{12}^{(1)} &   G_{11}^{(1)} &  E_{22}^{(1)} & 0 \\   
G_{22}^{(1)} & G_{12}^{(1)} - E_{12}^{(1)} & 0 & - E_{22}^{(1)} \\
 E_{11}^{(1)} & 0 & G_{12}^{(1)} - E_{12}^{(1)}  & - G_{11}^{(1)} \\
 0 &  - E_{11}^{(1)} & -G_{22}^{(1)}  &  G_{12}^{(1)} - E_{12}^{(1)}
 \end{pmatrix}}_{=: {\cal A} }
 \underbrace{ \begin{pmatrix} x_{11} \\ x_{12} \\ x_{21} \\ x_{22} \end{pmatrix}}_{ =: \vec x}
  =
  \underbrace{
    \begin{pmatrix} 
 -F_{21}^{(1)}  \\  F_{22}^{(1)} \\ - F_{11}^{(1)} \\  F_{12}^{(1)}
 \end{pmatrix}
 }_{=: \vec f}.
\end{align}
We solve this system using the following result, verified by a direct calculus.
\begin{lem}\label{LemmaSylvy}
The determinant of the matrix 
\begin{equation}\label{formA}
A := \begin{pmatrix}  a &  b & c & 0 \\  d & a & 0 & - c \\
e & 0 & a  & -b \\
 0 & - e & -d  &  a
 \end{pmatrix} 
\end{equation}
where $ a,b,c, d , e  $ are real numbers,   is 
\begin{equation}\label{Sylvydet}
 \det A = a^4 -2 a^2 (b d + c e)+(b d - c e)^2 =  (bd-a^2)^2 -2ce\big(a^2 +bd-\frac12 ce\big) \, .
\end{equation}
If $ \det A \neq 0 $ then $ A  $ is invertible and 
\begin{align}\label{Sylvyinv}
  A^{-1}   =  
\footnotesize{\frac{1}{ \det A} \left(
\begin{array}{cccc}
  \! a \left(a^2-b d - c e\right) &  \!  b \left(-a^2+b d - c e\right) & 
-c \left(a^2+b d - 
 c e\right) & \! - 2 a b c \\
  \! d \left(-a^2+b d -  c e\right) &  \! a \left(a^2-b d - c e\right) & 2 a c d & 
   \! - c \left(-a^2-b d + c e\right) \\
  \! - e \left(a^2+b d -  c e\right) &  \! 2 a b e & a \left(a^2-b d - c e\right) & 
 \!  b \left(a^2-b d + c e\right) \\
 \!  - 2 a d e &  \! - e \left(-a^2-b d + c e\right) & d \left(a^2-b d + c e\right) &
  \!  a \left(a^2-b d - c e\right)  
\end{array}
\right)}  \, . 
\end{align}
\end{lem}
The Sylvester matrix $ \cal A $ in \eqref{Sylvymat} has the  form 
\eqref{formA} where, by \eqref{BinH1prep}-\eqref{BinH3prep} 
and since $ \tanh(\tth\mu) = \tth\mu + r(\mu^3)$, 
\begin{align}
& a =  G_{12}^{(1)} - E_{12}^{(1)}  \label{abcde} = 
- \te_{12}\frac{\mu}{2} \big(1 +r(\e^2, \mu \e, \mu^2)\big) \, , \ b =  G_{11}^{(1)} 
=\mu +   r_8(\mu\e^2, \mu^3 \e  
) \, , \\
&  c = 
E_{22}^{(1)} =-\te_{22}\frac{\mu}{8}(1+r_5(\e,\mu)) \, 
, \ d = G_{22}^{(1)}
= \mu \tth  + r( \mu \e, \mu^3 )\, , 
\
 e =  E_{11}^{(1)} = r(\mu\e^2, \mu^3) \, , \notag
\end{align}
where  $\te_{12}$ and 
$\te_{22}$, defined   respectively in \eqref{te12}, 
\eqref{te22}, are positive for any $ \tth > 0 $. 

By \eqref{Sylvydet}, the determinant of the matrix $ {\cal A} $ is 
\begin{align} 
\det {\cal A} = (bd-a^2)^2 + r(\mu^4\e^2,\mu^6) = 
 \mu^4 \mathtt{D}_\tth^2(1+r(\e, \mu^2))  \label{detcalA} \, 
\end{align}
 where $\mathtt{D}_\tth$ is defined in \eqref{defDh}.
By \eqref{Sylvyinv}, \eqref{abcde}, \eqref{detcalA}
and, since $\mathtt{D}_\tth=\tth-\frac14\te_{12}^2$, we obtain 
\begin{align}\label{calA-1} 
 {\cal A}^{-1} = (1+r(\e,\mu)) \displaystyle{\frac{1}{\mu \mathtt{D}^2_\tth}}\, \begin{pmatrix}
   \frac12{\te_{12}}\mathtt{D}_\tth &   \mathtt{D}_\tth & 
 \frac{1}{32} \te_{22}  (\te_{12}^2+4\tth)  &  -\frac18{\te_{12}}\, \te_{22}   \\
 \tth\mathtt{D}_\tth &   \frac12{\te_{12}}\mathtt{D}_\tth  
  &\frac18 \te_{12}\te_{22} \tth  & 
 - \frac{1}{32}\te_{22} \, (\te_{12}^2+4\tth)   \\
 r(\e^2, \mu^2)  &  r(\e^2, \mu^2)   
&   \frac12{\te_{12}} \mathtt{D}_\tth & 
- {\mathtt{D}_\tth} \\
    r(\e^2, \mu^2)  &  r(\e^2, \mu^2)   &  -\tth \mathtt{D}_\tth  &
  \frac12{\te_{12}}\mathtt{D}_\tth 
\end{pmatrix} \, . 
\end{align}
Therefore, 
for any $\mu\neq 0$, there exists a unique solution $\vec x = {\cal A}^{-1} \vec f  $
of the linear system \eqref{Sylvymat}, namely  a unique matrix $ X $ which solves  
the Sylvester equation \eqref{Sylvestereq}.

\begin{lem}
The matrix  solution $X $ of the Sylvester equation \eqref{Sylvestereq} 
is analytic in $(\mu,\e) $,  and admits an expansion as in \eqref{Xsylvy}.
\end{lem}

\begin{proof}
By \eqref{Sylvymat}, \eqref{calA-1}, \eqref{splitEFGprep}, \eqref{BinH3prep} we obtain,  for any $\mu \neq 0$
\begin{align*} 
\footnotesize
\begin{pmatrix}
x_{11} \\ x_{12} \\ x_{21} \\ x_{22} \end{pmatrix}  
 \footnotesize  = \frac{1}{\mathtt{D}^2_\tth} \begingroup 
\setlength\arraycolsep{-2pt} \begin{pmatrix}
   \frac12{\te_{12}}\mathtt{D}_\tth &   \mathtt{D}_\tth & 
 \frac{1}{32} \te_{22}  (\te_{12}^2+4\tth)  &  -\frac18{\te_{12}}\, \te_{22}   \\
 \tth\mathtt{D}_\tth &   \frac12{\te_{12}}\mathtt{D}_\tth  
  &\frac18 \te_{12}\te_{22} \tth  & 
 - \frac{1}{32}\te_{22} \, (\te_{12}^2+4\tth)   \\
 r(\e^2, \mu^2) \quad &  r(\e^2, \mu^2)  
&   \frac12{\te_{12}} \mathtt{D}_\tth & 
- {\mathtt{D}_\tth} \\
    r(\e^2, \mu^2) \quad &  r(\e^2, \mu^2)  &  -\tth \mathtt{D}_\tth  &
  \frac12\te_{12}\mathtt{D}_\tth \end{pmatrix} \endgroup
   \begin{pmatrix}  r(\e) \\ r(\e) 
  \\ -\tf_{11}\e + r(\e^3,\mu\e^2,\mu^2\e) \\ 
  \ch^{-\frac12}\e+r(\e^2, \mu \e) \end{pmatrix}(1+r(\e,\mu)) \, ,
\end{align*} 
which proves \eqref{Xsylvy}.
In particular each $x_{ij}$ admits an analytic extension at $\mu = 0$. 
Note that, for $\mu = 0$, one has  $E^{(2)}=G^{(2)}=F^{(2)}= 0$
and the Sylvester equation reduces to tautology.
\end{proof}
 Since the matrix $ S $ solves the homological equation $\lie{S}{ D^{(1)} }+  R^{(1)} =0$, identity \eqref{lieexpansion} simplifies to 
\begin{equation}
\label{Lie2}
\tL_{\mu,\e}^{(2)}
 =   D^{(1)}  +\frac12\lie{S}{ R^{(1)} }+
\frac12 \int_0^1 (1-\tau^2) \, \exp(\tau S) \, \text{ad}_S^2( R^{(1)} ) \, \exp(-\tau S) \de \tau \, . 
\end{equation}
 The matrix $\frac12 \lie{S}{ R^{(1)} }$ is, by \eqref{formaS}, 
\eqref{LDR},   the block-diagonal Hamiltonian and reversible matrix
\begin{equation}\label{Lieeq2}
\begin{aligned}
& \frac12 \lie{S}{ R^{(1)} }\\ 
& = \begin{pmatrix}  \frac12 \tJ_2 ( \Sigma \tJ_2 [F^{(1)}]^*- F^{(1)} \tJ_2 \Sigma^*) & 0 \\ 0 &  \frac12 \tJ_2 ( \Sigma^* \tJ_2 F^{(1)}- [F^{(1)}]^* \tJ_2 \Sigma) \end{pmatrix} = \begin{pmatrix} \tJ_2 \tilde E & 0 \\ 0 &\tJ_2 \tilde G \end{pmatrix},
\end{aligned}
\end{equation}
where, since $ \Sigma = \tJ_2 X $,  
\begin{equation}\label{EGtilde}
\tilde E := \text{\textbf{Sym}} \big( \tJ_2 X \tJ_2  [F^{(1)}]^* \big) 
\, , \qquad 
\tilde G :=  \text{\textbf{Sym}} \big(  X^* F^{(1)} \big) \, , 
\end{equation}
denoting $ \text{\textbf{Sym}}(A) := \frac12 (A+ A^* )$. 
\begin{lem}\label{lem:EGtilde}
The  self-adjoint and reversibility-preserving matrices 
 $ \tilde E,\ \tilde G  $ in \eqref{EGtilde} have the form
\begin{equation}
\begin{aligned}
& \tilde E = \begin{pmatrix} 
 \tilde\te_{11}\mu\e^2  +  \tilde r_1(\mu\e^3,\mu^2\e^2) &  \im  \tilde r_2(\mu\e^2) \\ 
 - \im  \tilde r_2(\mu\e^2) &   \tilde r_5(\mu\e^2) \end{pmatrix} \, , 
 \quad \tilde G = 
\begin{pmatrix}  \tilde r_8(\mu\e^2) &   \im \tilde r_9 (\mu\e^2)   \\
-\im \tilde r_9(\mu\e^2)  &  \tilde r_{10}(\mu\e^2)
\end{pmatrix} \, ,  \\
 &\tilde\te_{11} :=  -\mathtt{D}_\tth^{-1} \big( \ch^{-1} + \tth \tf_{11}^2 +\te_{12}\tf_{11}\ch^{-\frac12} \big)\, . \end{aligned} \label{tilde.E.G}
\end{equation}
\end{lem}
\begin{proof}
For simplicity we set $F=F^{(1)}$.  By   \eqref{Xsylvy}, \eqref{BinH3prep}, 
one has 
\begin{align*}
\tJ_2 X \tJ_2  F^* &  = 
 \begin{pmatrix}   x_{21}F_{12}-x_{22}F_{11} &  \im  (x_{21}F_{22}+x_{22} F_{21})
 \\ 
 \im (x_{11}F_{12}+x_{12}F_{11}) & - x_{11}F_{22} + x_{12}F_{21} \end{pmatrix} 
= \begin{pmatrix} 
 \tilde\te_{11}\mu\e^2 + r(\mu\e^3,\mu^2\e^2) &  \im  r(\mu\e^2) \\ 
 \im  r(\mu\e^2) &  r(\mu\e^2) \end{pmatrix} 
\end{align*}
with $ \tilde \te_{11}$ defined in \eqref{tilde.E.G}. The expansion of $\tilde E$ in \eqref{tilde.E.G} follows in view of \eqref{EGtilde}. 
Since $X = \cO(\e)$ by \eqref{Xsylvy}  and $  F = O(\mu \e) $ by \eqref{BinH3prep} 
we deduce that $ X^*   F  = \cO (\mu \e^2 )$ and 
the expansion of $\tilde G $ in \eqref{tilde.E.G} follows. 
\end{proof}

Note that the term $  \tilde\te_{11}\mu\e^2 $ in the matrix $ \tilde E $
in \eqref{EGtilde}-\eqref{tilde.E.G}, has the same order
of the $(1,1)$-entry of  $ E^{(1)} $ in \eqref{BinH1prep}, thus will contribute to the Whitham-Benjamin function $ \teWB $ in the $(1,1)$-entry of  $ E^{(2)} $ in \eqref{Bsylvy1}.
Finally we show that the last term in \eqref{Lie2} is 
small. 

\begin{lem}\label{lem:series}
The $ 4 \times 4 $  Hamiltonian and reversibility  matrix 
\begin{equation}\label{series}
\frac12 \int_0^1 (1-\tau^2) \, \exp(\tau S) \, \textup{ad}_S^2( R^{(1)} ) \, \exp(-\tau S) \, \de \tau
= \begin{pmatrix}
\tJ_2 \widehat E & \tJ_2 F^{(2)}\\
\tJ_2 [ F^{(2)}]^* & \tJ_2 \widehat G
\end{pmatrix}
\end{equation}
where the $ 2 \times 2 $ self-adjoint and reversible  matrices  $\widehat E $, $ \widehat G$ have   entries 
\begin{equation}\label{E2G2.0}
\widehat E_{ij} \ , \widehat G_{ij}   =  r(\mu\e^3) \, , \quad 
i,j = 1,2 \, , 
\end{equation}
and   the $2\times 2$ reversible matrix $ F^{(2)}$ admits an expansion as in 
\eqref{Bsylvy3}.
\end{lem}
\begin{proof}
Since $S $ and $  R^{(1)}  $ are Hamiltonian and reversibility-preserving
then $ \textup{ad}_S  R^{(1)}  = [S, R^{(1)} ] $ is Hamiltonian   and reversibility-preserving as well.
Thus 
each $ \exp(\tau S) \, \textup{ad}_S^2( R^{(1)} ) \, \exp(-\tau S)$ 
 is Hamiltonian   and reversibility-preserving, and formula  \eqref{series} holds. 
In order to estimate  its entries we first compute $\textup{ad}_S^2( R^{(1)} )$. 
Using the form of  $ S $ in \eqref{formaS} and  $[S,  R^{(1)} ]$ in \eqref{Lieeq2} one gets
\begin{equation}\label{tildeF}
\textup{ad}_S^2(R^{(1)})  =  \begin{pmatrix} 0 & \tJ_2\tilde F \\ \tJ_2 \tilde F^* & 0\end{pmatrix}\qquad \text{where} \qquad \tilde F:= 
2\left(  \Sigma \tJ_2 \tilde G - \tilde E \tJ_2 \Sigma \right)
\end{equation}
and  $\tilde E$, $\tilde G$ are defined in  \eqref{EGtilde}. 
Since $ \tilde E ,  \tilde G = \cO (\mu \e^2 )$ 
by \eqref{tilde.E.G}, and $\Sigma = \tJ_2 X =
\cO ( \e ) $ by \eqref{Xsylvy}, we deduce that 
 $\tilde F = \cO(\mu\e^3)  $. 
Then, for any $ \tau \in [0,1]$, the matrix 
$\exp(\tau S) \, \textup{ad}_S^2( R^{(1)} ) \, \exp(-\tau S) =  \textup{ad}_S^2( R^{(1)} ) (1 + \cO(\mu,\e))$. 
In particular the matrix $F^{(2)}$ in \eqref{series} has the same expansion of $\tilde F$, namely  
$ F^{(2)} = \cO(\mu\e^3)  $, and the matrices $\widehat E$, $\widehat G$ have entries 
as in \eqref{E2G2.0}.
\end{proof}

\begin{proof}[Proof of Lemma \ref{decoupling2}.]
It follows by \eqref{Lie2}-\eqref{Lieeq2}, \eqref{LDR} and Lemmata  \ref{lem:EGtilde} and 
\ref{lem:series}. 
The matrix  $E^{(2)} := E^{(1)} + \tilde E + \widehat{ E}$ has the expansion  in \eqref{Bsylvy1}, with $ \teWB = \te_{11} + \tilde \te_{11} $ as in \eqref{te2}.
Similarly $G^{(2)} := G^{(1)} + \tilde G + \widehat{G} $ has the expansion  in \eqref{Bsylvy2}.
\end{proof}

\subsection{Complete block-decoupling and proof of the main results}\label{section34}

We now  block-diagonalize the  $ 4\times 4$ Hamiltonian and reversible 
matrix $\tL_{\mu,\e}^{(2)}$    in \eqref{sylvydec}.  First we split it 
into its $2\times 2$-diagonal and off-diagonal Hamiltonian and reversible matrices
\begin{align}
& \qquad  \qquad   \qquad  \qquad \quad \tL_{\mu,\e}^{(2)} = D^{(2)} + R^{(2)} \, , 
\notag \\
& 
D^{(2)}:= 
  \begin{pmatrix} \tJ_2 E^{(2)} & 0 \\ 0 & \tJ_2 G^{(2)}  \end{pmatrix}, \quad 
R^{(2)}:= \begin{pmatrix}  0 & \tJ_2 F^{(2)} \\ \tJ_2 [F^{(2)}]^* & 0 \end{pmatrix} . \label{LDR2}
\end{align}

\begin{lem}\label{ultimate}
There exist a  $4\times 4$ reversibility-preserving Hamiltonian  matrix $S^{(2)}:=S^{(2)}(\mu,\e)$ of the form \eqref{formaS}, analytic in $(\mu, \e)$, of size 
{$\cO(\e^3)$}, and a $4\times 4$ block-diagonal reversible Hamiltonian matrix $P:=P(\mu,\e)$, analytic in $(\mu, \e)$, of size ${ \cO(\mu\e^6)}$  such that 
\begin{equation}\label{ultdec}
\exp(S^{(2)})(D^{(2)}+R^{(2)}) \exp(-S^{(2)}) = D^{(2)}+P \, . 
\end{equation}
\end{lem}
\begin{proof}
We set for brevity $ S = S^{(2)} $. 
The equation \eqref{ultdec} is equivalent
to  the system
\begin{equation}\label{equazionisplittate}
\begin{cases}  \Pi_{D}\big( e^{ S} \big(D^{(2)}+R^{(2)}   \big) e^{- S} \big ) - D^{(2)}  = P   \\
\Pi_{\off}\big( e^{S} \big(D^{(2)}+R^{(2)}   \big) e^{- S}\big)  = 0 \, , 
\end{cases}
\end{equation}
where $\Pi_D$ is the projector onto the block-diagonal matrices and $\Pi_\off$ onto 
the block-off-diagonal ones.
The second equation  in \eqref{equazionisplittate} is equivalent, by a Lie expansion, 
and since $ [S, R^{(2)}] $ is block-diagonal,  to 
\begin{equation}\label{nonlinhomo}
R^{(2)} +  \lie{S}{D^{(2)}} +  \underbrace{\Pi_\off \int_0^1 (1-\tau) e^{\tau S} \text{ad}_S^2\big(D^{(2)}+R^{(2)} \big)e^{- \tau S} \de \tau}_{=: \mathcal{R}(S)} = 0 \, . 
\end{equation}
 The ``nonlinear homological equation" \eqref{nonlinhomo},  
\begin{equation}\label{nonlinhomo1}
[S,D^{(2)}] = -R^{(2)} - \mathcal{R}(S) \, , 
\end{equation}
is equivalent to solve the $4\times 4$ real linear system
\begin{equation}\label{sistAx2}
{\cal A} \vec{x} =  \vec{f}(\mu,\e,\vec{x}) \, ,\quad \vec{f}(\mu,\e,\vec{x}) = \mu \vec{v}(\mu,\e)+\mu \vec{g}(\mu,\e,\vec{x})
\end{equation}
associated, as in \eqref{Sylvymat}, to \eqref{nonlinhomo1}. 
The vector $ \mu \vec{v}(\mu,\e) $ is associated with $ -  R^{(2)} $
where $R^{(2)}  $ is in \eqref{LDR2}.
The vector $ \mu  \vec{g}(\mu,\e,\vec{x})  $ is associated with the matrix
$  - \mathcal{R}(S) $, 
which is a Hamiltonian and reversible block-off-diagonal 
matrix (i.e of the form \eqref{LDR}).
The factor $\mu$ is present in $D^{(2)}$ and $R^{(2)}$, see \eqref{Bsylvy1}, \eqref{Bsylvy2}, \eqref{Bsylvy3} and 
the analytic function $ \vec{g}(\mu,\e,\vec{x})  $ is quadratic in $ \vec{x} $ (for the presence of 
$ \text{ad}_S^2 $ in $ \cR(S)$). 
In view of \eqref{Bsylvy3} one has 
\begin{equation}\label{sizev}
\mu \vec{v}(\mu,\e):= (-F^{(2)}_{21},F^{(2)}_{22},-F^{(2)}_{11},F^{(2)}_{12})^\top, \quad F^{(2)}_{ij} =  \, {r(\mu \e^3)} \, .
\end{equation}
System \eqref{sistAx2}  is equivalent to $ \vec{x}  = {\cal A}^{-1}  \vec{f}(\mu,\e,\vec{x}) $ and,
writing  ${\cal A}^{-1} = \frac1\mu {\cal B} (\mu,\e) $ (cfr. \eqref{calA-1}), to 
$$
\vec{x} = {\cal B}(\mu,\e) \vec{v}(\mu,\e) + 
 {\cal B}(\mu,\e) \vec{g}(\mu, \e, \vec{x}) \, . 
$$
By the implicit function theorem
this equation admits a unique small solution $\vec{x}=\vec{x}(\mu,\e)$, analytic in 
$ (\mu, \e ) $,  
 with size ${\cO (\e^3)} $ as  $ \vec{v} $ in \eqref{sizev}.
Then  the first equation of  \eqref{equazionisplittate}  gives  
 $ P = [S, R^{(2)}] + \Pi_D \int_0^1 (1-\tau) e^{\tau S} \text{ad}_S^2\big(D^{(2)}+R^{(2)} \big)e^{- \tau S} \de \tau$, 
 and its estimate follows from those of $S$ and  $ R^{(2)} $ (see  \eqref{Bsylvy3}).  
\end{proof}
 
 \noindent
{\sc Proof of Theorems \ref{TeoremoneFinale} and \ref{thm:simpler}. }
By Lemma \ref{ultimate} and recalling  \eqref{calL}
the operator $ \cL_{\mu,\e} : \mathcal{V}_{\mu,\e} \to  \mathcal{V}_{\mu,\e} $ 
is represented by the $4\times 4$ Hamiltonian and reversible matrix 
$$
\im\ch \mu + \exp( S^{(2)})\tL_{\mu,\e}^{(2)} \exp(- S^{(2)}) = \im \ch \mu + 
\begin{pmatrix} \tJ_2 E^{(3)} & 0 \\ 0 &  \tJ_2 G^{(3)} 
\end{pmatrix} =:
 \begin{pmatrix} \mathtt{U} & 0 \\ 0 & \mathtt{S} \end{pmatrix} \, , 
$$ 
where the matrices $E^{(3)}$ and $G^{(3)}$ expand 
as in \eqref{Bsylvy1}, \eqref{Bsylvy2}.
Consequently the  matrices $\mathtt{U}$ and $\mathtt{S}$ 
expand  as in \eqref{S}. 
Theorem \ref{TeoremoneFinale} is proved. 
Theorem \ref{thm:simpler} is a straight-forward corollary.
The function $\underline{\mu}(\e) $ in \eqref{barmuep} is defined as the implicit solution of the function $\DeltaBF(\tth;\mu,\e)$ in  \eqref{primosegno} for $\e$ small enough, depending on $\tth$.

\appendix

\section{Expansion of the Kato basis }\label{ProofExpansion}

In this appendix we  prove Lemma \ref{expansion1}. 
We  provide the expansion of the basis $f_k^\pm(\mu,\e) = U_{\mu,\e}f_k^\pm $, $k=0,1$, in \eqref{basisF}, where $f_k^\pm$ defined in \eqref{funperturbed}  belong to the subspace $\mathcal{V}_{0,0}:=\text{Rg}(P_{0,0})$. We first Taylor-expand the transformation operators $U_{\mu,\e}$ defined in \eqref{OperatorU}.  We denote  $\pa_\e$ with an  apex and  $\pa_\mu$ with a dot. 
\begin{lem}
The first jets of $U_{\mu,\e}P_{0,0}$ are 
 \begin{align}
  U_{0,0}P_{0,0}&=P_{0,0} \, , \quad U_{0,0}'P_{0,0}=P_{0,0}'P_{0,0} \, , \quad \dot U_{0,0}P_{0,0}=\dot P_{0,0}P_{0,0} \, , \label{Ufirstorder}\\
\dot U_{0,0}'P_{0,0}&=
\big(\dot P_{0,0}'- \tfrac12 P_{0,0}\dot P_{0,0}' \big)P_{0,0} \, , \label{Umix} 
 \end{align}
where
\begin{align}\label{Pdereps}
 P_{0,0}' &= \frac{1}{2\pi\im} \oint_\Gamma (\sL_{0,0}-\lambda)^{-1} \sL_{0,0}' (\sL_{0,0}-\lambda)^{-1} \de\lambda \, ,  \\ 
\dot P_{0,0}  \label{Pdermu} &= \frac{1}{2\pi\im} \oint_\Gamma (\sL_{0,0}-\lambda)^{-1} \dot \sL_{0,0} (\sL_{0,0}-\lambda)^{-1} \de\lambda \, ,
\end{align}
and
\begin{subequations}
\begin{align}
\dot P_{0,0}' &= -\frac{1}{2\pi\im} \oint_\Gamma (\sL_{0,0}-\lambda)^{-1} \dot \sL_{0,0} (\sL_{0,0}-\lambda)^{-1}  \sL_{0,0}' (\sL_{0,0}-\lambda)^{-1} \de\lambda  \label{Pmisto1}\\
&\qquad   -\frac{1}{2\pi\im} \oint_\Gamma (\sL_{0,0}-\lambda)^{-1} \sL_{0,0}' (\sL_{0,0}-\lambda)^{-1} \dot \sL_{0,0} (\sL_{0,0}-\lambda)^{-1} \de\lambda \label{Pmisto2} \\ 
&\qquad   + \frac{1}{2\pi\im} \oint_\Gamma (\sL_{0,0}-\lambda)^{-1} \dot \sL_{0,0}' (\sL_{0,0}-\lambda)^{-1}  \de\lambda \label{Pmisto3} \, .
\end{align}
\end{subequations}
The operators $\sL_{0,0}'$ and $\dot \sL_{0,0}$ are
\begin{equation}
\sL_{0,0}' =
 \begin{bmatrix} \pa_x \circ p_1(x) & 0 \\ -a_1(x) & p_1(x)\circ \pa_x \end{bmatrix}, 
 \qquad 
 \dot \sL_{0,0} = 
 \begin{bmatrix} 0 &  \textup{sgn}(D) m(D) \\
  0 & 0 
  \end{bmatrix},
   \label{cLfirstorder}
\end{equation}
where $\sgn (D) $ is defined in \eqref{def:segno} and $m(D) $ is the real, even operator
\begin{equation}\label{multiplier}
m(D) := \tanh ( \tth |D|) + \tth |D| (1 - \tanh^2 (\tth |D|)) 
\end{equation}
and 
$ a_1(x) $ and $  p_1(x) $ are given in Lemma \ref{lem:pa.exp}. 

The operator $\dot \sL_{0,0}'$ is
\begin{equation}\label{cLmisto}
\dot \sL_{0,0}' = \begin{bmatrix} \im p_1(x) & 0 \\ 0 & \im p_1(x)\end{bmatrix}\,.
\end{equation}
\end{lem}
\begin{proof}
 By \eqref{OperatorU} and \eqref{rootexp} one has the Taylor expansion  in $\cL(Y)$
$$
   U_{\mu,\e}P_{0,0}  = P_{\mu,\e}P_{0,0} + \frac{1}{2}(P_{\mu,\e}-P_{0,0})^2P_{\mu,\e}P_{0,0} +\cO(P_{\mu,\e}-P_{0,0})^4   \, ,
  $$
  where  $\cO(P_{\mu,\e}-P_{0,0})^4 = \cO(\e^4,\e^3\mu,\e^2\mu^2,\e\mu^3,\mu^4) \in \cL(Y)$.
Consequently one derives \eqref{Ufirstorder}, \eqref{Umix},
using also the identity
$\dot P_{0,0} P_{0,0}' P_{0,0} + P_{0,0}' \dot P_{0,0} P_{0,0} = - P_{0,0} \dot P_{0,0}' P_{0,0}$, 
which follows  differentiating $P_{\mu,\e}^2 = P_{\mu,\e}$.  
Differentiating  \eqref{Pproj}  one gets  \eqref{Pdereps}-\eqref{Pmisto3}. Formulas
\eqref{cLfirstorder}-\eqref{cLmisto} follow by \eqref{calL2}
using also that the Fourier multiplier
$ \Pi_0 \big( \tanh(\tth |D|) + \tth |D| \big(1-\tanh^2(\tth |D|) \big)\big) = 0 $.   
\end{proof}
By the previous lemma we have the Taylor expansion
\begin{equation}\label{ordinibase}
f_k^\sigma(\mu,\e) = f_k^\sigma + \e P_{0,0}' f_k^\sigma +\mu \dot P_{0,0} f_k^\sigma + \mu\e  \big(\dot P_{0,0}'- \frac12 P_{0,0}\dot P_{0,0}' \big) f_k^\sigma + \cO(\mu^2,\e^2) \, .
\end{equation}
In order to compute the vectors
 $P_{0,0}' f_k^\sigma$ and $\dot P_{0,0} f_k^\sigma$ using 
 \eqref{Pdereps} and \eqref{Pdermu}, it is useful to know the action of  $(\sL_{0,0} - \lambda)^{-1}$ on the vectors 
\begin{equation}
\begin{aligned}
\label{fksigma}
& f_k^+:=\vet{\ch^{1/2}\cos(kx)}{\ch^{-1/2}\sin(kx)} \, ,
\quad  f_k^- :=\vet{-\ch^{1/2}\sin(kx)}{\ch^{-1/2}\cos(kx)} \, , \\
&  f_{-k}^+ :=\vet{\ch^{1/2}\cos(kx)}{-\ch^{-1/2}\sin(kx)}\, ,
\quad 
f_{-k}^- :=\vet{\ch^{1/2}\sin(kx)}{\ch^{-1/2}\cos(kx)} \, , \quad k \in \bN \, .
\end{aligned}  
\end{equation}
\begin{lem}\label{lem:VUW}
The space $ H^1(\bT) $ decomposes as 
$
H^1(\bT) =  \cV_{0,0} \oplus \cU \oplus {\cW_{H^1}}
${, with  $\cW_{H^1}=\overline{\bigoplus\limits_{k=2}^\infty \cW_k}^{H^1}$}
where the subspaces $\cV_{0,0}, \cU $ and $ \cW_k $, defined below, are 
invariant  under   $\sL_{0,0} $ and  the following properties hold:
\begin{itemize}
\item[(i)] $ \cV_{0,0} = \text{span} \{ f^+_1, f^-_1, f^+_0, f^-_0\}$  is the generalized kernel of $\sL_{0,0}$. For any $ \lambda \neq 0 $ the operator 
$ \sL_{0,0}-\lambda :  \cV_{0,0} \to \cV_{0,0} $ is invertible and  
 \begin{align}\label{primainversione1}
& (\sL_{0,0}-\lambda)^{-1}f_1^+ = -\frac1\lambda f_1^+ \, ,
\quad 
(\sL_{0,0}-\lambda)^{-1}f_1^- = -\frac1\lambda f_1^-,
\quad  (\sL_{0,0}-\lambda)^{-1}f_0^- = -\frac1\lambda f_0^- \, ,  \\
& \label{primainversione2}
(\sL_{0,0}-\lambda)^{-1}f_0^+ = -\frac1\lambda f_0^+ + \frac{1}{\lambda^2} f_0^- \, .
\end{align} 
\item[(ii)] $\cU := \text{span}\left\{ f_{-1}^+, f_{-1}^-  \right\}$.   For any 
$ \lambda \neq \pm 2 \im $ the operator 
$ \sL_{0,0}-\lambda :  \cU \to \cU $ is invertible and
\begin{equation}
\label{primainversione3}
\begin{aligned}
&  (\sL_{0,0}-\lambda)^{-1} f_{-1}^+ =
   \frac{1}{\lambda^2+4 \ch^2}\left(-\lambda f_{-1}^+ + 2 \ch f_{-1}^-\right),  \\
&  (\sL_{0,0}-\lambda)^{-1} f_{-1}^- 
  = \frac{1}{\lambda^2+4\ch^2}
  \left(-2 \ch f_{-1}^+ - \lambda f_{-1}^-\right) \, .
  \end{aligned}
\end{equation}
\item[{(iii)}]  Each
subspace $\cW_k:= \text{span}\left\{f_k^+, \ f_k^-, f_{-k}^+, \ f_{-k}^- \right\}$ is  invariant under $ \sL_{0,0} $.  Let $\cW_{L^2}=\overline{\bigoplus\limits_{k=2}^\infty \cW_k}^{L^2}$. For any
$|\lambda| < 
\delta(\tth)$ small enough, the operator 
$ {\sL_{0,0}-\lambda :  \cW_{H^1}\to \cW_{L^2} }$ is invertible and for any $f \in {\cW_{L^2}}$
\begin{equation}
\label{primainversione4}
 (\sL_{0,0}-\lambda)^{-1} f  =  \big(\ch^2 \pa_x^2 + |D|\tanh(\tth |D|)\big)^{-1} \begin{bmatrix} \ch \partial_x & - |D|\tanh(\tth |D|) \\ 1 & \ch \partial_x\end{bmatrix} f + \lambda \varphi_f(\lambda, x) \, ,
\end{equation}
for some analytic  function  $\lambda \mapsto \varphi_f(\lambda, \cdot) \in H^1(\bT, \bC^2)$.
\end{itemize}
\end{lem}
\begin{proof}
By inspection the spaces $\cV_{0,0}$, $\cU$ and ${ \cW_k}$ are invariant under $ \sL_{0,0}$ 
and, by Fourier series, they decompose $H^1(\bT, \bC^2)$. 
 Formulas  \eqref{primainversione1}-\eqref{primainversione2} follow using that 
$f_1^+, f_1^-, f_0^-$ are in the kernel of $\sL_{0,0}$, and $\sL_{0,0}f_0^+ =-f_0^- $.
 Formula \eqref{primainversione3} follows using that  $\sL_{0,0} f^+_{-1} = -2\ch f^{-}_{-1}$ and  $\sL_{0,0} f^-_{-1} =  2\ch f^{+}_{-1}$.
Let us prove item $(iii)$. Let $\cW:= \cW_{H^1}$.  The operator 
${ \restr{(\sL_{0,0}-\lambda\uno)}{\cW}} $ is invertible for any $ \lambda \notin 
\{ 
\pm \im 
\sqrt{|k| \tanh{(\tth |k|)}} \pm \im k \ch, k \geq 2, k \in {\mathbb N}  \}$ and 
$$ \footnotesize   
(\restr{\sL_{0,0}}{\cW})^{-1} 
= 
\left(\ch^2 \pa_x^2 + |D| \tanh(\tth |D|)\right)^{-1} 
\begin{bmatrix}
\ch \pa_x & -|D| \tanh(\tth |D|)) \\ 1 & \ch\pa_x\end{bmatrix}_{|\cW}  \, . 
$$
By Neumann series, for any  $ \lambda $ such that 
 $ |\lambda |  \|(\restr{\sL_{0,0}}{\cW})^{-1}\|_{{\cL(\cW,H^1(\bT))}}
 < 1 $ 
  we have
$$
  (\restr{\sL_{0,0}}{\cW}-\lambda)^{-1} = 
(\restr{\sL_{0,0}}{\cW})^{-1}  
\big( \uno - \lambda (\restr{\sL_{0,0}}{\cW})^{-1} \big)^{-1} =  
   (\restr{\sL_{0,0}}{\cW})^{-1} \sum_{k \geq 0} ((\restr{\sL_{0,0}}{\cW})^{-1}\lambda)^k 
  \, .
$$
 Formula 
  \eqref{primainversione4} follows  with 
 $\varphi_f(\lambda, x):= (\restr{\sL_{0,0}}{\cW})^{-1} 
 \sum_{k \geq 1} \lambda^{k-1} [(\restr{\sL_{0,0}}{\cW})^{-1}]^k f $.
\end{proof}
We shall also use the following formulas obtained  by 
\eqref{cLfirstorder}, \eqref{multiplier} and \eqref{funperturbed}:
\begin{equation}\label{derivoeps}
\begin{aligned}
&\sL_{0,0}'f_1^+ = \vet{2 \ch^{-1/2} \, \sin(2x)}{\frac{1}{2}
 \ch^{5/2}(1-\ch^{-4}) (1+\cos(2x) ) } \, , \qquad  
\sL_{0,0}'f_1^- = \vet{2\,  \ch^{-1/2} \, \cos(2x)}{-\frac{1}{2} 
\ch^{5/2}(1-\ch^{-4}) \sin(2x)} \, ,  \\
& 
\sL_{0,0}'f_0^+ = \vet{2\ch^{-1} \sin(x)}{\left(\ch^2 + \ch^{-2}\right)\cos(x)} \, , \qquad
\sL_{0,0}'f_0^- = 0 \, , \\
& \dot\sL_{0,0}f_1^+ = -\im b(\tth) \vet{\cos(x)}{0}\, , \qquad \dot\sL_{0,0}f_1^-= \im  b(\tth) \vet{\sin(x)}{0} \, , 
\quad b(\tth):= \ch^{-1/2} \big( \ch^2 + \tth(1-\ch^4) \big) \, , 
\\
&
\dot\sL_{0,0}f_0^+ = 0 \, , \qquad
 \dot\sL_{0,0}f_0^- = 0  \, .
\end{aligned}
\end{equation}
\begin{rem}
In deep water  we have $ \dot\sL_{0,0}f_0^- = f_0^+$ (cfr. formula (A.14)  in \cite{BMV1}). 
In finite depth instead $ \dot\sL_{0,0}f_0^- = 0$ because
the Fourier multiplier $ \sgn(D) m(D) $  in \eqref{multiplier} vanishes on the constants.
\end{rem}
We  finally compute $P_{0,0}' f_k^\sigma$ and $\dot P_{0,0}f_k^\sigma$.
\begin{lem}
One has
\begin{equation}\label{tuttederivate}
\begin{aligned}
 & P_{0,0}'f^+_1 ={ 
  \vet{\frac12 \ch^{-\frac{11}{2}} (3+\ch^4) \, \cos(2x)}{  \frac14 \ch^{-\frac{13}{2}} (1+\ch^4)(3-\ch^4) \sin(2x)}  } \, , \quad P_{0,0}'f^-_1 = {
  \vet{-\frac12 \ch^{-\frac{11}{2}} (3+\ch^4) \, \sin(2x)}{  \frac14 \ch^{-\frac{13}{2}} (1+\ch^4)(3-\ch^4) \cos(2x)}  } \, , \\
&   P_{0,0}'f^+_0 =   \tfrac14\ch^{-\frac52}(3+\ch^4)  f^+_{-1} \, , \quad   P_{0,0}'f^-_0 =0 \, , \quad    \dot P_{0,0} f_0^+=0 \, ,\quad  
  \dot P_{0,0} f_0^-=0 \, ,   \\
 &\dot P_{0,0} f_1^+ = \frac{\im}{4}\big(1+\ch^{-2}\tth (1-\ch^4)\big) f^{-}_{-1} \, ,
 \quad
  \dot P_{0,0} f_1^- = \frac{\im}{4}  \big(1+\ch^{-2} \tth (1-\ch^4)\big) f^+_{-1}\, .
\end{aligned}
\end{equation}
\end{lem}
\begin{proof}
We first compute $P_{0,0}'f_1^+$. By \eqref{Pdereps}, \eqref{primainversione1}  and \eqref{derivoeps}  we deduce
$$
P_{0,0}'f_1^+ = -\frac{1}{2\pi\im} \oint_\Gamma \frac{1}{\lambda}(\sL_{0,0}-\lambda)^{-1}
\vet{2 \ch^{-1/2} \, \sin(2x)}{\frac{1}{2} \ch^{5/2}(1-\ch^{-4}) (1+\cos(2x) ) }
  \de\lambda \,  .
$$
We note that   
$ \footnotesize \vet{2 \ch^{-1/2} \, \sin(2x)}{\frac{1}{2}\ch^{5/2}(1-\ch^{-4}) (1+\cos(2x) ) }  = 
\frac{1}{2} \ch^{5/2}(1-\ch^{-4}) f_0^- + \cW $. 
 Therefore by \eqref{primainversione1}
and  \eqref{primainversione4}  there is an analytic function $\lambda \mapsto \varphi(\lambda, \cdot) \in H^1(\bT, \bC^2)$ so that  
\begin{equation*}
P_{0,0}'f_1^+
 = -\frac{1}{2\pi\im} \oint_\Gamma \frac{1}{\lambda} \Big(- \dfrac{\ch^{5/2}(1-\ch^{-4})}{2\lambda} f_0^- 
{-
\frac{1+\ch^4}{4\ch^6} 
  \vet{2\ch \frac{\ch^{-\frac12} (3+\ch^4)}{1+\ch^4} \cos(2x)}{ \ch^{-\frac12} (3-\ch^4) \sin(2x)}  }
  +
 \lambda \varphi(\lambda) \Big) \, \de\lambda \, ,
\end{equation*}
where we exploited the identity $ \tanh(2\tth) = \frac{2\ch^2}{1+\ch^4}$ in applying \eqref{primainversione4}. Thus, by means of  residue Theorem we obtain the first identity in \eqref{tuttederivate}.
Similarly one computes $P_{0,0}'f_1^-$. 
 By \eqref{Pdereps}, \eqref{primainversione1}  and \eqref{derivoeps}, one has $P_{0,0}'f_0^-=0$. 
Next we compute $P_{0,0}'f_0^+$. 
 By \eqref{Pdereps}, \eqref{primainversione1},  \eqref{primainversione2}  and \eqref{derivoeps} we get
$$
P_{0,0}'f_0^+  
= -\frac{1}{2\pi\im} \oint_\Gamma \frac{1}{\lambda}(\sL_{0,0}-\lambda)^{-1}  \vet{2\ch^{-1} \sin(x)}{(\ch^2 + \ch^{-2})\cos(x)}  \de\lambda \, . 
$$
Next we decompose
$ \footnotesize
 \vet{2\ch^{-1} \sin(x)}{(\ch^2 + \ch^{-2})\cos(x)} 
 =\underbrace{{ \frac12 \ch^{-\frac32}(\ch^4+3)}}_{=: \alpha} f^{-}_{-1} + \underbrace{{\frac12 \ch^{-\frac32}(\ch^4-1)}}_{ =: \beta} f^-_1 $. 
By \eqref{derivoeps} and \eqref{primainversione3} we get 
$$
P_{0,0}'f_0^+   = 
-\frac{1}{2\pi\im} \oint_\Gamma \Big(-\frac{2 \alpha \ch}{\lambda(\lambda^2+4\ch^2)}f_{-1}^+- \frac{\alpha }{\lambda^2+4\ch^2} f_{-1}^-  + \frac{\beta}{\lambda^2} f^-_1 \Big)  \de\lambda  = \frac{\alpha}{2\ch} f^+_{-1}
  \, ,  
$$
where in the last step we used the residue theorem.
We compute now $\dot P_{0,0} f^+_1$. 
First we have
$ \dot P_{0,0}f_1^+   =\ \frac{\im}{2\pi\im} b( {\mathtt h}) \oint_\Gamma \frac{1}{\lambda}(\sL_{0,0}-\lambda)^{-1} \footnotesize \vet{\cos(x)}{0} \de\lambda  $, where $b(\tth)$ is in \eqref{derivoeps},
and then, writing $ \footnotesize  \vet{\cos(x)}{0} =\frac{1}{2} \ch^{-\frac12} ( f_1^+ + f_{-1}^+ )$ and using \eqref{primainversione3}, we conclude  
using again the residue theorem
$ \dot P_{0,0} f_1^+
 = \frac{\im}{4} \big( 1 + {\mathtt h} (1 - \ch^4) \ch^{-2} \big) f^{-}_{-1} $. 
 The computation of $\dot P_{0,0}f^-_1$ is analogous. 
  Finally, in view of \eqref{derivoeps}, we have
  \begin{align*}
&   \dot P_{0,0}f^+_0 = \frac{1}{2\pi\im} \oint_\Gamma (\cL_{0,0}-\lambda)^{-1} \dot\cL_{0,0} \big( \frac1{\lambda^2} f_0^--\frac{1}{\lambda}f_0^+ \big) \de \lambda =0\, , \\ 
&  \dot P_{0,0}f^-_0 = -\frac{1}{2\pi\im} \oint_\Gamma \frac{1}{\lambda} (\cL_{0,0}-\lambda)^{-1} \dot\cL_{0,0}f_0^- \de \lambda =0 \,.
  \end{align*}
In conclusion all the formulas  in \eqref{tuttederivate} are proved.
\end{proof}
So far we have obtained the linear terms of the expansions  \eqref{exf41}, \eqref{exf42}, \eqref{exf43}, \eqref{exf44}. 
We now provide further information about the expansion of the basis at $\mu=0$. The proof of the next lemma follows as that of Lemma A.4 in \cite{BMV1}. 
 \begin{lem}
 The basis $\{f_k^\sigma(0,\e), \, k = 0,1 , \sigma = \pm\}$ is real.
 For any $\e $  it results $f_0^-(0,\e) \equiv f_0^- $. The property \eqref{nonzeroaverage} holds.
 \end{lem}

We now provide further information about the expansion of the basis at $\e=0$. The following lemma follows as Lemma A.5 in \cite{BMV1}. The key observation is that the operator $\restr{\sL_{\mu,0}}{\mathcal{Z}}$, where $\mathcal{Z}$ is the invariant subspace $\mathcal{Z}:=\text{span}\{f_0^+,\,f_0^-\}$, has the two eigenvalues $\pm\im \sqrt{\mu \tanh(\tth \mu)}$, which, for small $\mu$, lie inside the loop $\Gamma$ around $0$ in \eqref{Pproj}.

\begin{lem}\label{lem:A5}
For any small $\mu$, we have $f_0^+(\mu,0) \equiv f_0^+ $
and $f_0^-(\mu,0) \equiv f_0^- $. Moreover the vectors $f_1^+(\mu,0)$ and $f_1^-(\mu,0)$ have both components with zero space average.
\end{lem}

We finally consider  the $\mu\e$ term in the expansion \eqref{ordinibase}.
\begin{lem}
The derivatives
$ (\pa_{\mu} \pa_\e f_k^\sigma)(0,0) = 
 \left(\dot P_{0,0}'- \frac12 P_{0,0}\dot P_{0,0}' \right)f_k^\sigma $
 satisfy 
 \begin{equation}\label{struttura2}
\begin{aligned}
&  (\pa_{\mu} \pa_\e f_1^+)(0,0)  = \im \vet{odd(x)}{even(x)},  \qquad 
(\pa_{\mu} \pa_\e f_1^-)(0,0) - = \im \vet{even(x)}{odd(x)},\\
& (\pa_{\mu} \pa_\e f_0^+)(0,0) = \im \vet{odd(x)}{even_0(x)},
\qquad 
(\pa_{\mu} \pa_\e f_0^-)(0,0) = \im   \vet{even_0(x)}{odd(x)}  \, . 
\end{aligned}
\end{equation}
\end{lem}

\begin{proof}
We prove that $\dot P'_{0,0}= \eqref{Pmisto1} + \eqref{Pmisto2} + 
\eqref{Pmisto3}$ is purely imaginary, see footnote \ref{apuim}.
This follows since the operators in $\eqref{Pmisto1}$,  $\eqref{Pmisto2}$ and  $\eqref{Pmisto3}$ 
 are purely imaginary 
 because $\dot \sL_{0,0}$ is purely imaginary, 
  $\sL_{0,0}' $  in \eqref{cLfirstorder}  is real
and $\dot \sL_{0,0}'$ in \eqref{cLmisto} is purely imaginary 
(argue as in  Lemma \ref{propPU}-$(iii)$ of \cite{BMV1}).
Then, applied to the real vectors $f^\sigma_k$, $k = 0,1$, $\sigma = \pm$, give purely imaginary vectors. 

 The property 
\eqref{reversiblebasisprop} implies that  $(\pa_{\mu} \pa_\e f_k^\sigma)(0,0) $ have the claimed parity structure
in \eqref{struttura2}.
We shall now prove that 
 $(\pa_{\mu} \pa_\e f_0^\pm)(0,0)$
 have zero average.
We have, by \eqref{primainversione2}  and \eqref{derivoeps}
\begin{align*}
\eqref{Pmisto1} f_0^+:= \frac{1}{2\pi\im} \oint_\Gamma (\sL_{0,0}-\lambda)^{-1} \dot \sL_{0,0} (\sL_{0,0}-\lambda)^{-1} \frac{1}{\lambda} \vet{2\ch^{-1} \sin(x)}{\left(\ch^2 + 
\ch^{-2}\right)\cos(x)}    \, \de\lambda
\end{align*}
and since the operators $(\sL_{0,0}-\lambda)^{-1}$  and $\dot \sL_{0,0}$ are  both Fourier multipliers, hence they preserve {the absence of average} of the vectors, then  $\eqref{Pmisto1} f_0^+$ has zero average.
Next     $ \eqref{Pmisto2} f_0^+ = 0$ since 
$\dot \sL_{0,0} f_0^\pm = 0$, cfr. \eqref{def:segno}. 
Finally, by  \eqref{primainversione2} and \eqref{cLmisto} where $ p_1 (x) =p_1^{[1]} \cos (x) $, 
$$
\eqref{Pmisto3} f_0^+ = 
 \frac{\im p_1^{[1]}}{2\pi\im} \oint_\Gamma (\sL_{0,0}-\lambda)^{-1}
  \Big( 
- \frac{1}{\lambda} \vet{ \cos (x)}{0} + \frac{1}{\lambda^2} \vet{0}{ \cos (x)} \Big) \, 
  \de\lambda  
$$
is a vector with zero average. 
We conclude that $\dot P_{0,0}' f_0^+$ is an imaginary vector with zero average, as well as  $(\pa_\mu \pa_\e f_0^+)(0,0)$ since $P_{0,0}$ sends zero average functions in zero average functions. 
Finally,  by \eqref{reversiblebasisprop},  $(\pa_\mu \pa_\e f_0^+)(0,0)$ has the claimed structure in \eqref{struttura2}.

We  finally consider $(\pa_{\mu} \pa_\e f_0^-)(0,0)$.
By \eqref{primainversione1} and $\sL_{0,0}'f_0^-=0$ (cfr. \eqref{derivoeps}), it results
$$
\eqref{Pmisto1}f_0^- = - \frac{1}{2\pi\im} \oint_\Gamma \frac{(\sL_{0,0}-\lambda)^{-1}}{\lambda} \dot \sL_{0,0} (\sL_{0,0}-\lambda)^{-1} \sL_{0,0}' f_0^- \de\lambda=0 \,  
 . 
$$
Next by \eqref{primainversione1} and 
$\dot \sL_{0,0} f_0^- = 0$ we get
$\eqref{Pmisto2} f_0^- = 0$.
 Finally by \eqref{primainversione1} and \eqref{cLmisto}
 \begin{align*}
\eqref{Pmisto3} f_0^- = 
-\frac{1}{2\pi\im} \oint_\Gamma (\sL_{0,0}-\lambda)^{-1}
\frac{1}{\lambda} \vet{0}{\im p_1^{[1]} \cos(x)} \de \lambda 
\end{align*}
has zero average 
since $(\sL_{0,0}-\lambda)^{-1}$ is a Fourier multiplier (and thus preserves {average absence}).
\end{proof}
   
   This completes the proof of Lemma \ref{expansion1}. 

\section{Expansion of the Stokes waves in finite depth}\label{sec:App2}

In this Appendix we provide the expansions  \eqref{exp:Sto}-\eqref{expcoef}, \eqref{expfe}, 
\eqref{pino1fd}-\eqref{aino2fd}.
\\[1mm]
\noindent
{\bf Proof  of \eqref{exp:Sto}-\eqref{expcoef}.}
Writing
 \begin{equation}\label{etapsic}
\begin{aligned}
 & \eta_\e(x) = \e \eta_1(x) + \e^2 \eta_2(x) + \cO(\e^3) \, , \\
 &  \psi_\e(x) = \e \psi_1(x) + \e^2 \psi_2(x) + \cO(\e^3) \, ,  
  \end{aligned}
\qquad \quad c_\e = \ch + \e c_1 + \e^2 c_2+ \cO(\e^3) \, ,  
\end{equation}
where  $\eta_i$ is $even(x)$ and $\psi_i$ is $odd(x)$ for $i=1,2$, 
we solve order by order in $ \e $ the equations \eqref{travelingWWstokes},  
that we rewrite  as
\begin{equation}
\label{Sts}
\begin{cases}
-c \, \psi_x +  \eta   + \dfrac{\psi_x^2}{2} - 
\dfrac{\eta_x^2}{2(1+\eta_x^2)} ( c  -  \psi_x )^2  = 0 \\
c \, \eta_x+G(\eta)\psi = 0 	\, ,
\end{cases}
\end{equation}
having substituted $G(\eta)\psi $ with $-c \, \eta_x $  in the first equation.
 We expand the Dirichlet-Neumann operator 
$ G(\eta)=  G_0+ G_1(\eta) + G_2(\eta) + \cO(\eta^3)  $
where, according to \cite{CS}[formula (2.14)], 
\begin{equation} \label{expDiriNeu}
\begin{aligned}
 G_0 & := D\tanh(\tth D) = |D| \tanh(\tth |D|) \,, \\
 G_1(\eta) & := D \big( \eta -  \tanh(\tth D)\eta \tanh(\tth D) \big)D = -\pa_x \eta \pa_x - |D| \tanh(\tth|D|)\eta|D| \tanh(\tth|D|), \\
 G_2(\eta) & := -\frac12 D \Big( 
 D{\eta}^2 \tanh(\tth D) +\tanh(\tth D){\eta}^2 D - 2\tanh(\tth D)\eta D\tanh(\tth  D)\eta\tanh(\tth D) \Big)D \, .
\end{aligned}
\end{equation}
{\bf First order in $ \e $.}  Substituting in \eqref{Sts} the expansions   in \eqref{etapsic},
we get the linear system 
\begin{equation}\label{cB0}
 \left\{\begin{matrix} -\ch (\psi_1)_x + \eta_1 = 0 \\
 \ch (\eta_1)_x + G_0\psi_1 =0 \, ,  \end{matrix}\right.
 \quad  \text{i.e.} \, \vet{\eta_1}{\psi_1} \in \text{Ker }\cB_0  \text{ with } \cB_0 := \begin{bmatrix} 1 & -\ch\pa_x \\ \ch\pa_x & G_0  \end{bmatrix},  
\end{equation}
where $\eta_1$ is $even(x)$ and $\psi_1$ is $odd(x)$.
\begin{lem}
The kernel of the linear operator $\cB_0$ in \eqref{cB0} is
\begin{equation}\label{chk}
\text{Ker }\cB_0= \text{span}\,\Big\{\vet{\cos(x)}{\ch^{-1}\sin(x)} \Big\}.
\end{equation}
\end{lem}

\begin{proof} The action of $\cB_0$ on each subspace span$\footnotesize{\,\Big\{\vet{\cos(kx)}{0}, \vet{0}{\sin(kx)}\Big\}} $, $k\in \bN$, is represented by the $2\times 2$ matrix $\footnotesize{ \begin{bmatrix} 1 & -\ch k \\ -\ch k & k\tanh(\tth k)  \end{bmatrix}}$. Its determinant 
 $ k \tanh(\tth k) - \ch^2 k^2= k^2 \Big(\frac{\tanh(\tth k)}{k} -\tanh(\tth) \Big)$ vanishes  if and only if $k=1$. Indeed 
the function $x\mapsto\frac{\tanh(\tth x)}{x} $ is monotonically decreasing for $x>0$, since its derivative
$
\frac{2x\tth -\sinh(2\tth x)}{2\cosh^2(\tth x)x^2}
$
is negative for $x>0$. For $k=1$ we obtain the kernel of $\cB_0$ given in \eqref{chk}. For $k=0$ it has no kernel since $\psi_1(x)$ is odd.
\end{proof}
We set
$
\eta_1(x) := \cos(x)$, $\psi_1(x) := \ch^{-1} \sin(x) 
$
in agreement with 
\eqref{exp:Sto}. 
\\[1mm]
{\bf Second order in $ \e $.} 
By \eqref{Sts}, and since 
$ \ch^2 (\eta_1)_x^2 = (G_0\psi_1)^2  $, we get  the linear system
\begin{equation}\label{syslin2}
 \cB_0 \vet{\eta_2}{\psi_2} = \vet{c_1(\psi_1)_x-\frac12 (\psi_1)_x^2 + \frac12 (G_0\psi_1)^2 }{-c_1(\eta_1)_x - G_1(\eta_1)\psi_1} \, , 
\end{equation}
where  $\cB_0$ is the self-adjoint operator in \eqref{cB0}. 
System \eqref{syslin2} admits a solution if and only if its right-hand term is orthogonal to the Kernel  of $\cB_0$ in \eqref{chk}, namely
\begin{equation}\label{orth1}
 \Big(\vet{c_1(\psi_1)_x-\frac12 (\psi_1)_x^2 + \frac12 (G_0\psi_1)^2 }{-c_1(\eta_1)_x - G_1(\eta_1)\psi_1}\;,\;\vet{\cos(x)}{\ch^{-1}\sin(x)}\Big)=0 \, . 
\end{equation}
In view of the first order expansion \eqref{exp:Sto}, \eqref{expDiriNeu} and the identity
 $
  \tanh(2\tth)  = 
  \displaystyle{\frac{2\ch^2}{1+\ch^4}}
 $,
  it results $
 [G_0\psi_1](x)= \ch \sin(x)$, $\big[G_1(\eta_1)\psi_1\big](x) 
 =\frac{1-\ch^4}{\ch(1+\ch^4)}\sin(2x)$
so that \eqref{orth1} implies
 $c_1=0$, in agrement with \eqref{exp:Sto}. Equation
 \eqref{syslin2} reduces to 
\begin{align}\label{sisto2}
\begin{bmatrix} 1 & -\ch\pa_x \\ \ch\pa_x & G_0  \end{bmatrix}
\vet{\eta_2}{\psi_2} 
  = \vet{ - \frac14 (\ch^{-2}-\ch^2) - \frac14  (\ch^{-2}+\ch^2) \cos(2x) }{ - \frac{1-\ch^4}{\ch(1+\ch^4)} \sin(2x) }.
\end{align}
Setting $ \eta_2 = \eta_2^{[0]} +  \eta_2^{[2]} \cos(2x) $ and  $ \psi_2 =
 \psi_2^{[2]}  \sin (2x) $, system 
\eqref{sisto2} amounts to 
\begin{align*}
\left\{ \begin{matrix} \eta_2^{[0]} +\big( \eta_2^{[2]} -2\ch  \psi_2^{[2]} \big) \cos(2x)  = - \frac14\left(\ch^{-2}-\ch^2\right) - \frac14  \left(\ch^{-2}+\ch^2 \right) \cos(2x)  \\ (-2\ch \eta_2^{[2]}   + 2  \psi_2^{[2]} \tanh(2\tth))\sin(2x)  =  - \frac{1-\ch^4}{\ch(1+\ch^4)} \sin(2x) \, , \end{matrix}\right.
\end{align*}
which leads to the expansions of $ \eta_2^{[0]} $, $ \eta_2^{[2]} $, $ \psi_2^{[2]} $ 
given in
\eqref{exp:Sto}-\eqref{expcoef}.

\noindent
{\bf Third order in $ \e $.} 
It remains to determine 
$ c_2 $ in \eqref{expc2}.
 We get the linear system 
\begin{equation}\label{syslin3}
 \cB_0 \vet{\eta_3}{\psi_3} = \vet{c_2(\psi_1)_x 
 - (\psi_1)_x (\psi_2)_x - (\eta_1)_x^2 (\psi_1)_x \ch + 
 (\eta_1)_x (\eta_2)_x \ch^2
 }{-c_2(\eta_1)_x - G_1(\eta_1)\psi_2- G_1(\eta_2)\psi_1 - G_2(\eta_1)\psi_1} \, . 
\end{equation}
System \eqref{syslin3} has 
a solution if and only if the right hand side is orthogonal to the Kernel of
$ \cB_0 $ given in \eqref{chk}. This condition determines uniquely $ c_2 $.
Denoting  $\Pi_1$ the $L^2$-orthogonal projector on span$\, \{\cos(x),\sin(x)\} $, it results 
\begin{align*} 
& c_2 (\psi_1)_x = c_2 \ch^{-1} \cos(x)\, , \quad c_2 (\eta_1)_x = -c_2 \sin(x) \, , \quad \Pi_1[ (\psi_1)_x (\psi_2)_x] = 
\psi_2^{[2]} \ch^{-1}  \cos(x)\, ,\\ 
& \Pi_1 [\ch (\eta_1)_x^2 (\psi_1)_x ] = \tfrac14 \cos(x)\, , \quad \Pi_1[\ch^2 (\eta_1)_x (\eta_2)_x] = \eta_2^{[2]}\ch^2 \cos(x) \, , 
\end{align*}
and, in view of \eqref{expDiriNeu}, and \eqref{exp:Sto}, \eqref{expcoef}, 
\begin{align*}
 \Pi_1[ G_1(\eta_1)\psi_2]  &= \psi_2^{[2]}\frac{1-\ch^4}{1+\ch^4} \sin(x) \, ,  \quad 
  \Pi_1[G_2(\eta_1)\psi_1] 
 =   \ch \frac{3\ch^4-1}{4(1+\ch^4)} \sin(x) \, , \\
 \Pi_1[G_1(\eta_2)\psi_1] 
&= \ch^{-1} 
 \Big( \eta_2^{[0]}(1-\ch^4) +   \tfrac12 \eta_2^{[2]} (1+\ch^4) \Big)\sin(x) \, . 
\end{align*}
Therefore the orthogonality condition proves  \eqref{expc2}.

\noindent{\bf Proof of \eqref{expfe}.} 
We expand  the function $\mathfrak{p}(x)  = \e\mathfrak{p}_1(x) + \e^2 \mathfrak{p}_2(x) + \cO(\e^3)  $ defined by the fixed point equation \eqref{def:ttf}.   We first 
note that  the constant $\mathtt{f}_\e=\cO(\e^2)$  because
 $\eta_1(x) = \cos(x)$ has zero average. 
Then 
$ \mathfrak{p}(x)
 = \frac{\mathcal H}{\tanh(\tth |D|)} \big[\e\eta_1 +\e^2\big(\eta_2 + 
 (\eta_1)_x \mathfrak{p}_1 \big)+\cO(\e^3)\big] $, 
and, using that $\mathcal H \cos(kx) = \sin(kx)$, for any  $ k \in \bN$, we get 
\begin{align}
 \mathfrak{p}_1(x) 
 & = \frac{\mathcal H}{\tanh(\tth |D|)}\cos(x)  = \ch^{-2}\sin(x) \, , \label{pfra1} \\
 \mathfrak{p}_2(x) &= \frac{\mathcal H}{\tanh(\tth |D|)}((\eta_1)_x
 \mathfrak{p}_1 +\eta_2 ) 
 =  \frac{(1+\ch^4)(\ch^4+3)}{8\ch^8}\sin(2x) \, . \label{pfra2}
\end{align}
Finally
\begin{align*}
 \tf_\e & = \frac{\e^2}{2\pi} \int_\bT \big( \eta_2 + (\eta_1)_x\mathfrak{p}_1 \big) \de x + \cO(\e^3) 
 = \e^2\big(\eta_2^{[0]} -\tfrac1{2} \ch^{-2} \big)+\cO(\e^3 ) \stackrel{\eqref{expcoef}}{=} \e^2 \frac{\ch^4-3}{4\ch^2}+\cO(\e^3 ) \, .
\end{align*}
The expansion \eqref{expfe} is proved.
\\[1mm]
{\bf Proof of Lemma \ref{lem:pa.exp}.} 
In view of \eqref{exp:Sto}-\eqref{expcoef}, the expansions of the functions $B$, $V$  in \eqref{espVB} are
\begin{align}
 B= : \e B_1(x) + \e^2 B_2(x) + \cO(\e^3) 
= \e\ch \sin(x) + \e^2 \frac{3-2\ch^4}{2\ch^5}\sin(2x) + \cO(\e^3) \label{espB1}  
  \end{align}
  and
  \begin{align}
 V= : \e V_1(x) + \e^2 V_2(x)  + \cO(\e^3)
 = \e \ch^{-1} \cos(x) + \e^2 \Big[\frac\ch2  +\frac{3-\ch^8}{4\ch^7}\cos(2x)\Big]+\cO(\e^3)
 \label{espV1} \, . 
\end{align}
In  view of  \eqref{def:pa}, denoting derivatives w.r.t $x$ with  an apex and suppressing dependence on $x$ when trivial, we have
\begin{align}
\ch+p_\e(x)  &= (\ch+\e^2 c_2 - V(x) - V'(x)\mathfrak{p}(x)+\cO(\e^3)) (1-\mathfrak{p}'(x)+(\mathfrak{p}'(x))^2+\cO(\e^3))\notag \\
&= \ch + \e \underbrace{(- V_1 -\ch \mathfrak{p}_1')}_{=: p_1}+\e^2 \underbrace{\big( c_2 + V_1\mathfrak{p}_1' - V_2 - V_1'\mathfrak{p}_1 -\ch \mathfrak{p}_2' +\ch (\mathfrak{p}_1')^2 \big)}_{=:p_2} + \, \cO(\e^3) \, .\label{pino12imp}
\end{align} 
Similarly by \eqref{def:pa}
\begin{align}
 1+a_\e(&x) : = \frac{1}{1+\mathfrak{p}_x(x)} - (\ch+p_\e(x))B_x(x+\mathfrak{p}(x)) \notag \\
   = &  1+\e\underbrace{\big(-\mathfrak{p}_1'- \ch B_1'\big)}_{=:a_1} +\e^2\underbrace{\big((\mathfrak{p}_1')^2-\mathfrak{p}_2'- \ch B_2'-\ch B_1''\mathfrak{p}_1(x)+ B_1' V_1 + \ch  B_1'\mathfrak{p}_1' \big)}_{=:a_2}+\cO(\e^3)\, . \label{aino12imp}
 \end{align}
By \eqref{espV1}, \eqref{pfra1}, \eqref{exp:Sto}, \eqref{pfra2}, \eqref{espB1} we 
deduce that the functions $p_1 $, $p_2 $, $a_1 $, $a_2 $ in \eqref{pino12imp} and \eqref{aino12imp} have an expansion as in \eqref{pino1fd}-\eqref{aino2fd}.\qed

 \begin{footnotesize}

 \end{footnotesize}
\end{document}